%% file: zigzag_6a.tex
\documentclass[reqno,11pt]{amsart}

\usepackage{amsmath,amssymb,amsthm,amsxtra} %
\usepackage{graphicx} 
\usepackage{geometry}

\usepackage{esint}

\usepackage{lipsum} \usepackage[normalem]{ulem}

\usepackage{autonum}

\usepackage[utf8]{inputenc} 

\usepackage{enumitem}

\DeclareMathOperator{\Lip}{Lip}

\newtheorem{theorem}{Theorem}[section] \newtheorem{lemma}[theorem]{Lemma}
\newtheorem{definition}[theorem]{Definition}

\newtheorem{proposition}[theorem]{Proposition}

\newtheorem{remark}[theorem]{Remark}

\newcommand{\CUS}{\Subset}
\renewcommand{\S}{{\mathbb S}} \newcommand{\CN}{{\mathcal N}} %
\newcommand{\ed}[1]{#1} 
\newcommand{\hks}[1]{#1} 

\newcommand{\hkcc}[1]{}
\newcommand{\hkrmrm}[1]{} %
\newcommand{\jf}[1]{{\RoseVYDP{#1}}}

\newcommand{\T}{{\mathbb T}} %
\newcommand{\OL}{\overline} \newcommand{\wtos}{\stackrel{*}{\rightharpoonup}}

\newcounter{margcount} 
\input{all-def}

\renewcommand{\QQ}{{Q_\ell}}

\renewcommand{\FU}[1]{\text{for $#1$}}%

\begin{document} 

\title{$\Gam$--limit for two--dimensional charged magnetic zigzag domain
  walls}

\author[H.~Knüpfer]{Hans Knüpfer} \email{hans.knuepfer@math.uni-heidelberg.de}
\address[H.K.]{University of Heidelberg, MATCH and IWR, INF 205, 69120 Heidelberg, Germany}
\author[W.~Shi]{Wenhui Shi} \email{wenhui.shi@monash.edu}
\address[W.S.]{9 Rainforest Walk Level 4, Monash University, VIC 3800, Australia}

\begin{abstract}
  Charged domain walls are a type of domain walls in thin ferromagnetic films
  which appear due to global topological constraints. The non--dimensionalized
  micromagnetic energy for a uniaxial thin ferromagnetic film with in--plane
  magnetization $m \in \mathbb{S}^1$ is given by
  \begin{align} \label{E-abs} %
    E_\eps[m] \ = \ %
    \eps\NT{\nabla m}^2 + \frac 1{\eps} \NT{m \cdot e_2}^2 %
    + \frac{\pi\lam}{2|\ln\eps|} \NNN{\nabla \cdot (m-M)}{\dot H^{-\frac 12}}^2,
  \end{align}
  where $M$ is an arbitrary fixed background field to ensure global neutrality
  of magnetic charges. We consider a material in the form a thin strip and
  enforce a charged domain wall by suitable boundary conditions on $m$. In the
  limit $\eps \to 0$ and for fixed $\lam > 0$, corresponding to the macroscopic
  limit, we show that the energy $\Gam$--converges to a limit energy where jump
  discontinuities of the magnetization are penalized anisotropically. In
  particular, in the subcritical regime $\lam \leq 1$ one--dimensional charged
  domain walls are favorable, in the supercritical regime $\lam > 1$ the limit
  model allows for zigzaging two--dimensional domain walls.
\end{abstract}



\maketitle



\tableofcontents

\section{Introduction and statement of main results}

Magnetic domain walls are transition layers in ferromagnetic
samples where the magnetization vector rapidly rotates and transitions between
two regions with almost constant magnetization. A type of transition layer which
is observed in thin ferromagnetic films with uniaxial in--plane anisotropy are
the so-called \textit{zigzag walls}
(e.g. \cite{2006-JMagMagMat-Engel,2016-JScNmag-Kaplan}). These walls carry a
global charge, usually necessitated by global topological constraints
\cite{HubertSchaefer-Book,DKMO-review}. The competition between the
magnetostatic energy and other more local effects leads to the formation of
two--dimensional zigzag structures as in Fig. \ref{fig-zigzag}. In this
work, we derive a macroscopic limit for a model for zigzag walls in the
framework of $\Gam$--convergence. In the limit, the jump discontinuity is
penalized by an effective anisotropic line energy.

\medskip

Although it is known that two--dimensional transition layers may appear for
systems with vectorial phase function, we are only aware of few analytical
results on such systems
\cite{ContiFonsecaLeoni-2002,FonsecaPopovici-2005,RiviereSerfaty-2001,AlougesRiviereSerfaty-2002,RiviereSerfaty-2003}.
The structure and energy of a charged domain wall in a one-dimensional setting
has been considered by Hubert \cite{Hubert-1979} on the basis of a specific
ansatz function. The structure of the zigzag wall has been experimentally and
numerically investigated e.g. in
\cite{2006-JMagMagMat-Engel,1984-ApplPhyLett-Hamzaoui,2007-PhysRevB-Cerruti,2015-JScNMag-Ukleev,2016-JScNmag-Kaplan}. In
particular, the angle of the zigzag structure and its dependence on temperature
and thickness of the magnetic films have been studied in
\cite{2016-JScNmag-Kaplan,2006-JMagMagMat-Engel}. The dynamics of the zigzag
walls have been investigated numerically in
\cite{1984-ApplPhyLett-Hamzaoui,2007-PhysRevB-Cerruti,2015-JScNMag-Ukleev,2016-JScNmag-Kaplan}. It
has been observed in \cite{1977-JPhysD-Sanders} that the zigzag wall consists of
a combination of Bloch wall core and a logarithmic N\'eel wall tail.

\medskip

\textbf{Setting.} In order to state our results more precisely, we present the
set-up for our model: We consider a two-dimensional model for thin ferromagnetic
films with uniaxial in-plane anisotropy for a magnetic sample in the form of an
infinite strip $Q_\ell:=\R \times \T_\ell$, where $\T_\ell=\R\slash \ell \Z$ is the
one-dimensional torus of length $\ell$ which is assumed to be large. The
periodicity assumption in $x_2$--direction is purely technical, and the choice
of $\ell$ does not affect our results.  We enforce a charged transition layer by
assuming that the magnetization $m \in H_{\rm loc}^1(Q_\ell; \S^1)$ satisfies
the boundary conditions
\begin{align} \label{bc} %
  m \ = \ \pm e_1 \qquad\qquad \text{for  $\pm x_1 \ > \ 1$}.
\end{align}
The boundary condition \eqref{bc} imposes a wall with transition angle of
$\pi$ (modulo $2\pi$).  Also by \eqref{bc}, the total charge
is $\int \nabla \cdot m \ dx = 2\ell \neq 0$, where we recall that
$\nabla \cdot m$ is the magnetic charge density associated with $m$.  The energy
for this problem is given by
\begin{align} \label{EE} %
  \begin{aligned}
      E_\eps[m] \ &= \ %
      \frac 12\int_{Q_\ell} \eps|\nabla m|^2   \ dx + \frac 12 \int_{Q_\ell} \frac 1{\eps} |m \cdot e_2|^2 \ dx \\
      &\qquad\qquad+ \frac{\pi\lam}{2|\ln\eps|} \int_{Q_\ell} \big| |\nabla|^{-\frac 12} \nabla \cdot (m-M) \big|^2 \ dx
         \end{aligned}
\end{align}
for some fixed background magnetization $M \in C^1(Q_\ell;\S^1)$  with $\spt(DM) \Subset Q_\ell$, which is chosen such that the
system is charge-free, i.e.
\begin{align} \label{charge-zero} %
  \int_{\QQ} \nabla\cdot(m-M) \ dx \ = \ 0
\end{align}
(the fractional Sobolev norm in \eqref{EE} is defined in
  \eqref{frac-sob}). We note that the background magnetization is needed to
allow for states with finite energy and that our results do not depend on the
specific choice of $M$ (a possible choice is the transition layer in Lemma
\ref{lem-nonlupper} with $\eps= \bet=1$). The components of the energy in order
are called \textit{exchange energy}, \textit{anisotropy energy} and
\textit{stray field energy}.  The small parameter $\eps > 0$ describes the
relative size of the transition layer with respect to the width of the strip.
The material parameter $\lam \geq 0$, describes the relative strength of
  the stray field and anisotropy energy (for a derivation see Section
  \ref{sus-derivation}).

\medskip

The class of admissible functions $\AA$ for the energy \eqref{EE} is given by
  \begin{align}
    \AA \ = \ \big \{  m \in H^1_{\textrm{loc}}(Q_\ell; \S^1) \ : \ \text{$m$ satisfies \eqref{bc}} \big \}.
  \end{align}
  We extend $E_\eps$ to a functional on the affine space
    $M + L^{1}(Q_\ell;\R^2)$ by setting $E_\eps[m] := + \infty$ for
  $m \nin \AA$. We note that the space does not depend on the specific
    choice of $M$ above.

    \medskip

     The transition layer we consider is called a charged domain wall, since
      by the boundary condition the magnetization $m$ has a net charge
      $\int \nabla \cdot m \ dx \neq 0$ as explained above. In contrast,
      transition layers where the total net charge vanishes are called
      charge-free (cf.  \cite{HubertSchaefer-Book}). Transition layers in thin
    films with in--plane rotation, as considered in this work, of the
    magnetization are also called N\'eel walls.

\medskip

\textbf{Main result and discussion.} The main result in this paper is the
derivation of an effective model for the energy \eqref{EE} in the macroscopic
limit $\eps \to 0$ for any fixed $\lam \geq 0$.  In this limit, both the local
and the nonlocal part of the energy concentrate on the one-dimensional jump set
of the magnetization. Moreover, the stray field energy yields an anisotropic
contribution to the penalization of the jump discontinuity:
\begin{theorem}[$\Gam$--convergence] \label{thm-gamma} \text{} %
  Let $\lam \geq 0$. For any sequence $m_\eps \in \AA$, $\eps\to 0$, with
  \begin{align} \label{meps-bound} %
    \limsup_{\eps > 0} E_\eps[m_\eps] \ \leq \ K \ < \infty
  \end{align}
  there is a subsequence $\eps_k \to 0$ with $m_{\eps_k} \to m_0$ in $L^1$
  for some $m_0 \in \AA_0$ as $k \to \infty$, where $\AA_0$
  $=$ $\{$ $m \in BV_{\textrm{loc}}(Q_\ell; \{\pm e_1\})$ : $m$ satisfies \eqref{bc} $\}$.
  Furthermore, the energies $E_\eps$ $\Gam$--converge to $E_0$ in
  the $L^1$-topology, where
  \begin{align} \label{E0} 
    E_0[m] = 
    \displaystyle 2 \int_{\SS_m} \left(1+\lambda |e_1\cdot n|^2
    \right)\chi_{\{|e_1\cdot n|\leq \frac1{\sqrt{\lam}}\}} + 2
    \sqrt{\lambda} |e_1\cdot n| \chi_{\{|e_1\cdot n|>
    \frac1{\sqrt{\lam}}\}}\ d\HH^1 \quad
 \end{align}
 if $m \in \AA_0$ and $E_0[m] = +\infty$ otherwise for
 $m \in M + L^1(Q_\ell;\R^2)$. Here, $\SS_m$ is the jump set of $m$ with the
 measure theoretic unit normal $n$. In particular,
  \begin{enumerate}
  \item For any sequence $m_\eps \in \AA$ with $m_\eps \to m \in \AA_0$ in $L^1$ we have
    \begin{align} \label{low-bd} %
      \liminf_{\eps > 0} E_\eps[m_\eps] \ \geq \ E_0[m].
    \end{align}
  \item For any $m \in \AA_0$, there is a sequence $m_\eps \in \AA$ with $m_\eps \to m$ in $L^1$ and
    \begin{align}
      \limsup_{\eps > 0} E_\eps[m_\eps] \ \leq \ E_0[m].
    \end{align}
  \end{enumerate}
\end{theorem}
It is well--known that the exchange
energy and  anisotropy energy together asymptotically lead to an isotropic
penalization of the length of the jump set $\SS_m$
\cite{AnzellottiBaldiVisintin-1991}. However, in our model the presence of the
magnetostatic energy yields an additional penalization for the jump
discontinuity contributing to the limit energy, which depends on the line
charge density given by the jump of the normal derivative of $m$ over the jump
set $\SS_m$ ($= 2 |n \cdot e_1|$ in our setting). We note that both the local
terms and the nonlocal stray-field energy contribute to the limit energy in
leading order.

\medskip
 
The minimal energy for given $\lam \geq 0$ for the limit problem is
\begin{align}  \label{e-ground} %
  e(\lam) \ %
  := \ \min_{m \in \AA_0} E_0[m]  \ %
  = \ 2 \ell
  \begin{cases}
    1+\lambda \qquad &\text{ for }\lambda\leq 1,\\
    2\sqrt{\lambda} &\text{ for } \lambda>1
  \end{cases}
\end{align}
(see Proposition \ref{prp-limit}). The crossover at the critical value
$\lam = 1$ in \eqref{e-ground} signifies that zigzag configurations are
energetically preferable on the $\eps > 0$ level for $\lam > 1$ but not for
$\lam \leq 1$, when the isotropic part of the limit energy
dominates. Correspondingly, minimizers for the limit energy are degenerate for
$\lam > 1$: In this case any jump set $\SS_m$ which can be written as a graph in
$x_2$ and with measure theoretic normal $n$ satisfying
$|n \cdot e_1| \geq \lam^{-\frac 12}$ is a global minimizer of the energy, see
Fig. \ref{fig-zigzag}b). In particular, for $\lam > 1$ the set of minimizers of
the limit problem includes zigzag configurations. For $\eps > 0$, these
minimizers can be approximated by zigzag--shaped transition layers with normal
satisfying $|n \cdot e_1| = \lam^{-\frac 12}$ and with rapid oscillation in
tangential direction (Lemma \ref{lem-tOme}).

\medskip

\begin{figure}
  \centering %
  a)
  \includegraphics[height=4cm]{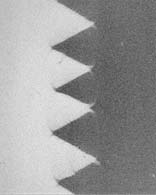} \qquad %
  b)
  \includegraphics[height=3.7cm]{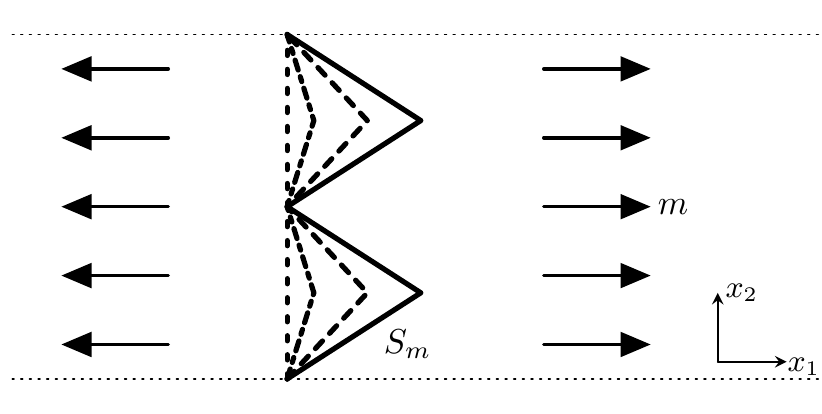} \qquad %
  \caption{a) Experimental picture of zigzag wall separating antiparallel
      domains \cite[Fig. 5.61]{HubertSchaefer-Book} b) Different global
    minimizers for the limit energy in the supercritical case. For all
    minimizers, the normal of the jump set satisfies
    $|n\cdot e_1|\geq \lambda^{-\frac 12}$.}
  \label{fig-zigzag}
\end{figure}

\medskip

Transition layers between two phases are usually
  one--dimensional --- such as e.g. for
  Ginzburg--Landau type energies and the Aviles-Giga energy
  \cite{DKMO-2001}. While it is known that transition layers for models with
  vectorial phase field function might be two--dimensional, only few
    analytical results exist for this case. In particular, we are not aware of
  another analytical result for a thin--film micromagnetic energy as in
  \eqref{EE}. Two--dimensional structures for related micromagnetic energies are
  investigated for the cross--tie wall by Alouges, Rivi\`ere and Serfaty in
  \cite{AlougesRiviereSerfaty-2002,RiviereSerfaty-2003} and for a zigzag
  transition layer by Moser in \cite{Moser09} and by Ignat and Moser in
  \cite{IgnatMoser-2012}. In these works, a setting is considered where the
  magnetization is constant in one coordinate direction and where the nonlocal
  energy is given by the square of the $H^{-1}$--norm (relevant for bulk
  materials). In particular, Ignat and Moser \cite{IgnatMoser-2012} consider
  non--charged transition layers for a prescribed transition angle and derive a
  $\Gam$--limit for the energy in the macroscopic regime $\eps \to 0$ and based
  on the weak$^*$ $L^\infty$--topology. In this situation, the strong
  penalization of the stray field energy enforces divergence free configurations
  in the limit and leads to zigzag configurations for small transition angles.
  Different from our situation, the nonlocal energy does not contribute to the
  limit energy.  The proof in \cite{IgnatMoser-2012} is based on the entropy
  method. Since we consider the critical $\frac{1}{|\ln\eps|}$ scaling of the
  nonlocal terms which allows for charged walls in the limit and where the
  nonlocal energy contributes to the limit energy (different from
  \cite{IgnatMoser-2012}), the entropy method does not seem to apply to our
    model.

\medskip

For thin films, charge--free transition layers with a $\pi$ transition angle (of
N\'eel wall type) have been investigated e.g. in
\cite{Melcher-2003a,Melcher-2004,ignat08,IgnatKnuepfer-2010}. In particular, in
\cite{DKO-2006}, DeSimone, Otto and the first author show that for the
charge--free N\'eel wall, one--dimensional transition layers are asymptotically
energetically optimal. This explains why zigzag type transition layers do not
appear for charge--free N\'eel wall domain walls. On the other hand, for the
charged N\'eel wall considered in this work, the concentration of line charges
leads to formation of zigzag patterns as described above. We note that both
charged and charge-free N\'eel wall exhibit a characteristic logarithmically
decaying tail \cite{Melcher-2003a,Melcher-2004,1977-JPhysD-Sanders}. However,
the leading order contribution to the energy is carried in the tail for the
charge--free N\'eel wall, while it is concentrated in the core for the charged
N\'eel wall.

\medskip

Our argument for the liminf inequality in the $\Gamma$-convergence is based on a duality
argument and the construction of a suitable test function as in
\cite{DKO-2006}.  For the construction of the test function there are
fundamental differences compared to the previous work, where the test
  function is a characteristic function constructed by a Poincar\'e--Bendickson
  argument: In particular, the test function in this work is supported in the
neighborhood of a so-called separating curve with logarithmic decaying
profile. In the construction of the test function, we need to develop and use some
new level set estimates. The detailed strategy of our proof is described in
Section \ref{sus-strategy}.
\begin{remark}[One--dimensional setting]
  One--dimensional transition layers for \eqref{bc}--\eqref{EE} are also simply
  called charged domain walls. The energy (with $\ell = 1$) then takes the form
  \begin{align} \label{E-1d} %
    \frac 12\int_\R \eps |\frac {dm}{d x_1}|^2 + \frac 1\eps |m \cdot e_2|^2 \
    d x_1 + \frac {\pi\lambda} {2|\ln \eps|} \int_\R \big| |\frac d{d x_1}|^{\frac
    12}(m - M) \cdot e_1\big|^2 \ d x_1.
  \end{align}
  The limit energy then simply counts the number of jumps of the
  one--dimensional transition layer $m \in \AA_0$, and each jump is penalized by
  the factor $2(1+\lambda)$. The one-dimensional energy
  \eqref{E-1d} has been analyzed in \cite{Hubert-1979} in terms of specific
  ansatz functions.
\end{remark}
Finally, we note that the existence of minimizers for the energy for the
three--dimensional micromagnetic model has been shown by Anzellotti, Baldi and
Visintin in \cite{AnzellottiBaldiVisintin-1991}, and the arguments can be easily
adapted to our setting.  We also note that variants of the Modica-Mortola model
in the presence of nonlocal interactions have been considered e.g. in
\cite{AB98}. The competition between interfacial and nonlocal energies also
plays a role for the Ohta-Kawasaki model. We mention a few, but by far not
exhaustive list of related works, in which this energy is studied in a periodic
\cite{ChoPel-10,ChoPel-11,GolMurSer-15} or bounded domain
\cite{AceFusMor-13,KnuMur-2013,KnuMur-2014,GolMurSer-15,BonCri-14,JulPis-17}.

\medskip

\textbf{Notation.} Throughout the paper, we denote by $C$ a positive universal
constant unless specified; $\eps\in(0,\frac{1}{4})$ is a small parameter and
$\ell \gg 1$ is a large parameter. \ed{For any set $E \SUS \R^2$, we write
$d_E(x) := \dist(x,E)$ for the distance to this set, noting that the distance to
the empty set is infinite.} For a set $E \subset \R^2$ we write
$\mathcal{N}_t(E):= \{ x \in \R^2 : d(x,E) < t \}$ for its
$t$-neighborhood. The $k$--dimensional Hausdorff measure of the set
  $E \SUS \R^n$ is denoted by $\HH^k(E)$.

\medskip

\textit{BV functions, sets of finite perimeter:} Given an open subset
$U\subset \R^n$, $BV_{loc}(U)$ denotes the space of functions which have
locally bounded variation in $U$ (see \cite{Maggi-Book} for further details). A
measurable set $F\subset \R^n$ has locally finite perimeter in $U$ if the
characteristic function $\chi_F\in BV_{loc}(U)$. We let $\|Du\|(U)$ denote the
total variation measure in $U$ and $\|D\chi_{F}\|(U)$ denote the relative
perimeter of $F$ in $U$. The \textit{reduced boundary} $\p^* F$ of $F$
is the set of points $x \in \spt D\chi_F$ where the measure theoretic outer
normal $n(x)$ exists. Any function $u \in L^1_{loc}(U)$ has an approximate limit
for a.e. $x\in U$, i.e.  $\lim_{\rho\to 0+}\fint_{B_\rho(x)}|u(y)-z| dy =0$ for
some $z\in \R$. The jump set $\t\SS_u$ is the set of points at which the
approximate limit  does not exist. For $u\in BV_{loc}(U;\{\pm 1\})$, we write
$\{x\in U: u(x)=1\}$ for the set of points where the approximate limit of $u$ is
$1$. In this case, the jump set $\t\SS_u$ is $\HH^{n-1}$-a.e. equal to
$\SS_u:=\p^*\{x\in U:u(x)=1\}$, the reduced boundary of this
set. Furthermore, $\NN{Du}(K)=2\HH^{n-1}(\SS_u \cap K)$ for any $K\Subset
U$.

 \medskip
 
 \textit{Some notions for functions and sets on $Q_\ell=\R\times \T_\ell$:}
 We note that there is a canonical projection $\Pi : \R^2 \to
   Q_\ell$. Correspondingly, we identify functions on $Q_\ell$ with
   $Q_\ell$--periodic functions on $\R^2$. Similarly, any set $\Ome \SUS Q_\ell$
   can be identified with its periodic extension onto $\R^2$. For
 $\phi\in L^1(Q_\ell)$ we write
 \begin{align}
   \widehat \phi(\xi)\ := \ \frac{1}{\sqrt{2\pi \ell}}\int_{Q_\ell} e^{i\xi\cdot
   x}\,\phi(x)\,dx \qquad\qquad %
   \text{where $\xi\in\R\times\frac{2\pi}{\ell}\Z$,}
 \end{align}
 i.e. we use the Fourier transformation in $x_1$ and the Fourier series in
 $x_2$. We will use the short notation
 $\int_{\R\times\frac{2\pi}{\ell}\Z}\cdot d\xi %
 : = \ \sum_{\xi_2\in\frac{2\pi}{\ell}\Z}\int_\R \cdot
 d\xi_1$. Plancherel's identity then takes the form
   \begin{align} \label{plancherel-id} %
     \int_{Q_\ell} \OL{\phi(x)} \psi(x) \ dx \ %
     = \ \int_{\R \times \frac{2\pi}{\ell}\Z} \OL{\widehat \phi(\xi)} \widehat \psi(\xi) \ d\xi.
   \end{align}
   The fractional Sobolev norms on $Q_\ell$ for $\alp\in \R$ are defined by
\begin{align} \label{frac-sob} %
  \int_{Q_\ell} \big||\nabla|^\alp\phi \big|^2 \ dx \ %
  := \  \int_{\R\times\frac{2\pi}{\ell}\Z} \big| |\xi|^{\alp}|\widehat
   \phi(\xi)| \big|^2 \ d\xi.
\end{align}
In the appendix we give two more representations of the homogeneous $H^{\frac 12}$--norm.

\subsection{Overview and strategy for the proofs} \label{sus-strategy} %

In this section, we give an overview of the proofs for our results. In
particular, we describe the strategy for the proof of the liminf inequality
in Theorem \ref{thm-gamma}, which represents the main novelty in this paper.
Solution for the limit problem is given in Proposition \ref{prp-limit} in
Section \ref{sec-limit}. 

\medskip

\textbf{Compactness.} The compactness follows by a well--known argument
  (see \cite{AnzellottiBaldiVisintin-1991}). For the sequence
$m_\eps = (u_\eps,v_\eps) \in \AA$ from Theorem \ref{thm-gamma}, we have
\begin{align}\label{eq-compactness}
  \int_{Q_\ell} |\nabla u_\eps| \ dx \ %
  &\leq \ \frac 12 \int_{Q_\ell}  \Big(\frac{\eps|\nabla u_\eps|^2}{1-u_\eps^2} + \frac {1-u_\eps^2}{\eps} \Big)  \ dx \ %
    = \ \frac 12 \int_{Q_\ell} \Big(\eps |\nabla m_\eps|^2 + \frac {v_\eps^2}{\eps} \Big) \ dx \ \leq \ K.
\end{align}
Together with the boundary conditions \eqref{bc} it follows that $v_\eps \to 0$
in $L^1(Q_\ell)$. After selection of a subsequence, we also have
$u_\eps \to u \in BV_{\rm loc}(Q_\ell;\{ \pm 1\})$ in $L^1(Q_\ell)$. Since the
boundary conditions are still satisfied in the limit, we get
$m_\eps \to m = (u,0)$ in $L^1$ for $m \in \AA_0$.

\medskip

\textbf{\ed{Liminf inequality.}} We describe the strategy of the proof, the
details are given in Section \ref{sec-lower}: We consider a sequence
$m_\eps \in \AA$ with $m_{\eps} \to m = (u,0) \in \AA_0$ in $L^1$ for
$\eps \to 0$ such that \eqref{meps-bound} holds. The jump set of $m$ (or
equivalently of $u$) is $\HH^1$ a.e. equal to $\SS_m:= \p^* \{x\in Q_\ell: u(x)=1\}$.  The
unit outer normal of $\{x\in Q_\ell: u(x)=1\}$ along $\SS_m$ is denoted by $n$.
 
\medskip

\textit{Step 1: Localization argument.} The first step of the proof is a
localization argument (see Section \ref{sus-localize} for details). The idea
is to choose a family of pairwise disjoint balls $B_k \SUS Q_\ell$ with
sufficiently small radius which almost covers $\SS_m$, and suitable cut-off
functions $\chi_{\eps,k}$ with $\spt \chi_{\eps,k}\Subset B_k$ and
$\chi_{\eps,k}\to \chi_{B_k}$ as $\eps \to 0$. We write $E_\eps$  in the form
 \begin{align}
   E_\eps[m_\eps] \ = \ \sum_{k} \Big(\nu_\eps(B_k) + N_\eps[\chi_{\eps,k}] \Big) + R_\eps.
 \end{align}
 The two terms $\nu_\eps(B_k)$ and $N_\eps[\chi_{\eps,k}]$ represent the
   interfacial energy in the ball $B_k$ and the self--interaction energy within
 the ball respectively, i.e.
 \begin{align}
   \nu_\eps(B_{k}) \ %
   &:= \ \frac 12\int_{B_{k}} \Big(\eps|\nabla m_\eps|^2 + \frac{v_\eps^2}{\eps}  \Big) \  dx, \\
   N_\eps[ \chi_{\eps,k}] \ %
   &: = \ \frac{\lam}{4|\ln\eps|} \iint_{\R^2\times \R^2}\frac{(\chi_{\eps,k}\sig_\eps)(x)(\chi_{\eps,k}\sig_\eps)(y)}{|x-y|}\ dxdy, \label{def-Neps}
 \end{align}
 where $\sig_\eps:= \nabla \cdot (m_\eps-M)$ is the magnetic charge
 density. \hks{Here and in the sequel with a slight abuse of notation we identify
 $\chi_{\eps,k}$ with the cut-off function associated with a single
 representative of the ball $B_k$ in $\R^2$.}  The remainder $R_\eps$ can be
 estimated from below as a lower order term if the balls are chosen carefully
 (cf. Proposition \ref{prp-localization}).  Hence, the estimate is reduced to
 local estimates on the balls $B_k$.

 \medskip
 
 \textit{Step 2: Local estimate of leading order terms.} %
 We claim that for any ball $B:=B_k$ and corresponding cut--off
 function $\chi_\eps:=\chi_{\eps,k}$ with $\spt \chi_\eps \CUS B$, we have
 \begin{align}
   \liminf_{\eps\to  0}\ \Big(\nu_\eps(B) +  N_\eps[\chi_\eps] \Big)  \ \geq \ %
   2 \int_{\SS_m \cap B} f(\sqrt{\lam}|n \cdot e_1|) \ d\HH^1,
   \end{align}
   where $f$ is the energy density of the limit functional, i.e. 
 \begin{align} \label{def-f} %
   f(s) \ := \ \left(1+s^2 \right)\chi_{\{s\leq 1\}} + 2s \chi_{\{s>1\}} \ %
   = \ \inf_{\alp \geq 1} \Big[ \alp + \frac{s^2}{\alp} \Big] \ %
  \qquad \text{for $s > 0$.}
 \end{align}
 The lower bound  is determined by a balance between interfacial and
 magnetostatic terms: We first note that by \eqref{eq-compactness} and the lower
 semi-continuity of the BV norm, we have
 \begin{align} \label{def-alp} %
   \liminf_{\eps \to 0} \nu_\eps(B) \ \geq \ \liminf_{\eps \to 0} \|Du_\eps
   \|(B) \ =: \ \alp \NN{Du}(B)
\end{align}
for some $\alp\geq 1$, where the difference $\alp-1 \geq 0$ quantifies the local
presence of oscillations  as $\eps \to 0$. In view of the second identity in \eqref{def-f},
it is then enough to show that
\begin{align} \label{N-need}
  \liminf_{\eps\to  0} N_\eps[\chi_\eps] \ \geq \    \frac{2\lam}{\alp} \int_{\SS_m \cap B} |n \cdot e_1|^2 \ d\HH^1.
\end{align}

\medskip

\textit{Step 3: Estimate of main nonlocal term.} For the estimate of
\eqref{N-need}, we first note that 
\begin{align}
N_\eps[\chi_\eps]=\frac{\pi\lambda}{2|\ln\eps|}\int_{\R^2} \big||\nabla|^{-\frac 12}(\chi_{\eps} \sig_\eps) \big|^2 \ dx.
\end{align}
Then we use the dual
characterization of the $\dot{H}^{-\frac 12}(\R^2)$-- norm, i.e. we use that for
any $\Phi\in \dot H^{\frac 12}(\R^2)$ we have 
\begin{align} \label{dual-intro} %
  \int_{\R^2} \chi_\eps^2 \sig_\eps \Phi \ dx \ %
  \leq \ %
  \Big(\int_{\R^2}\big| |\nabla|^{-\frac 12}(\chi_\eps\sig_\eps)\big|^2 \ dx
  \Big)^{\frac 12}%
  \Big( \int_{\R^2} ||\nabla|^{\frac 12} (\chi_\eps\Phi)|^2 \ dx \Big)^{\frac
  12}.
\end{align}
For the construction of the test functions, we choose a cut-off function
$\eta_\eps \in \cciL{\R}$ with logarithmically decaying profile (see
Definition \ref{def-psieps}). The functions $\Phi_\eps$ have the form
\begin{align} \label{def-Phieps-0} %
  \Phi_\eps(x) \ := \ \eta_{\eps}( \dist{(\gam_\eps,x)}), 
\end{align}
where $\gam_\eps$ are carefully modified level sets
$\{u_\eps=t\}$ in $B_k$ with certain $t$ such that in particular the following
properties hold:
\begin{enumerate}
\item\label{gam-seplen} %
  For each $\eps > 0$, the set $\gam_\eps$ separates the regions where
  $m_\eps \approx e_1$ and $m_\eps \approx - e_1$ (up to small sets) and
  converges to $\SS_m$ in a weak sense as $\eps \to 0$.  Furthermore, the sets
  $\gam_\eps$ have uniformly controlled length.
\item\label{gam-cap} %
  the length of level sets of certain distance from $\gam_\eps$ is
  controlled, i.e.
  $\HH^1(d_{\gam_\eps}^{-1}(t) \cap B) \leq 2\HH^1(\gam_\eps)$ for  a.e.
  $t\in (0, \d_0)$ and some $\d_0 > 0$.
\end{enumerate}
The precise statements and the details of the construction are given in Sections
\ref{sus-sepcurve} and \ref{sus-levelset} (see Lemma \ref{lem-separate2}).  To
construct the sets $\gam_\eps$ such that they satisfy \ref{gam-seplen} it would
be enough to choose them as suitably chosen level sets of $u_\eps$.  To achieve
\ref{gam-cap} we modify $\gam_\eps$, and we use and adapt the level sets
estimates from \cite{BZ88}, which are based on the Gauss--Bonnet theorem and
rely crucially on the two--dimensionality of our problem.

\medskip

Estimate \eqref{N-need} then follows by deriving sharp
estimates for the terms in \eqref{dual-intro}, i.e.
 \begin{align}
   \int_{\R^2} \big| |\nabla|^{\frac 12} (\chi_\eps \Phi_\eps )\big | ^2 \ dx \ %
   &\leq  \ \Big(\frac{\pi}{|\ln {\eps}|}+ o(\frac 1{|\ln \eps|})\Big) \HH^1(\gam_\eps) , \label{gen-1}\\
   \int_{\R^2} \chi_\eps^2 \Phi_\eps\sig_\eps\ dx \ %
   &\geq   \ 2\int_{\SS_m}\chi_\eps^2 (n_{\eps} \cdot e_1) \ d\HH^1 - o(1) \label{gen-2}
 \end{align}
 for $\eps \to 0$. It is essential that we get the precise leading order
 constant in both \eqref{gen-1}--\eqref{gen-2}. We note that \eqref{gen-1} means
 that the capacity of the curves $\gam_\eps$ is asymptotically controlled by
 their length, where we recall that the capacity of a set is the reciprocal of
 the minimal stray field energy created by a charge distribution with total
 charge one on the set (cf. e.g. \cite{LiebLoss-Book}). Thus the assertion
 \ref{gam-cap} is the right estimate for the derivation of \eqref{gen-1}.

\medskip

\textbf{Limsup inequality.} The limsup inequality follows by
  constructing a suitable recovery sequence. This recovery sequence is
constructed by patching one--dimensional transition layers together. Additional
care is taken in the supercritical case, where we replace transition layers with
large slopes by fine combination of suitable zigzags. The estimate of the stray
field energy relies on the singular integral representation for the
$\dot{H}^{-\frac 12}$-norm (cf. Lemma \ref{lem-sing_H12}). Using this
representation, we can localize the self--interaction term to each patch which
yields the leading order contribution of the energy. It can be furthermore shown
that the interaction energy between different patches is of lower order. The
construction and estimates for the recovery sequence are given in Section
\ref{sec-recovery}. 

\subsection{Formal derivation of the model} \label{sus-derivation}

Before we give the proofs of the main results, we show how the model
  \eqref{EE}--\eqref{charge-zero} can be derived from a non-dimensionalization
of the underlying three-dimensional physical model
\cite{LandauLifshitz-1935} (for similar arguments see
e.g. \cite{Garcia-1999phd,Melcher-2003a}). We apply some heuristic
simplifications which we believe can be justified rigorously.

\medskip

We consider a uniaxial ferromagnetic in the shape of a thin plate of the form
$\Ome = \R^2 \times [0,t]$ and magnetization
$\OL m = (m, m_3):\R^2\times [0,t]\to \mathbb{S}^2$. \ed{The single, energetically
preferred magnetization direction of the material in consideration is given by
the $e_1$--axis. } Since we are interested in a charged transition layer, we
enforce this transition by boundary conditions, i.e. we assume $\OL m = \pm e_1$
for $\pm x_1> w$ for some $w>0$. In order to formulate the problem, we
assume that $\OL m$ is $L$--periodic in $x_2$ direction for some large
periodicity $L$, noting that our estimates do not depend on $L$. Let
$Q_L:=\R\times \T_L$, where $\T_L:=\R / (L \Z)$. As described before, we assume
that there is a background magnetization $\OL M$ which ensures that the system
is charge free, i.e. the analogous assumption to \eqref{charge-zero} holds.
Physically, this corresponds to the fact that there are no magnetic monopoles.
In a partially non--dimensionalized form, the Landau-Lifshitz energy
\cite{LandauLifshitz-1935,DKMO-2001} then takes the form
\begin{align} %
  \begin{aligned} \label{E-1}
    \EE[\OL m] 
    &= \ d^2 \int_{Q_L\times (0,t)} |\OL \nabla \OL m|^2 \ d \OL x + Q \int_{Q_L\times (0,t)} (m_2^2 + m_3^2) \ d \OL x  + \int_{Q_L\times \R} |\OL h|^2  \ d \OL x,
  \end{aligned}
\end{align}
with the notation $\OL x = (x, x_3)$ and $\OL \nabla = (\nabla, \p_3)$.  Here,
the material parameter $d$ is the so called \textit{exchange length}, modelling
the relative strength of the exchange and magnetostatic or stray field
energy. The dimensionless constant $Q  > 0$ is the \textit{quality factor},
which describes the relative strength of the material anisotropy. The stray
field $\OL h \in L^2(Q_L \times \R; \R^3)$ is given by
$\OL h \ := \ \OL\nabla (-\Delta_{Q_L\times \R})^{-1} \OL\nabla \cdot (\OL m -
\OL M)$, which is the Helmholtz projection of $\OL m-\OL M$ on the gradient
fields.  If the magnetic film is sufficiently thin, it is reasonable to assume
that the magnetization does not vary in the thickness direction within the
film. In this case, the stray field equations can be solved explicitly,
cf. \cite{DKMO-review}. Also assuming that $m$ varies on length scales much
larger than $t$ we can apply a standard low frequency approximation for the
stray field energy (see e.g. \cite{Carbou-2001, KohnSlastikov-2005}). With the
change of variables $\OL x\mapsto \frac{\OL x}{w}$, denoting
$\ell:=\frac{L}{w}$, we arrive at the reduced energy
\begin{align} \label{enmea} %
  \begin{aligned}
    \EE_{\textrm{red}}[\OL m] \ %
    &= \ d^2t \int_{Q_\ell}|\nabla \OL m|^2 \ dx%
    + Q t w^2\int_{Q_\ell} (m_2^2 + m_3^2) \ dx + t w^2 \int_{Q_\ell} m_3^2 \ dx
    \\ %
    &\qquad + \ \frac{t^2w}{2} \int_{Q_\ell} \big|
    |\nabla|^{-\frac 12}\nabla\cdot (m-M)\big|^2 \ dx %
  \end{aligned}
\end{align}
for $m \in \AA$. We introduce the dimensionless parameters $\eps$, $\lam$,
and $\alp$ by
\begin{align}
  \eps \ := \ \frac{d }{w Q^{\frac 12}}, \qquad
  \lam \ := \ \frac{t |\ln \eps|}{ 2\pi d Q^{\frac 12}}, \qquad
  \alp \ := \ \frac wt.
\end{align}
Note that $\eps$ represents the ratio of Bloch wall width and sample width $w$
\cite{DKMO-review}. The parameter $\lam$ is related to the relative strength of
the stray field energy for the charged N\'eel wall to the local energy
(cf. \eqref{EE}). The parameter $\alp$ describes the aspect
  ratio. Rescaling the energy, we arrive at
\begin{align} \label{enmea} %
  \frac{\EE_{\textrm{red}}[\OL m]}{2 d tw
    Q^{\frac 12} \ell} \ %
  = \ E_\eps[m] \ + \Big( \frac 1\eps +  \frac{\pi \lam \alp}{2|\ln \eps|} \Big) \int_{Q_\ell} m_3^2 \
  dx. %
\end{align}
For sufficiently thin films we have
$\alp \gg \frac 1\eps |\ln \eps|$ and the out--of--plane component of the
magnetization is penalized heavily. This suggests to assume $m_3 = 0$, and we
arrive at the form \eqref{EE} for the non--dimensionalized energy.
\begin{remark}[Statement of results in initial variables] %
  The $\Gam$--limit in Theorem \ref{thm-gamma} corresponds to the following
  scaling of the initial energy: In the regime $t \ll d$, $d^2 \ \ll \ wt$ and
  $Q^{\frac 12} \ %
  \approx \ \frac{t}{2\pi \lambda d} |\ln (\frac{d^2}{wt})|$, the ground state
  energy in leading order is given by
\begin{align} \label{scaling-0} %
  \min_{\OL m} \EE [\OL m] \ %
  &\approx \  %
    t^2L \big|\ln\big(\frac{d^2}{wt}\big)\big|
    \begin{cases} %
      \displaystyle \left(1+\lambda\right)  %
      \qquad &\text{ if } \lambda\leq 1, \vspace{0.7ex}\\
      \displaystyle \sqrt{4\lambda} %
      &\text{ if } \lambda > 1,
    \end{cases}
\end{align}
where the minimum is taken over all configurations $\OL m \in H^1(Q_L;\mathbb{S}^2)$
  with $\OL m = \pm e_1$ for $\pm x_1> w$. In order to get a corresponding
$\Gam$--limit for the full energy \eqref{E-1} and to rigorously prove
  \eqref{scaling-0}, it is necessary to show that the assumptions made in this
section only lead to errors which are negligible with respect to the leading
order terms in the energy.
\end{remark}

\section{Proof of Theorem \ref{thm-gamma} -- \ed{liminf--inequality}} \label{sec-lower}

\subsection{Level set estimates} \label{sus-levelset}

In this section, we give some general results for the length of level sets for
the distance function to the boundary of sets $\Ome \SUS \R^2$. These
results are used in the construction of our test function in the proof of Lemma
\ref{lem-separate2}, which then is used in the proof of the
  liminf--inequality. The main result is Theorem \ref{thm-level-local}, which
shows that we can modify a set locally such that the boundary of the new set has
a controlled capacity.  The key in the proofs is an application of the
Gauss-Bonnet theorem. The proofs also rely heavily on the two--dimensionality of
the problem.

\medskip

We first consider the situation of bounded simply connected domains $G$ before
addressing more general sets.  The proof of the next lemma follows from the
ideas in \cite[Lemma 3.2.2-3.2.3]{BZ88}. We note that Lemma 3.2.2 in \cite{BZ88}
is stated for inner level sets, i.e. level sets for $d_{\p E}$ inside $E$.
However, it is not hard to see that it holds for outer level sets as
well. For any set $E \SUS \R^2$, we write $d_E(x) := \dist(x,E)$ for the
  distance to this set, noting that the distance to the empty set is infinite.
\begin{lemma}[Level sets of simply connected domains]\label{lem-level-simple}
  For any bounded simply connected domain $G \SUS \R^2$ with piecewise $C^2$
  boundary, we have
  \begin{align}
    \HH^1(d_{ \p G}^{-1}(t)) \ %
    \leq \ 2\HH^1(\p G) \qquad\quad\quad %
    \text{for a.e. $0 \ \leq t \leq \ \frac 1{2\pi} \HH^1(\p G)$.}
  \end{align}
\end{lemma}
\begin{proof}
  By an approximation argument, we may assume that $G$ is
  polyhedral. \ed{Indeed, by linear interpolation one can find a sequence of
    simply connected polygons $G_n\rightarrow G$ such that
   $|G_n|\rightarrow |G|$, $\HH^1(\p G_n)\rightarrow \HH^1(\p G)$ and
    moreover $d_{\p G_n}(x)\rightarrow d_{\p G}(x)$ uniformly. The latter
    implies that $\{d_{\p G_n}(x)>t\}\rightarrow \{d_{\p G}(x)>t\}$ and
    $\HH^1( d_{\p G}^{-1}(t))\leq \liminf_{n\rightarrow \infty} \HH^1( d_{\p
      G_n}^{-1}(t))$ for a.e. $t$. Thus the desired inequality for
    $G$ follows after passing to the limit of the inequality for $G_n$. }

  \medskip
  
  By the isoperimetric inequality, we have
  $t_* :=\max_{x\in G} d_{\p G}(x) \leq t_0:=\frac1{2\pi}\HH^1(\p G)$. For
  $t>0$, let $\mathcal{N}_t:=\mathcal{N}_t(\p G)=\{x\in \R^2: d_{\p G}(x) <t\}$
  denote the $t$-neighborhood of $\p G$.  \hks{Note that $\mathcal{N}_t$ is
    connected \hks{for all $t > 0$} and has at least one hole for
    $t \in (0,t_*]$.} In terms of the Euler characteristics
  $\chi(\mathcal{N}_t)$ of $\mathcal{N}_t$, we can express this as
  \begin{align} \label{euler-char} %
    \chi(\mathcal{N}_t) \ %
    \ed{\leq} \ %
    \begin{TC}
      0 \qquad\qquad &\text{for $t \in (0,t_*]$,}  \\
      1 &\text{for $t > t_*$}.
    \end{TC}
  \end{align}
  Let $\ell_t:= d_{\p G}^{-1}(t) \SUS \R^2$ be the $t$-level set of $d_{\p
    G}$. Let $\widetilde{\mathcal{N}_t}$ \hks{be the completion of
    $\mathcal{N}_t$ with respect to the intrinsic metric $\rho_{\mathcal{N}_t}$,
    where $\rho_{\mathcal{N}_t}(x,y)$ is defined as the infimum of the Euclidean
    length of curves joining $x$ and $y$ in $\mathcal{N}_t$}, and let
  $\widetilde{\ell_t}:= \widetilde{\mathcal{N}_t} \BS \mathcal{N}_t$. \hks{We
    note that $\widetilde{\mathcal{N}_t}$ admits a surjective $1$-Lipschitz map
    to the closure $\OL{\mathcal{N}}_t$ (w.r.t the standard metric) of
    $\mathcal{N}_t$, which induces a surjective 1-Lipschitz map
    $\widetilde{\ell_t}\rightarrow \ell_t\subset \mathbb{R}^2$; this implies
    $\mathcal{H}^1(\widetilde{\ell_t})\geq \mathcal{H}^1(\ell_t)$. The
    inequality can be strict since tangentially aligned boundaries of different
    components are counted twice for $\t \ell_t$ (see Fig \ref{fig-levelset})}.
  By construction, $\t \ell_t$ consists of a collection of oriented, closed,
  piecewise $C^2$ curves with a finite number of vertex points
  $\VV_t \SUS \widetilde {\ell_t}$.  We choose the orientation of each curve such that
  $\mathcal{N}_t$ lies to its left and write $\tau_t$ for the total rotation of
  $\widetilde{\ell_t}$. We also denote the rotation at the vortex points
  $x^* \in \VV_t$ by $\tau_{x^*} \in \hks{[}-\pi,0)$, noting that the case
  $\tau_{x^*} > 0$ does not occur since such singularities are smoothed out (see
  Fig. \ref{fig-levelset}).

  \medskip
  
  \begin{figure}
    \centering %
    \includegraphics[width=10cm]{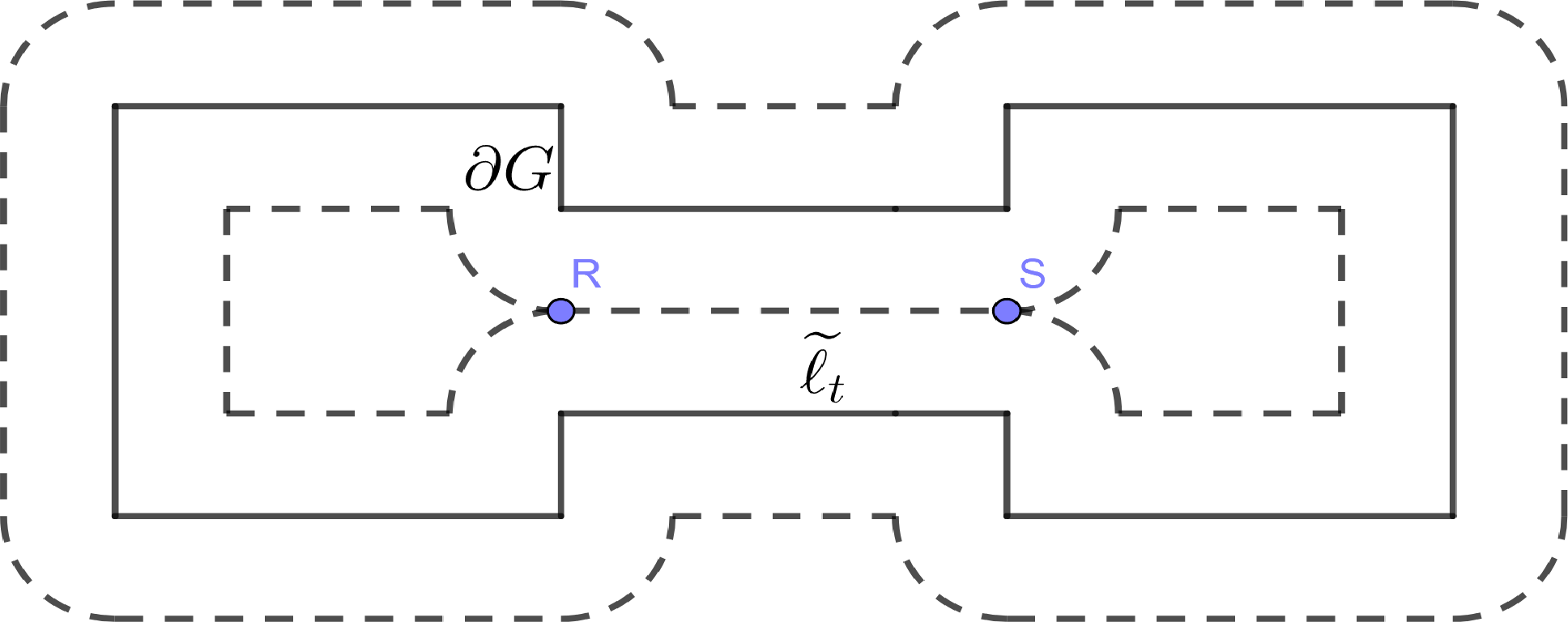}
    \caption{Level set associated to the boundary of a polygonal
      domain: The level set $\ell_t = d_{\p G}^{-1}(t)$ is given by the
        dashed lines. For the curve $\t \ell_t$ coming from the metric
      completion, the line between $R$ and $S$ is counted twice.}
    \label{fig-levelset}
  \end{figure}
  
  We define $g(t):=\HH^1(\widetilde{\ell_{t}})$ for $t > 0$ and
  $g(0) := \lim_{t \to 0} g(t) = 2\HH^1(\ell_0)$. It is then enough to show that
  \begin{align} \label{des-est} %
      g(t) \ \leq \ g(0) \qquad\qquad \text{ for $t\in \big[0,t_0\big]$}. 
  \end{align}
  \textit{Proof of \eqref{des-est}:} We have $g \in C^1(\R \BS \SS)$ for some
  finite set $\SS \subset [0,\infty)$ with
  \begin{align} \label{eq-40book} 
    g'(t) \ = \ %
    \tau_{t} \  - \ \sum_{x^* \in \VV_t} \Big(\tau_{x^*} - 2\tan\frac{\tau_{x^*}}{2} \Big) \qquad %
    \text{for $t\in [0,\infty) \BS \SS$, }
  \end{align}
  and $g(t+0)\leq g(t-0)$ for any $t\in
    \SS$ \cite[Lemma 3.2.3]{BZ88} (note that only the
    "inner" level set $G \cap (\widetilde{\mathcal{N}_t} \BS
    \mathcal{N}_t)$ and with reverse orientation of the curves is considered in
    \cite{BZ88}).
    By the Gauss-Bonnet theorem, we have $\tau_t \ = \ 2\pi
    \chi(\mathcal{N}_t)$ for $t > 0$, which implies $\tau_t \ \leq \ 0$ for $t
    \in (0, t_*]$ by \eqref{euler-char}.  We also note that $\phi -
    2\tan\frac{\phi}{2} \geq 0$ for $\phi \in
    [-\pi,0]$. Hence, by integrating \eqref{eq-40book} over $(0,t)$, i.e.
  \begin{align}\label{eq-inner}
    g(t) \ \leq \ g(0)+\int_0^{t} \tau_s  \ ds \ %
    \leq \ g(0) \qquad %
    \text{ for } t\in [0,t_*],
  \end{align}
  it follows that  \eqref{des-est} holds for $t\in [0,t_*]$.

  \medskip
  
  It remains to show \eqref{des-est} for $t \in (t_*,t_0]$: Without loss of
  generality, we can assume that $t_*<t_0$. In this case, $g$ may have a
  downward--jump at $t_*$, i.e. $g(t_*+0)<g(t_*-0)$. Furthermore, for $t > t_*$ we have
    $\mathcal{N}_t$ $=$ $\mathcal{N}_t(\p G)$ $=$ $\mathcal{N}_t(G)$, where
    $\mathcal{N}_t(G)$ is the $t$-neighborhood of $G$. We define
    $\widetilde{h_t} :=\widetilde{\mathcal{N}_t(G)}\BS \mathcal{N}_t(G)$ for
    $t\geq 0$, where $\widetilde{\mathcal{N}_t(G)}$ is the completion of
    $\mathcal{N}_t(G)$ with respect to the intrinsic metric
    $\rho_{\mathcal{N}_t(G)}$.  For $t\geq 0$, we define $f(t):=\HH^1(\widetilde{h_{t}})$ and
  $f(0) := \lim_{t \to 0} f(t) = \HH^1(\p G)$. With the same arguments as before
  we then get
  \begin{align}\label{eq-outer}
    f(t) \ \leq \ f(0)+ 2\pi \int_0^{t} \chi(\mathcal{N}_t(G)) \ ds \ %
    \leq \ \HH^1(\p G) + 2\pi t \qquad
    \text{ for } t \geq 0,
  \end{align}
  where we have used that $\chi(\mathcal{N}_t(G)) \ed{\leq} 1$ for all $t \geq 0$.  By
  construction, for $t >t_*$, the inner level sets are empty and we thus have
  $g(t) = f(t)$.  In particular, for $2\pi t \leq \HH^1(\p G)$ we have
      $g(t) \leq 2\HH^1(\p G)$.
\end{proof}
\ed{For not simply connected sets, the level set estimate in Lemma
  \ref{lem-level-simple} does not hold in general: a simple counterexample is
  given by the annulus $B_1\BS \overline{B_\delta}$ for
  $\delta\in (0,\frac{1}{4})$. 

\medskip
  
Generally, there is a decomposition $\Ome = \Ome^{(0)} \Delta \Ome^{(1)}$ such
that $\Ome^{(0)}$ satisfies a level set estimate and such that the connected
components of $\Ome^{(1)}$ have controlled size:}
\begin{lemma}[Global level set estimate]\label{lem-level} %
  \hks{Let $\Ome$ be a bounded open subset in $\R^2$ with piecewise $C^2$ boundary } and let
  $0 < \d_0 < \frac 1{2\pi}\mathcal{H}^1(\p \Ome)$. Then there exist two open
  subsets $\Ome^{(0)}$, $\Ome^{(1)}$ of the convex hull of $\Ome$ with
  $\Ome = \Ome^{(0)} \Delta \Ome^{(1)}$ and
  $\p \Ome = \p \Ome^{(0)}\cup \p \Ome^{(1)}$
  such that the following holds:
  \begin{enumerate}
  \item\label{ass-E0} %
    The set $\Ome_0$ satisfies the level set estimate
    \begin{align} \label{est-E0} %
      \HH^1(d_{\p \Ome^{(0)}}^{-1}(t)) \ %
     \leq \ 2\HH^1(\p \Ome^{(0)}) \qquad\qquad %
     \text{for a.e. $t\in \big[0,\delta_0]$.} %
    \end{align}
  \item For any connected component $G$ of $\Ome^{(1)}$, we have
    \begin{align} \label{ass-E11} 
      \HH^1(\p G) \ \leq \ 2\pi \min \{ \delta_0, \dist(\p G, \p \Ome^{(0)}) \}.
    \end{align}
  \end{enumerate}
\end{lemma}
\begin{proof}
  \textit{Construction:} Since $\Ome$ is bounded with piecewise $C^2$
    boundary, its boundary consists of finitely many simple closed curves
    $\CC_k \SUS \p \Ome$, $k \in \II$. Furthermore, by using polygon approximation
    (similar as in Lemma \ref{lem-level-simple}) one may assume that each
    $\CC_k$ is piecewise linear and the distance
    $d_{ij} := \dist(\CC_i,\CC_j)$ between $\CC_i$ and $\CC_j$ are positive
    and pairwise different. By the Jordan-Schoenflies theorem, there is a unique
    decomposition
  \begin{align} \label{dieform} %
    \Ome \ = \ \hks{\bigcup_{k \in \II_+} \bigg(G^{(k)}\BS \bigcup_{\substack{j \in \II_-, G^{(j)} \subset G^{(k)}}} \overline{G^{(j)}}\bigg)}
  \end{align}
  for some index set $\II_+$ and $\II_-$, where $G^{(k)}$, $k\in \II$, is the
  bounded simply connected domain with $\CC_k=\p G^{(k)}$. Roughly speaking,
  $\Ome$ is a union of finitely many connected components, where each component
  is a simply connected domain minus the closure of finitely many simply
  connected domains.  We set
$\II_{00}:=\{j\in \II: \HH^1(\CC_j)\geq 2\pi\delta_0\}$ and iteratively define
  \begin{align}\label{eq-I00}
    \II_{0k} \ := \ \Big\{\ell\in \II \ : \ \HH^1(\CC_\ell) \ \geq \  2\pi \dist \big(\CC_\ell, \bigcup\nolimits_{j\in \cup_{i=0}^{k-1}\II_{0i}} \CC_j \big)\  \Big\} \qquad \text{for $k \geq 1$
    .}
  \end{align}
  With $\II_0:=\bigcup_{i=0}^{\infty}\II_{0i}$, $\II_1:=\II\BS \II_0$ and
  $\II_j^\pm := \II_j \cap \II_\pm$ for $j=0,1$, we then define
  \begin{align} \Ome^{(0)} \ := \ \bigcup_{k \in \II_0^+}
      \bigg(G^{(k)} \BS \bigcup_{\substack{j \in \II_0^-, G^{(j)}\subset
          G^{(k)}}} \overline{G^{(j)}} \bigg)\quad \text{and} \quad %
                      \Ome^{(1)} := \Ome\Delta \Ome^{(0)}.
  \end{align}
Note that if $k\in \II_1^+$, then by our selection procedure $\{j\in \II: G^{(j)}\subset G^{(k)}\}\subset \II_1$. In other words, if an outer loop $\CC_k$ is not selected for $\Ome^{(0)}$, then the whole component $G^{(k)}$ is not contained in $\Ome^{(0)}$. Similarly, if $k
\in \II_1^-$, then $\{j\in \II: G^{(j)}\subset G^{(k)}\}\subset \II_1$. Thus $\Ome^{(1)}$ is the union of $G^{(k)}$, where $k\in \II_1^+\cup \II_1^-$.

\medskip \textit{Conclusion of proof}: By construction, the open sets
$\Ome^{(0)}$, $\Ome^{(1)}$ are subsets of the convex hull of $\Ome$ and satisfy
$\Ome = \Ome^{(0)} \Delta \Ome^{(1)}$ and
$\p \Ome = \p \Ome^{(0)}\cup \p \Ome^{(1)}$. Moreover, \eqref{ass-E11} holds
since by construction for any $\ell \in \II_1$ we have
$\HH^1(\CC_\ell) \leq 2\pi \d_0$ and
$\HH^1(\CC_\ell) \leq 2\pi \dist (\CC_\ell, \bigcup_{j \in \II_0}
\CC_j)$. Hence, it remains to prove \ref{ass-E0}.  In the sequel we write
$\mathcal{N}_t(E) := \{ x \in \R^2 : \dist(x,E) < t\}$ for the open
$t$--neighborhood of $E$ for any set $E \SUS \R^2$ and any $t > 0$.

  \medskip
 \begin{figure}
     \centering %
     \includegraphics[width=6cm]{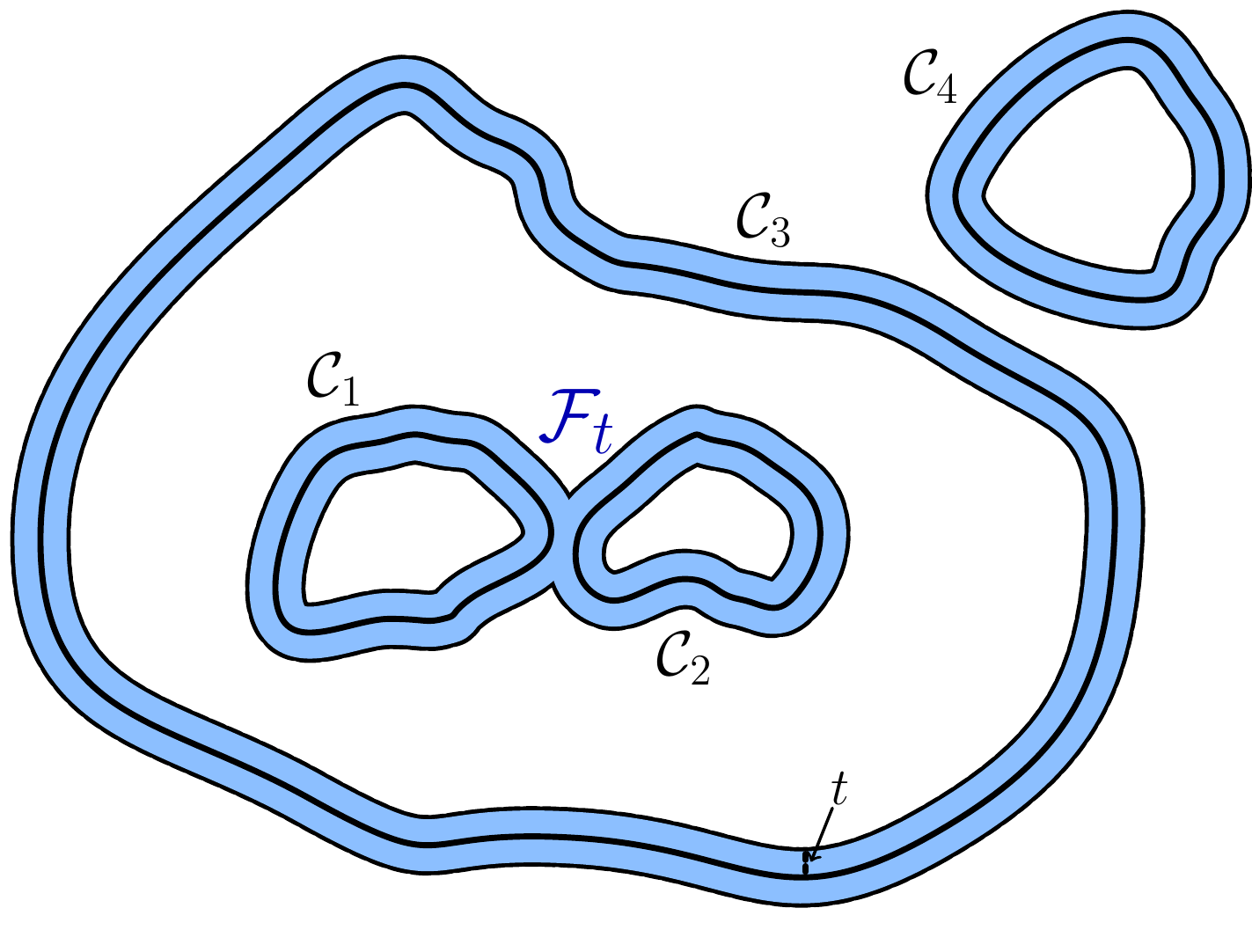}
     \caption{Sketch of a set $\Ome^{(0)}$ with boundary
         $\p \Ome^{(0)} = \CC_1 \cup$ $\CC_2 \cup \CC_3$ $\cup \CC_4$ and
         its $t$--neighborhood $\mathcal{N}_t(\p \Ome^{(0)})$. The
         connected component $\FF_t$ of $\mathcal{N}_t(\p \Ome^{(0)})$ is
         generated by the loops $\CC_1$ and $\CC_2$.}
     \label{fig-levelset}
   \end{figure}

   The number of connected components $N_t$ of the $t$-neighborhood of
   $\mathcal{N}_t(\p \Ome^{(0)})$ of $\p \Ome^{(0)}$ is nonincreasing
     in $t$ and piecewise constant, except at a finite number of
   \textit{merging times}, when  components merge. We consider any
   interval $I = (t_0,t_1]$ such that $N_t$ is constant in $I$ and
   $t_0 < t_1$ are either merging times or $t_0,t_1 \in \{ 0, \d_0 \}$. Then it
   is enough to show \eqref{est-E0} for every single connected component of
   $\mathcal {N}_t(\p \Ome^{(0)})$, $t\in I$. Within this time interval, we
     consider any connected component $\FF_{t_1} \SUS \mathcal{N}_{t_1}(\p \Ome^{(0)})$.  Upon relabelling we can
   assume that the loops contained in $\FF_{t_1}$ are given by
   $\CC_i = \p G^{(i)}$, $1\leq i\leq N$ for some $N\in \N$. We write
   \begin{align}
     \FF_t \ := \ \mathcal{N}_t \big(\bigcup_{i=1}^N \CC_i \big) \ %
     \subset \ \mathcal{N}_t(\p\Ome^{(0)}) \qquad \text{for $t > 0$}
   \end{align}
   for the $t$-neighborhood associated to the loops $\{\CC_i\}_{i=1}^N$.
   Therefore, it suffices to show that
    \begin{align} \label{claim-fi} %
      \HH^1(\ed{\p\FF_t}) \ \leq \ 2 \sum_{i=1}^N c_i, \qquad %
          \text{for $t \in I = (t_0, t_1]$ where  $c_i := \HH^1(\CC_i)$.}
    \end{align}
    By construction $\FF_t$ is connected for all $t\in I$ (as $\FF_{t_1}$ is connected and there is no merging time in $I$) and thus
      the Euler characteristics satisfies $\chi(\FF_t) \leq 1$ for $t \in I$.

  \medskip

  \textit{Proof of \eqref{claim-fi}:} By \cite[Lemma 3.2.3]{BZ88},
  $\HH^1(\ed{\p\FF_t})$ can only have a downward jump
  discontinuity. Hence, by finite induction we can assume that \eqref{claim-fi}
  holds at time $t_0+0$, i.e
  $\lim_{t \searrow t_0} \HH^1(\ed{\p\FF_t}) \ \leq \ 2 \sum_{i=1}^{N} c_i$. %
  We choose $t_* \in [t_0,t_1]$ maximal such that $\chi(\FF_t) \leq 0$ in
  $(t_0,t_*]$. If $t_*$ does not exist, then we set $t_*=t_0$. \ed{If $t_* > t_0$
  then the same argument as for \eqref{eq-inner} shows that $\HH^1(\p \FF_t)$ is
  nonincreasing for $t \in (t_0,t_*]$.}  Together with the induction hypothesis
  this shows that \eqref{claim-fi} holds for $t \in (t_0,t_*]$.

  \medskip

  It remains to show \eqref{claim-fi} for $t \in (t_*,t_1]$: Without loss of
  generality, we can assume that $t_*<t_1$. By definition of $t_*$ we have
  $\lim_{t \searrow t_*} \chi(\FF_t) = 1$ (note that $\chi(\FF_{t}) \leq 1$
  for $t \in I$ since $\FF_t$ is connected in $I$). Since $\FF_t$ is
  connected, this implies that $\FF_t$ is simply connected for $t = t_*+0$. It
  follows that $G^{(i)} \SUS \FF_t$ for $t > t_*$ and for all $1 \leq i \leq
  N$. In particular, with the notation
  \begin{align}
    \GG_t \ := \ \mathcal{N}_t\big(\bigcup_{i=1}^N G^{(i)}\big) \qquad
    \text{for $t > 0$}
  \end{align}
  we have $\FF_t = \GG_t$ for $t > t_*$. It then is enough to show
  \begin{align} \label{claim-fi-2} %
    \HH^1(\ed{\p\GG_t}) \ \leq \ 2 \sum_{i=1}^N c_i \qquad %
    \FA{t \in (t_*,t_1].}
    \end{align}
    In fact, we will show that \eqref{claim-fi-2} holds for all $t \in [0,t_1)$.

    \medskip

    \ed{We recall that by construction $\FF_0$ consists of $N$ connected
      components and $\FF_t$ is connected for all $t \in I$. Furthermore,
      since the set of merging times of the connected components of $\FF_t$ for
      $t\in (0,t_1]$ is a subset of $\{\frac 12 d_{ij}\}$ with $d_{ij}$ pairwise
      different by our assumption at the beginning of the proof, hence
      there are $N-1$ merging times $0 < s_1 < s_2 < \ldots < s_{N-1} \leq t_1$
      for connected components in $\FF_t$. Moreover, the number of connected
      components of $\FF_t$ decreases by precisely $1$ at each merging
      time. Thus in the time interval $(s_{i-1},s_i]$ we have $N_i := N-i+1$
      connected components for $1 \leq i \leq N$, where we have set $s_0 := 0$
      and $s_N :=t_1$.}

  \medskip
  
  Since the length of the outer boundary $\HH^1(\p\GG_t)$ grows at a rate of at
  most $2\pi$ for each single component of $\GG_t$ (cf. \eqref{eq-outer}), and
  moreover, the number of components of $\GG_t$ is no larger than the number of
  components of $\FF_t$ for each $t\in (0,t_1]$, integrating over $(0,t)$ yields
  \begin{align} \label{relo} %
    \HH^1(\p \GG_t) \ %
    \leq \ \sum_{i=1}^N c_i + 2\pi \sum_{i=1}^{N} N_{i}(s_{i}-s_{i-1}) \ %
    = \ \sum_{i=1}^N c_i + 2\pi \sum_{i=1}^{N} s_i.
  \end{align}
  We want to show that $2\pi\sum_{i=1}^{N}s_i\leq \sum_{i=1}^N c_i$. Without
  loss of generality, we assume $c_1 \geq \ldots \geq c_N$. Let
  $\{\FF_{s_i}^{(j)}, 1 \leq j \leq N_i\}$ be the set of connected components
  included in $\FF_{t}$ for $t \in (s_{i-1},s_i]$, and let $\JJ_{s_i}^{(j)}$ be
  the corresponding index set of the loops $ \CC_\ell$,
  $\ell \in \JJ_{s_i}^{(j)}$, included in $\FF_{s_i}^{(j)}$, i.e.
  $\FF_{s_i}^{(j)} = \mathcal{N}_t(\bigcup_{\ell \in \JJ_{s_i}^{(j)}} \CC_\ell)$
  for $t \in (s_{i-1},s_i]$. From our construction \eqref{eq-I00} and since
    $s_i \leq \d_0$, we have
  \begin{align} \label{rel} 2\pi s_i \ %
    \lupref{eq-I00}\leq \ \min_{1 \leq j \leq N_i} \ \max_{k \in
    \JJ_{s_i}^{(j)}} \ c_k \ %
    \ \leq \ c_{1 + (N_i -1)} \ = \ c_{N-i+1}.
\end{align}
Inserting \eqref{rel} into \eqref{relo}, we obtain \eqref{claim-fi-2} for
$t \in (0,t_1]$.
\end{proof}
For a sufficiently regular set $\Ome \SUS \R^2$ and for some
$\widehat x \in \R^2$ and $\rho > 0$, we next derive a local level set estimate
for $\p\Ome \cap B_\rho(\widehat x)$.
\begin{theorem}[Local level set estimate] \label{thm-level-local} %
  Let $\Ome \CUS \R^2$ be bounded and open with $\p \Ome \in C^2$. For
    $\hat x \in \R^2$, $\rho > 0$ let $B_\rho := B_\rho(\hat x)$ and suppose
    that $\Ome\cap B_\rho \neq \emptyset$. Let
  $0 < 2\pi \delta_0 \leq \min \{ \frac \rho{8}, \HH^1(\p \Ome\cap B_\rho)
  \}$. Then there exist two subsets
  $\Ome_\rho^{(0)}, \Ome_\rho^{(1)} \SUS B_\rho$ with %
  \begin{align} \label{levloc-symdiff} %
    \Ome \cap B_{\rho-2\delta_0} \ = \ (\Ome_\rho^{(0)} \Delta \Ome_\rho^{(1)}) \cap B_{\rho-2\delta_0}
\end{align}
such that with $\gamma_\rho:=\p\Ome^{(0)}_\rho\cap B_\rho$ the following
holds:
\begin{enumerate}
\item The length of the boundaries is estimated by
\ed{\begin{align}\label{levl-locl-24}
  \max \{ \mathcal{H}^1(\gamma_\rho),  \HH^1(\p \Ome^{(1)}_\rho\cap B_{\rho}) \} \ %
  \leq \ \HH^1(\p\Ome\cap B_\rho).
\end{align}}
\item We have the level set estimate
\begin{align}
        \HH^1(d_{\gamma_\rho}^{-1}(t) \cap B_{\rho-4\delta_0}) \ %
    &\leq \ 2\HH^1(\gamma_\rho) \quad\qquad %
      \text{for a.e. $t\in (0, \d_0)$.} \label{levl-locl-3}
\end{align}
\item Any connected component $G$ of $\Ome_\rho^{(1)}\cap B_{\rho-2\delta_0}$
  satisfies
\begin{align}\label{levl-locl-5}
      \HH^1(\p G) \ \leq \ 2\pi \min \{ \d_0, \dist(\p G, \gamma_\rho) \}.    
\end{align}
\end{enumerate}
\end{theorem}
\begin{proof}
  Let $\Ome_\rho:=\Ome\cap B_\rho$ and
    $S_{2\delta_0} := B_\rho \BS \overline{B_{\rho-2\delta_0}}$. We first construct a new
    set $\t \Ome_\rho$ such that
    \begin{align}
      \t \Ome_\rho \cap B_{\rho-2\delta_0} \ %
      &= \ \Ome_\rho \cap B_{\rho-2\delta_0}, \label{tOme-same}\\
      \mathcal{H}^1(\p\t{\Ome}_\rho\cap B_\rho) \ %
      &\leq \ \mathcal{H}^1(\p\Ome\cap B_\rho). \label{tOme-bd}
    \end{align}
    Furthermore, every connected component $U$ of $\t \Ome_\rho \cap S_{2\d_0}$
    and every connected component of $(\t \Ome_\rho^c)^o \cap S_{2\d_0}$ is a
    sufficiently wide annulus sector of the form
    \begin{align} \label{tOme-con} %
      U \ = \ \{ \hat x + r e^{i \phi}:  r \in (\rho-2\d_0, \rho), \phi \in
      (\theta, \theta+\Delta\theta) \text{ or } [0,2\pi) \} \ %
      \text{with $\rho\Delta\theta> 4 \d_0$}
    \end{align} %
    for some $\theta, \Delta\theta\in [0,2\pi)$, see
      Fig. \ref{fig-set-modification} for an illustration of the set and the
      construction.  We then apply Lemma \ref{lem-level} to the modified set
    $\t \Ome_\rho$.

  \medskip
  
  \textit{Construction of $\t \Ome_\rho$:} %
  \hks{To construct $\t \Ome_\rho$, we modify iteratively $\Ome_\rho$ as follows:
    \begin{itemize}
    \item We remove any connected component $U$ of
      $\Ome_\rho \cap S_{2\d_0}$ such that 
      \begin{align} \label{U-cond} %
        \HH^1(\p U \cap \p B_{\rho}) \ = \ 0 \qquad \text{or} \qquad %
        \HH^1(\p U \cap \p B_{\rho - 2\delta_0}) \ \leq \ 4 \d_0.
      \end{align}
    \item We fill in the hole related to any connected component $V$ of
      $(\Ome_\rho^c)^o \cap S_{2\d_0}$ such that
      \begin{align} \label{V-cond} %
        \HH^1(\p V \cap \p B_{\rho}) \ = \ 0 \qquad \text{or} \qquad %
        \HH^1(\p V \cap \p B_{\rho - 2\delta_0}) \ \leq \ 4 \d_0.
      \end{align}
    \end{itemize}
    We note that the above modifications might create a new boundary portion
    along $\p B_{\rho-2\delta_0}$. However, the total length of the boundary is
    not increased: If we e.g. remove a connected component $U$, then we might
    create a new boundary at $\p B_{\rho - 2\d_0} \cap \p U$ for the modified
    set. However, if $\HH^1(\p U \cap \p B_{\rho}) = 0$, then removing $U$ does
    not increase the total length of the boundary (since the inner set
    $B_{\rho-2\d_0}$ is convex). If $\HH^1(\p U \cap \p B_{\rho-2\d_0}) \neq 0$
    and $\HH^1(\p U \cap \p B_{\rho}) \leq 4 \d_0$ then again the total length
    of the boundary does not increase (since $\p B_{\rho-2\d_0}$ and
    $\p B_{\rho}$ have distance $2\d_0$). We note that the final set after
      application of the above algorithm is not unique (and depends on the order
      of steps taken). Our argument, however, works independently on the
      specific choice final set. 

    \medskip

    For any remaining connected component $U$ of $\Ome_\rho \cap S_{2\d_0}$, one
    has $\HH^1(\p U \cap \p B_\rho)>0$ and
    $\HH^1(\p U\cap \p B_{\rho-2\delta_0})>4\delta_0$, and moreover
    $\p U \cap \p B_{\rho-2\d_0} = \{ \widehat x + (\rho-2\d_0) e^{i\phi} : \phi
    \in [\theta,\theta+\Delta\theta]) \}$ for some
    $\theta, \Delta\theta\in [0,2\pi)$. Then for any such connected component
    \begin{itemize}
    \item we replace $U$ by the annulus sector of the form \eqref{tOme-con} with
      $\theta, \Delta\theta$ as above.
    \end{itemize}
    By construction, also this modification does not increase the total length
    of the relative boundary in $B_\rho$.  After applying these modifications,
    we hence obtain a set $\t \Ome_\rho$ which satisfies the conditions
    \eqref{tOme-same}, \eqref{tOme-bd} and \eqref{tOme-con}. }
    
  \jf{\begin{figure}
      \hspace{-2cm}
     \vspace*{-0.5cm}
     \includegraphics[height=7cm]{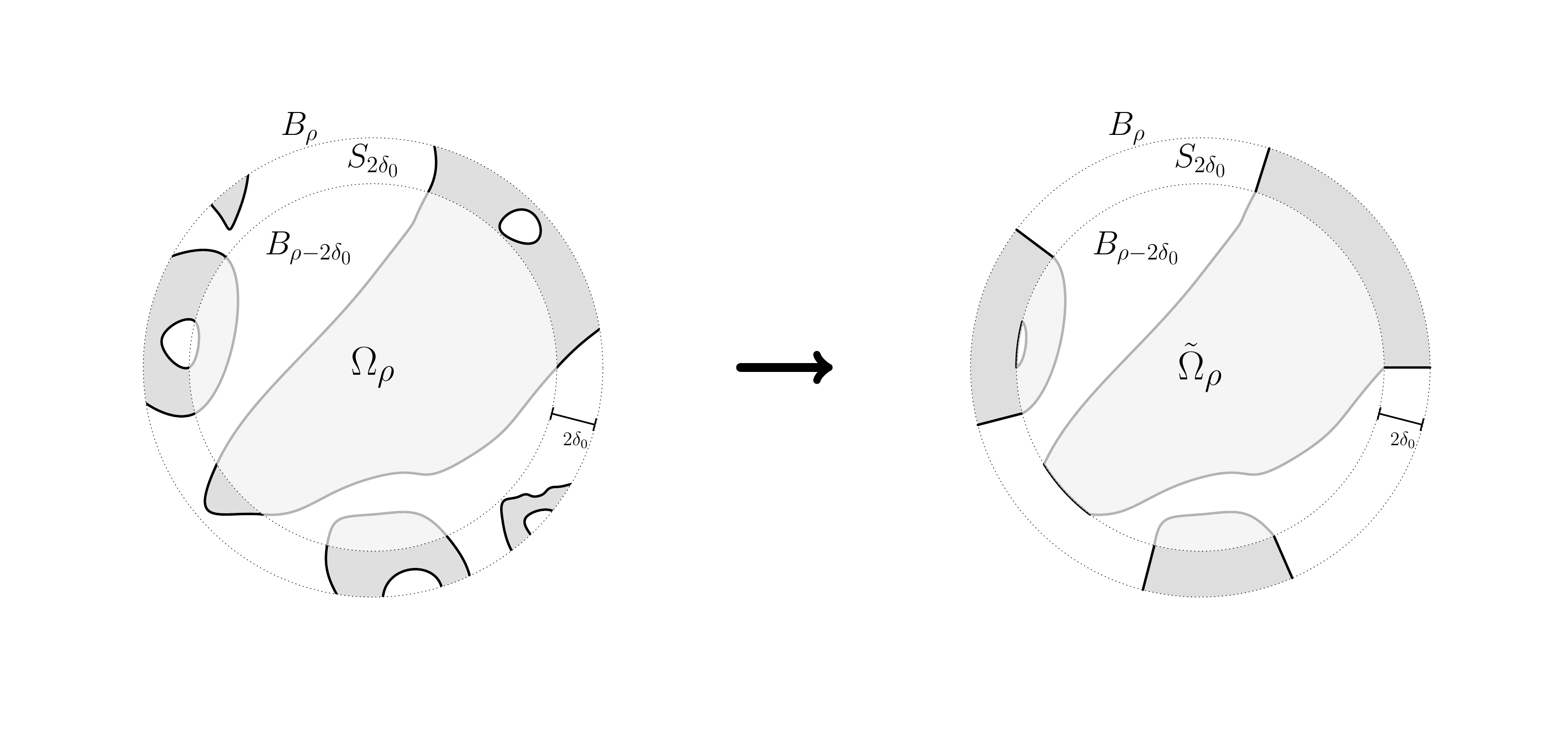}
     \vspace*{-0.3cm}
     \caption{Illustration for the construction of $\t\Ome_\rho$.}
     \label{fig-set-modification}
   \end{figure}}
  
  \medskip

  \textit{Conclusion of proof:} Let $\Ome_\rho^{(0)}$ and $\Ome_\rho^{(1)}$ be
  the sets constructed in Lemma \ref{lem-level} (with $\Ome$ replaced by
    $\tilde\Ome_\rho$). In view of \eqref{tOme-same}, assertions
    \eqref{levloc-symdiff} and  \eqref{levl-locl-5} follow directly from Lemma
  \ref{lem-level}. \ed{Assertion \eqref{levl-locl-24} follows from \eqref{tOme-bd}
  and since by Lemma \ref{lem-level} we have
  $\p\Ome_\rho^{(i)}\cap B_\rho\subset \p\t{\Ome}_\rho\cap B_\rho$ for $i=0,1$.}

  \medskip

  In order to show \eqref{levl-locl-3}, we note that by Lemma \ref{lem-level}(i)
  \begin{gather}
    \label{eq-level_Omega}
    \HH^1(d_{\p\Ome_\rho^{(0)}}^{-1}(t)) \ \leq \ 2\HH^1(\p \Ome_\rho^{(0)})
    \qquad\qquad %
    \text{for a.e. $t\in (0, \d_0)$}.
  \end{gather}
  Since
  $\HH^1(\p\Ome_\rho^{(0)})=\HH^1(\gamma_\rho)+\HH^1(\p\Ome_\rho^{(0)}\cap \p
  B_\rho)$ and since
    $d_{\p\Ome_\rho^{(0)}}^{-1}(t)\cap \overline{B_{\rho-4\delta_0}} \supseteq
    d_{\gam_\rho}^{-1}(t)\cap B_{\rho-4\delta_0}$, then from \eqref{eq-level_Omega} one has
\begin{align}
\HH^1(d_{\gam_\rho}^{-1}(t)\cap B_{\rho-4\delta_0})+ \HH^1(d_{\p\Ome_\rho^{(0)}}^{-1}(t)\cap (B_{\rho+\d_0}\BS B_{\rho-4\delta_0}) ) \leq 2\HH^1(\gamma_\rho)+ 2\HH^1(\p\Ome_\rho^{(0)}\cap \p
  B_\rho).
\end{align}
 Thus  \eqref{levl-locl-3} follows if we can show that
  \begin{align}\label{eq-annulus}
    2\HH^1(\p\Ome_\rho^{(0)}\cap \p B_\rho) \ %
    \leq \ \HH^1(d_{\p\Ome_\rho^{(0)}}^{-1}(t)\cap (B_{\rho+\d_0}\BS B_{\rho-4\delta_0}) ) %
    \quad \text{ for a.e. } t\in (0,\d_0).
  \end{align}
 \ed{We will prove \eqref{eq-annulus} by making use of the simple geometry of
  $\Ome^{(0)}_\rho$ in $S_{2\d_0}$: Indeed, by construction
  $\Ome^{(0)}_\rho \cap S_{2\delta_0}$ and
  $(\Ome^{(0)}_\rho)^c \cap S_{2\delta_0}$ consist of finitely many disjoint
  annulus sectors of the form \eqref{tOme-con}. Thus in view of the simple
  geometry for any connected component $U$ of $\Ome_\rho^{(0)} \cap S_{2\d_0}$
  we have
  \begin{align}\label{eq-annulus-U}
    2\HH^1(\p U \cap \p B_\rho) \ %
    \leq \ \HH^1(d_{\p U}^{-1}(t)\cap (B_{\rho+\delta_0} \BS B_{\rho-\delta_0})) %
    \quad \text{ for } t\in (0,\d_0),
  \end{align}
  We note that $\p\Ome_\rho^{(0)}\cap \p B_\rho$ is the disjoint union of the
  sets $\p U \cap \p B_\rho$ for connected components $U$ of
  $\Ome_\rho^{(0)} \cap S_{2\d_0}$ and --- by the geometry of $\Ome_\rho^{(0)}$
  --- the set
  $d_{\p\Ome_\rho^{(0)}}^{-1}(t)\cap (B_{\rho+\d_0}\BS B_{\rho-4\delta_0})$
  contains the union of the sets
  $d_{\p U}^{-1}(t)\cap (B_{\rho+\delta_0}\setminus B_{\rho-\delta_0})$.
  Furthermore, by construction the sets
  $d_{\p U}^{-1}(t)\cap(B_{\rho+\delta_0}\setminus B_{\rho-\delta_0})$ and
  $d_{\p V}^{-1}(t)\cap (B_{\rho+\delta_0}\setminus B_{\rho-\delta_0})$ are
  disjoint for any two different connected components $U,V$ of
  $\Ome_\rho^{(0)} \cap S_{2\d_0}$.  The estimate \eqref{eq-annulus} then
  follows summing up the estimates \eqref{eq-annulus-U} for each connected
  component $U$. }

\end{proof}

\subsection{Construction of test function} \label{sus-sepcurve}

For the proof of the liminf inequality \eqref{low-bd} in Theorem
\ref{thm-gamma}, we need to show that for any sequence $m_\eps \in \AA$ with
$m_\eps \to m \in \AA_0$ in $L^1$ we have
$\liminf_{\eps > 0} E_\eps[m_\eps] \ \geq \ E_0[m]$. \ed{For the proof, we may
assume that the functions are smooth, i.e. we consider sequences $m_\eps$
  which satisfy
  \begin{align} \label{ass-meps} %
    \begin{aligned}
          &m_\eps \to m  \text{ in $L^1$ as $\eps \to 0$} \\%
          &\qquad  \text{ where } m_\eps = (u_\eps, v_\eps) \in \AA \cap C^\infty(Q_\ell;\mathbb{S}^1) %
          \text{ and where } m = (u, 0) \in \AA_0.
        \end{aligned}
  \end{align}
  Indeed, for a general sequence $m_\eps \in \AA$ one can consider functions
  $m_{\eps,k} \in \AA \cap C^\infty(Q_\ell;\mathbb{S}^1)$ with
  $m_{\eps,k} \to m_\eps$ in $H^1$ for $k \to \infty$. In particular, since the
  energy is continuous with respect to the $H^1$--norm, we also have
  $E_\eps[m_{\eps,k}] \to E_\eps[m_\eps]$ as $k \to \infty$. For the proof of
  the liminf inequality, it is then enough to consider smooth $m_\eps$ by taking a diagonal sequence.} Throughout the
  section, we will also use the notation
  \begin{align} \label{def-Ome0} %
    \Ome_0 \ &:= \ \{ x \in Q_{\ell} \ : \ u(x) \ = \ 1 \}, \\
    \SS_m \ &:= \ \p^* \Ome_0 \text { with measure  theoretic outer normal $n$.}
  \end{align}
Since we often need logarithmic lengh scales, for notational convenience, we write
  \begin{align} \label{def-deps}
    \d_\eps \ := \ \frac 1{|\ln \eps|^{\frac 14}} \qquad
  \end{align}
  throughout this work. For the proof of the liminf inequality we use the
  strategy explained in Section \ref{sus-strategy}. For this, we first give the
  construction of the test functions $\Phi_{\eps,\rho}$
  (cf. \eqref{def-Phieps-0}), associated with the sequence $m_\eps$ and
  localized on the ball $B_{\rho}:=B_\rho(\widehat{x})$ for some fixed
  $\rho\in (0,1)$.  We start with the construction of the separating curves
  $\gam_{\eps,\rho}:=\gam_{\eps,B_{\rho}(\widehat x)}$ with
  $\widehat x\in \SS_m$ (cf. Lemma \ref{lem-separate2}). For that, we first
  choose suitable superlevel sets, whose boundaries converge weakly to the jump
  set of the limit $m$ in $B_{\rho}$ and satisfy a uniform upper bound on the
  lengths of their boundary:
\begin{lemma}[Choice of superlevel sets] \label{lem-separate} %
  Consider a sequence $m_\eps \to m$ which satisfies \eqref{ass-meps} for
    some sequence $\eps \to 0$ and let $\Ome_0$, $\SS_m$, $n$ be given by
    \eqref{def-Ome0}. Let $\widehat x\in \SS_m$ and $\rho\in (0,1)$.  Then there is a subsequence
    $\eps_j \to 0$ and a sequence $t_j \in (-1,1)$ with
    \begin{align} \label{def-teps} %
      |t_j| \ \leq \ 1- \d_\eps \qquad\qquad \forall j \in \N
    \end{align}
    ($\d_\eps$ is given in \eqref{def-deps}) such that the following holds:
    The superlevel sets
  \begin{align} \label{def-omeps-low} %
      \Ome_{\eps_j} \ := \ \{x\in Q_\ell : u_{\eps_j}>t_{j} \}
  \end{align}
  with outer normal $n_{\eps_j}$ satisfy $\p \Ome_{\eps_j} \in
  C^\infty$. Furthermore, with $B_\rho:=B_\rho(\widehat x)$ we have
  \begin{enumerate}
  \item \label{curve-1} %
    $\chi_{\Ome_{\eps_j}}\to \chi_{\Ome_0}\text{ in }L^1(B_\rho)$.
  \item \label{curve-3} %
    $D\chi_{\Ome_{\eps_j}}$ $\wtos$ $D\chi_{\Ome_0}$ weakly$^*$ in $B_{\rho}$.
    In particular,
    \begin{align}
      \NN{D\chi_{\Ome_0}}(B_\rho) \ \leq \ \liminf_{j\to\infty}\NN{D\chi_{\Ome_{\eps_j}}}(B_\rho) \ %
       \leq \ \frac{K}{2}. %
    \end{align}
  \item \label{curve-2} %
    We have (cf. \eqref{def-alp})
    \begin{align}
      \limsup_{j\to \infty} \frac{\NN{D\chi_{\Ome_{\eps_j}}}(B_\rho)}{\NN{D\chi_{\Ome_0}}(B_{\rho})} \ \leq \  \alp  \ %
      := \ \liminf_{\eps \to 0} \frac{\NN{Du_\eps}(B_\rho)}{\NN{Du}(B_\rho)}.
    \end{align} 
  \end{enumerate}
\end{lemma}
\begin{proof}
  For $t \in (-1,1)$ and $\eps \geq 0$, we write
  $\Ome^{\eps}_{t} :=\{x\in Q_\ell: u_{\eps}>t \}$. We choose a subsequence
  $\eps_j \to 0$ such that
  $\lim_{j\to \infty}\|Du_{\eps_j}\|(B_\rho)\rightarrow \alp\|Du\|(B_\rho)$,
  where $\alp$ is given in (iii).  Since $u_{\eps_j}\to u_0$ in $L^1(B_\rho)$,
  by taking a further subsequence (not relabeled), we have
  $\chi_{\Ome^{\eps_j}_{t}}\to \chi_{\Ome^0_t}$ in $L^1(B_\rho)$ for a.e.
  $t\in (-1,1)$, where $\Ome^0_t =\Ome_0$ for $t\in (-1,1)$.  Since
  $\chi_{\Ome^0_t}=\chi_{\Ome_0}$ is independent of $t$ and
  $t\mapsto \chi_{\Ome^{\eps}_t}$ is monotonically decreasing, the convergence
  holds for any superlevel set,
  \begin{align}\label{eq-L1-u}
    \chi_{\Ome^{\eps_j}_{t}} \ \to \ \chi_{\Ome_0} \qquad \text{ in $L^1(B_\rho)$ for any 
    $t\in (-1,1)$}.
  \end{align}
  Let $\mathcal{S}_j\subset (-1,1)$ denote the set of singular values $t$ such
  that $\p \Ome_t^{\eps_j}\cap B_\rho$ is not smooth for $t \in \SS_j$. Then
  $\mathcal{S}:=\bigcup_j \mathcal{S}_j$ has Lebesgue measure zero by Sard's
  theorem.  We aim to find a subsequence $(\eps_{\ell})_{\ell \in \N}$ of
  $\eps_j$ and a sequence $(t_\ell)_{\ell \in \N}$ such that the assertions are
  satisfied.  We claim that for any $\ell \in \N$ there exists
  $t_\ell\in (-1,1)\BS \SS$ and a subsequence $(\eps_{j_k(\ell)})_{k \in \N}$ of
  $(\eps_j)$ such that
  \begin{align}\label{cu-claim}
    \NN{D\chi_{ \Ome^{\eps_{j_k(\ell)}}_{t_\ell}}}(B_\rho) \ %
    \leq \ \alp\NN{D\chi_{\Ome_0}}(B_\rho)+ \frac 1\ell
    \qquad \text{for all $k \in \N$.}
  \end{align}
  Indeed, if not, then
  $\NN{D\chi_{\Ome^{\eps_j}_t}}(B_\rho) >
  \alp\NN{D\chi_{\Ome_0}}(B_\rho)+\frac 1\ell$ for any
  $t \in (-1,1)\BS \SS$ and for any $j \geq j_0(t)$ sufficiently large. By the
  coarea formula, we have
  \begin{align} \label{curve-coarea} %
    \NN{Du_{\eps_j}}(B_\rho) \ = \ \int_{-1}^1 \NN{D\chi_{\Ome^{\eps_j}_t}}(B_\rho) \ dt. %
  \end{align}
  Taking the limit in \eqref{curve-coarea} and using \eqref{def-alp} and
  \eqref{cu-claim} then implies $\alp\NN{D\chi_{\Ome_0}}(B_\rho)$ $>$
  $\alp\NN{D\chi_{\Ome_0}}(B_\rho)+ \frac 1\ell$, which is a contradiction and
  hence yields \eqref{cu-claim}.

  \medskip
  
  For each fixed $\ell \in \N$, there is $k_\ell \in \N$ large such that
  $|t_\ell| < 1 - \d_{\eps_{j_k(\ell)}}$ for all $k\geq k_\ell$. In view of
  \eqref{eq-L1-u}, by making $k_\ell$ possibly larger we have
  $\NNN{\chi_{\Ome^{\eps_{j_{k_\ell}(\ell)}}_{t_\ell}} -
    \chi_{\Ome_0}}{L^1(B_\rho)} \leq \frac 1\ell$. For the diagonal
  sequence $\eps_\ell:=\eps_{j_{k_\ell(\ell)}}$, we then have
  $|t_\ell| < 1 - \d_{\eps_{\ell}}$, $\Ome_{\eps_\ell}$ has a smooth boundary
  and (i) holds. In view of \eqref{cu-claim}, (iii) holds.
  
  \medskip
  
  It remains to show (ii): In view of \ref{curve-1}, we get
  \begin{align} \label{ii-fast}
    \int_{\Ome^{\eps_{\ell}}_{t_\ell}} \dv \zet \ dx \ \to \ \int_{\Ome_0} \dv
    \zet \ dx \qquad\qquad \text{for any $\zet\in C^1_c(B_\rho)$}. 
  \end{align}
  From \ref{curve-2}, it follows that the sets
  $\Ome_{\eps_\ell}:=\Ome^{\eps_{\ell}}_{t_\ell}$ have uniformly bounded perimeter. Together
  with \eqref{ii-fast}, this implies the weak convergence
  $D\chi_{\Ome_{\eps_\ell}} \wtos D\chi_{\Ome_0}$. The second claim in (ii) follows
  from the lower semi-continuity of the total variation measure.
\end{proof}
The sets $\p\Ome_{\eps_j} \cap B_\rho$, where $\Ome_{\eps_j}$ are the
  superlevel sets constructed in Lemma \ref{lem-separate}, cannot be
directly used for the definition of the separating curves as described in
  Section \ref{sus-strategy}, as the capacity of
$\p\Ome_{\eps_j} \cap B_\rho$ might be too large.  However, using
Theorem \ref{thm-level-local} in Section \ref{sus-levelset}, in the next
  lemma we construct a slightly modified set
  $\Ome_{\eps_j,\rho}^{(0)} \SUS B_\rho$ which approximates
  $\Ome_{\eps_j} \cap B_\rho$ and such that
  $\gam_{\eps_j,\rho} := \p\Ome_{\eps_j,\rho}^{(0)} \cap B_\rho$
  satisfies the desired properties of the separating curves in Section
  \ref{sus-strategy}:
\begin{lemma}[Separating curves]\label{lem-separate2}
 \hks{ Consider a sequence $m_\eps \to m$ with $\eps\to 0$ which satisfies \eqref{ass-meps} and let $\Ome_0$, $\SS_m$, $n$ be given by
    \eqref{def-Ome0}.} Then for any $\widehat x\in \SS_m$ there is
    $\hat\rho > 0$ such that for any $\rho \in (0,\hat\rho)$ the
    following holds: Let $\eps_j \to 0$ and $t_j \in (-1,1)$ be the sequences
    constructed in Lemma \ref{lem-separate} and let $\Ome_{\eps_j}$ be defined by
    \eqref{def-omeps-low}.  Then there exist sets
    $\Ome_{\eps_j,\rho}^{(0)}, \Ome_{\eps_j,\rho}^{(1)} \SUS
    B_\rho(\widehat x) =: B_\rho$ such that the following holds:
  \begin{enumerate}
\item \label{sep2-1} %
    $\chi_{\Ome_{\eps_j,\rho}^{(0)}}\to \chi_{\Ome_{0}}\text{ in }L^1(B_\rho)$ as $j \to  \infty$.
  \item \label{sep2-3} %
    Let
    $\gam_{\eps_j,\rho} := \p\Ome_{\eps_j,\rho}^{(0)} \cap B_\rho$
    and let $n_{\eps_j,\rho}$ be the outer normal of
    $\gam_{\eps_j,\rho}$ w.r.t.
    $\Ome_{\eps_j,\rho}^{(0)}$. Then we have
    $n_{\eps_j,\rho} \HH^1\llcorner \gamma_{\eps_j,\rho} \
    \wtos \ n \HH^1\llcorner \SS_m$ in $B_\rho$.  In particular
    \begin{align}
      \HH^1(\SS_m\cap B_\rho) \ \leq \ \liminf_{j \to
      \infty}\HH^1(\gamma_{\eps_j,\rho}).
    \end{align}
  \item\label{sep2-0a} %
    $\HH^1(\gamma_{\eps_j,\rho})\leq \HH^1(\p\Ome_{\eps_j}\cap
    B_\rho)\leq K$.
  \item \label{sep2-2} We have
    \begin{align} \label{def-alp2} %
      \limsup_{{\eps_j} \to 0} \frac{\HH^1(\gamma_{{\eps_j},\rho})}{\HH^1(\SS_m\cap B_\rho)} \ %
      \leq \ \alp \ := \ \liminf_{\eps \to 0} \frac{\NN{Du_\eps}(B_\rho)}{\NN{Du}(B_\rho)}.
    \end{align}
  \item \label{sep2-0b}  %
    We have the level set estimate
    \begin{align} 
      \qquad\qquad %
      \HH^1(d^{-1}_{\gamma_{{\eps_j},\rho}}(t)\cap B_{\rho-8\delta_{\eps_j}}) \ %
      \leq \ 2\HH^1(\gamma_{{\eps_j},\rho}) \quad \forall t\in (0,2\delta_{\eps_j}).
    \end{align}
  \item\label{sep2-auchnoch}
    $|\Ome_{{\eps_j},\rho}^{(1)}| \ \leq \ C\ed{(K+1)}\d_{\eps_j}  \ \to 0$ \ %
    \qquad \text{as $j \to \infty$}.
  \item \label{sep2-4a} %
    $\HH^1(\ed{\p \Ome_{{\eps_j},\rho}^{(1)}\cap B_{\rho-4\delta_{\eps_j}}}) \ %
    \leq \ \HH^1(\ed{\p \Ome}_{\eps_j}\cap
    B_\rho) \ \leq \ K $.
  \item \label{sep2-4b}
    The connected components  $G_{{\eps_j},\rho}^{(k)}$ of  $\Ome_{{\eps_j},\rho}^{(1)}\cap B_{\rho-4\delta_{\eps_j}}$ satisfy
 \begin{align}
   \frac 1{2\pi} \HH^1(\p G_{{\eps_j},\rho}^{(k)}) \ \leq \ \min \{2\delta_{\eps_j},\dist(\p G_{{\eps_j},\rho}^{(k)}, \gamma_{{\eps_j},\rho})\}.
\end{align}
\item \label{sep2-dazu} %
  \ed{We have
  \begin{align}
      1 + u_{{\eps_j}} \ \geq \ \d_{\eps_j} %
        &\qquad\qquad\text{in  $B_{\rho-4\delta_{\eps_j}} \cap \Ome_{{\eps_j},\rho}^{(0)} \BS \Ome_{{\eps_j},\rho}^{(1)}$,} \\
      1 - u_{{\eps_j}}\ \geq \ \d_{\eps_j}  %
        &\qquad\qquad\text{in $B_{\rho-4\delta_{\eps_j}} \cap (\Ome_{{\eps_j},\rho}^{(0)})^c\BS  \Ome_{{\eps_j},\rho}^{(1)}$}. 
  \end{align}}
  \end{enumerate}
\end{lemma}
\begin{proof}
  We use the short notation $\eps := \eps_j$.  Since $\hat x \in \SS_m$ and
    by Lemma \ref{lem-separate}(ii), for each $\rho\in (0,\hat\rho)$ with $\hat\rho$ sufficiently small we have
    $\HH^1(\p \Ome_\eps\cap B_\rho) \geq 4\pi \delta_\eps$ for sufficiently small $\eps$. Moreover, by Lemma \ref{lem-separate}(iii) and
    \eqref{def-alp2}, for $\eps$ sufficiently small one has
    $\HH^1(\p\Ome_\eps\cap B_\rho)\leq K$. We hence can apply Theorem
    \ref{thm-level-local} to the set $\Ome:=\Ome_{\eps}$ and
    $\delta_0:=2\delta_\eps$.  The application of this theorem yields two sets
    $\Ome_{\eps,\rho}^{(0)}$, $\Ome_{\eps,\rho}^{(1)} \SUS B_\rho$ which satisfy
    \ref{sep2-0a}, \ref{sep2-0b}, \ref{sep2-4a} and \ref{sep2-4b}.

  \medskip
    
  \textit{\ref{sep2-auchnoch}:} By the isoperimetric inequality and by
  assertions \ref{sep2-4a} and \ref{sep2-4b}, we have
    \begin{align}
      |\Ome^{(1)}_{\eps,\rho}| \ %
      &\stackrel{\ref{sep2-4b}}= \ \sum_k |G_{\eps,\rho}^{(k)}| +|B_\rho\BS B_{\rho-4\delta_\eps}| %
        \leq  \delta_{\eps}\HH^1(\p (\Ome^{(1)}_{\eps,\rho}\cap B_{\rho-4\d_\eps})) + |B_{\rho}\BS B_{\rho-4\delta_\eps}|\notag\\
      &\stackrel{\ref{sep2-4a}}\leq \ \d_\eps\HH^1(\p \Ome_{\eps}\cap B_\rho) + C\rho\delta_\eps\ %
        \lupref{meps-bound}\leq  \ C\d_\eps (K+\rho) \ \to 0.\label{eq-E1}
    \end{align}
    \textit{\ref{sep2-1}, \ref{sep2-3}:} Assertion \ref{sep2-1} follows from
    \eqref{eq-E1} together with Lemma \ref{lem-separate}\ref{curve-1}. Assertion
    \ref{sep2-3} then follows from \ref{sep2-1} and the uniform boundedness of
    the perimeter of $\Ome^{(0)}_{\eps,\rho}$, as in the proof for Lemma
    \ref{lem-separate}.

  \medskip

  \textit{\ref{sep2-2}:} \ed{From \ref{sep2-0a} we have
  $\HH^1(\gamma_{\eps,\rho})\leq\HH^1(\p\Ome_\eps\cap B_\rho)$.  Taking the
  limsup and using Lemma \ref{lem-separate}\ref{curve-2} we obtain
  \ref{sep2-2}.}

    \medskip

    \textit{\ref{sep2-dazu}:} \ed{In view of \eqref{levloc-symdiff}, we
      have
      $(\Ome_{\eps,\rho}^{(0)} \BS \Ome_{\eps,\rho}^{(1)}) \cap
      B_{\rho-4\delta_\eps} \subset \Ome_{\eps}$ and
      $((\Ome_{\eps,\rho}^{(0)})^c \BS \Ome_{\eps,\rho}^{(1)}) \cap
      B_{\rho-4\delta_\eps} \subset \Ome_\eps^c$. In view of the definition of
      $\Omega_\eps$ in \eqref{def-omeps-low} and \eqref{def-teps}, we conclude
      that \ref{sep2-dazu} holds.}
\end{proof}
We are ready to give the definition of the test function $\Phi_{\eps_j,\rho}$:
\begin{definition}[Test function $\Phi_{\eps_j,\rho}$] \label{def-psieps} %
  Consider a sequence $m_\eps \to m$ with $\eps\to 0$ which satisfies \eqref{ass-meps} and let
  $\Ome_0$, $\SS_m$, $n$ be given by \eqref{def-Ome0}.  Let $\hat\rho > 0$ be
  the constant from Lemma \ref{lem-separate2}.  Let $\gamma_{\eps_j,\rho}$ be a
  sequence of separating curves associated with $B_\rho(\widehat x)$ with
  $\widehat x\in \SS_m$ and $\rho\in (0,\hat\rho)$, which satisfies the
  assertions of Lemma \ref{lem-separate2}.  We define
  $\Phi_{\eps_j,\rho} \in \Lip(Q_\ell) $ with
  $\spt \Phi_{\eps_j,\rho} \Subset
  d_{\gam_{\eps_j,\rho}}^{-1}([0,2\delta_{\eps_j}))$ by
  \begin{align} \label{def-Phieps} %
    \Phi_{\eps_j,\rho}(x) \ := \ \eta_{\eps_j}(d_{\gam_{\eps_j,\rho}}(x)),
  \end{align}
  where $\d_\eps$ is defined in \eqref{def-deps} and $\eta_\eps \in C_c^\infty(\R)$ for $\eps\in (0,\frac 14)$ is given by
  \begin{align} \label{def-eta} %
    \hks{\eta_\eps(t) %
    := - \frac 1{|\ln \eps|} \ln \sqrt{\frac{1}{\delta_\eps^2} (|t|- \eps)^2 + \eps^2} \quad\qquad
    \FU{|t| \in ({\eps}, {\eps} + \d_\eps \sqrt{1-\eps^2})}}
  \end{align}
  with $\eta_\eps(t) := 1$ for $|t| < \eps$ and $\eta_\eps(t) := 0$ for
  $|t| > {\eps} + \delta_\eps \sqrt{1-\eps^2}$.
\end{definition}
 We collect some estimates for the one-dimensional logarithmic profile $\eta_\eps$: 
\begin{lemma}[Estimates for 1-d profile] \label{lem-test1d} %
  For $\eps \in (0,1)$, the function $\eta_\eps$, defined in \eqref{def-eta},
  satisfies $\spt \eta_\eps \CUS (-2\d_\eps,2\d_\eps)$ and 
  \begin{enumerate}
  \item $\displaystyle \int_\R |(\frac d{dt})^{\frac 12} \eta_\eps|^2 \ dt \ %
    \leq \ \pi\int_{0}^{1} t |\eta'_\eps(t)|^2 \ dt \ 
    \leq \ \frac{\pi}{|\ln \eps|} + \frac C{|\ln \eps|^{\frac 32}}$,
  \item
    $\displaystyle \frac 1{\d_\eps} \int_\R |\eta_\eps|^2 \ dt 
    + \eps \d_\eps  \int_\R | \eta'_\eps|^2 \ dt \ \leq \ \frac{C}{|\ln \eps|^2}$.
  \end{enumerate}
\end{lemma}
\begin{proof}
  We calculate
  \begin{align} 
    \eta_\eps'(t) \ &= \ 
                      - \frac 1{|\ln \eps|} \frac {t - {\eps}}{(t-
                      {\eps})^2 + \eps^2 \d_\eps^2} \qquad %
    &\FU{t \in ({\eps}, {\eps} + {\d_\eps} \sqrt{1-\eps^2})}.
                                                      \label{eq-dereta} 
  \end{align}
  By the homogeneity of the integrals in (i) and (ii), we can replace
    $\d_\eps$ by $1$ in the following estimates.  Hence, estimate (ii)
  follows by the calculation
  \begin{align}
    \int_\R \eta_\eps^2 \ d t \ %
    &\leq \ 2 {\eps}+\frac {2}{|\ln \eps|^2} \int_0^1  \Big|\ln \frac 1{t}\Big|^2 \ dt \  
      \leq \ \frac{C}{|\ln \eps|^2},\\ %
    {\eps} \int_\R \big|\eta_\eps' \big|^2 \ d t  \ %
    &\leq  \ \frac {2{\eps}}{|\ln \eps|^2} \int_0^1 \frac {t^2}{(t^2 + \eps^2)^2} \ dt \ %
      \leq \ \frac{C}{|\ln \eps|^2}. %
  \end{align}
  In order to show (i), we use the formula
  \begin{align} \label{h12-ext} %
    \int_\R |(\frac d{dt})^{\frac 12} \eta_\eps|^2 \ dt \ %
    = \ \inf \Big \{ \int_{\R^2_+} |\nabla \OL{\eta_\eps}|^2 \ dx \ : \
    \OL{\eta_\eps} \in H^1(\R^2_+) \text{ with }
    \OL{\eta_\eps}(\cdot,0) = \eta_\eps \Big \}. %
  \end{align}
  Choosing the radially symmetric extension in \eqref{h12-ext}, (i) then follows
  from
  \begin{align}
    \int_\R |(\frac d{dt})^{\frac 12} \eta_\eps|^2 \ dt \ %
    &\leq \ \frac 12\cdot 2\pi\int_{{\eps}}^{1}|\eta'_\eps(t)|^2 t \ dt \ %
    \upref{eq-dereta}{=} \ \frac {\pi}{|\ln \eps|^2}\int_0^{1-{\eps}} \frac{t^2 (t+{\eps})}{(t^2 + \eps^2)^2} \  dt   \\%
    &\leq \ \frac {\pi}{|\ln \eps|^2} \int_0^{\frac 1\eps} \frac{t^3}{(t^2 + 1)^2}  \ dt %
      + \frac {\pi {\eps}}{\eps |\ln \eps|^2} \int_0^{\frac 1\eps} \frac{t^2}{(t^2 + 1)^2} \ dt,
  \end{align}
  noting that the first integral can be estimated by $|\ln \eps| + C$ and
    the second integral is estimated by a constant.
\end{proof}
To localize the energy we will use frequently the family of cut-off functions as follows:
\begin{definition}[Cut-off function $\chi_{\tau, \rho}$]
\label{def-cut-off}
For $\rho\in (0,1)$ and $\tau\in (0,\frac{\rho}{4})$, let $\chi_{\tau,\rho}\in C^\infty_c(B_{\rho-\tau})$,
    be a family of cut-off functions
    with
\begin{align} \label{def-chitau} %
  \begin{aligned}
      \chi_{\tau,\rho}=1 \text{ in } B_{\rho-2\tau}, \quad \chi_{\tau,\rho}=0 %
  &\text{  outside } B_{\rho-\tau}, \quad
  \chi_{\tau,\rho}\leq 1, \quad |\nabla \chi_{\tau,\rho}|\leq 2\tau^{-1} %
  \text{ in } B_\rho.
  \end{aligned}
\end{align}
\end{definition}

\subsection{Estimate for the leading order terms} \label{sus-self-inter}

In this section we give a lower bound for the self-interaction term $N_\eps$,
localized in $B_{\rho}$, as sketched in Step 3 of the proof in Section
\ref{sus-strategy}. As stated in Section \ref{sus-strategy}, the proof of the
lower bound for $N_\eps[\chi_{\tau,\rho}]$ is based on the following
\emph{duality estimate:}
\begin{align} \label{duality-local} %
  \left|\langle \chi_{\tau, \rho}\sig_\eps, \chi_{\tau, \rho} \Phi_{\eps,\rho}\rangle_{L^2}\right|\ %
  &\leq \ \left\||\nabla|^{-\frac 12}(\chi_{\tau, \rho}\sig_\eps)\right\|_{L^2}\left\|
  |\nabla|^{\frac 12} (\chi_{\tau, \rho} \Phi_{\eps,\rho}) \right\|_{L^2} \\
  &\upref{def-Neps}= \
    N_\eps[\chi_{\tau,\rho}]^{\frac 12} \Big(\frac{2|\ln\eps|}{\pi \lam}\Big)^{\frac 12}   \ %
  \left\| |\nabla|^{\frac 12} (\chi_{\tau, \rho} \Phi_{\eps,\rho}) \right\|_{L^2},
\end{align}
where the test function $\Phi_{\eps,\rho}$ and the cut-off function
  $\chi_{\tau,\rho}$ are given in Definition \ref{def-psieps} and Definition
  \ref{def-cut-off}, respectively.  In view of \eqref{duality-local}, to find a
lower bound for $N_\eps[\chi_{\tau,\rho}]$ it suffices to estimate
$|\langle \chi_{\tau, \rho}^2\sig_\eps, \Phi_{\eps,\rho} \rangle_{L^2}|$ from
below and $\| |\nabla|^{\frac 12} (\chi_{\tau, \rho} \Phi_{\eps,\rho})\|_{L^2}$
from above as stated in Section \ref{sus-strategy}. These estimates will be
given in the following two propositions. The next proposition is mainly
concerned with the upper bound for $\dot{H}^{\frac 12}$ norm of the test
function $\Phi_{\eps,\rho}$ with the sharp constant in the leading term, cf. (i)
below. Note that this also gives an upper bound on the capacity of the
separating curve $\gamma_{\eps,\rho}$ in $\R^3$. We also collect some further
bounds for $\Phi_{\eps,\rho}$, which will be used later to estimate terms which
are not leading order:
\begin{proposition}[Upper bound for duality estimate] \label{prp-RHS} %
  \hks{Let $\gam_{\eps_j,\rho}$ be a sequence of separating curves given in
    Lemma \ref{lem-separate2}. Let $\Phi_{\eps_j,\rho}$ and $\chi_{\tau,\rho}$
    be the test function and cut-off function given in Definition
    \ref{def-psieps} and Definition \ref{def-cut-off}, respectively. Assume that
    $\tau\geq 8\delta_{\eps_j}$ with $\d_\eps$ given in \eqref{def-deps}.}  Then
  for some universal constant $C>0$, we have
  \begin{enumerate}
  \item
    $\displaystyle \int_{\R^2} \big| |\nabla|^{\frac 12} (\chi_{\tau, \rho}
    \Phi_{\eps_j,\rho} )\big | ^2 \ dx \ 
    \leq \ \frac{\pi}{|\ln {\eps_j}|} \Big( 1+
    \frac{C}{|\ln{\eps_j}|^{\frac 12}\tau}\Big)\HH^1(\gamma_{\eps_j,\rho}),$
  \item
    $\displaystyle \int_{B_{\rho}} \chi_{\tau,\rho}^2 \Phi_{\eps_j,\rho}^2 \ dx \ %
    \leq \ \frac{C}{|\ln \eps_j|^{\frac 94}} \HH^1(\gam_{\eps_j,\rho})$,
  \item
   $\displaystyle \int_{B_{\rho}} \chi_{\tau,\rho}^2 \Phi_{\eps_j,\rho} \ dx \ 
    \leq \ \frac{C}{|\ln \eps_j|^{\frac 54}} \HH^1(\gam_{\eps_j,\rho})$,
  \item $\displaystyle \int_{B_{\rho}}\chi_{\tau,\rho}^2 |\nabla \Phi_{\eps_j,\rho}|^2 \ dx \ %
    \leq \ \frac{C} {\eps_j |\ln \eps_j|^{\frac 74}} \HH^1(\gam_{\eps_j,\rho})$,
  \item \label{it-Phi-sup} %
    \ed{Let $G_{\eps_j,\rho}^{(k)}$ be the same sets as in Lemma
      \ref{lem-separate2} \ref{sep2-4b}. Then} we have%
  \begin{align} \label{etas-1} %
    \NIL{\nabla \Phi_{\eps_j}}{G_{\eps_j,\rho}^{(k)}} \ %
    \leq \ \frac{C}{\dist(\p G_{\eps_j,\rho}^{(k)},\gamma_{\eps_j,\rho}) |\ln\eps_j|^{\frac 12}}.
  \end{align}
  \end{enumerate}
\end{proposition}
\begin{proof}
  For the simplification of the notation in the proof we write $\eps:=\eps_j$.

  \textit{(i): } We estimate the $\dot{H}^{\frac 12}$-norm, using the
  characterization
  \begin{align} \label{estext0} %
    \int_{\R^2} \big| |\nabla |^{\frac 12} u \big|^2 \ dx \ = \ \inf_{\OL
    u(x,0)=u(x)}\int_{\R^3_+}|\OL \nabla \OL u|^2 \ d \OL x,
    \end{align}
    where the infimum is taken over all $H^1$-extensions of $u$ to
    $\R^3_+:= \R^2 \times \R_+$ and $\OL x := (x,x_3) \in \R^3_+$ (see
    e.g. \cite[p.26]{LM72}).  Let $\bar d_{\gam_{\eps,\rho}} : \R^3_+\to \R_+$
    denote the distance to $\gam_{\eps,\rho}$ in $\R^3_+$.  We choose the
    extension $\OL \psi_\eps \in H^1(\R^3_+)$ of
    $\psi_\eps :=\chi_{\tau,\rho}\Phi_{\eps,\rho} = \chi_{\tau, \rho} (\eta_\eps \circ d_{\gam_{\eps,\rho}})$ by taking
  \begin{align}
    \OL \psi_\eps(\OL x) \ := \ \chi_{\tau, \rho}(x)\eta_\eps(\bar d_{\gam_{\eps,\rho}}(\OL x)) \qquad\qquad %
    \text{for $\OL x = (x,x_3)\in \R^3_+$},
  \end{align}
  where $\eta_\eps$ is defined in \eqref{def-eta}. By \eqref{estext0} this yields
  $\NNN{\psi_\eps}{\dot H^{\frac 12}(\R^2)} \leq \NTL{\OL \nabla \, \OL
    \psi_\eps}{\R^3_+}$.  We calculate
  \begin{align}\label{eq-ext}
    \int_{\R^3_+}|\OL \nabla \, \OL \psi_\eps|^2 \ d \OL x \ %
    = \ \int_{\R^3_+}\left|\chi_{\tau, \rho} \OL \nabla (\eta_\eps\circ \bar d_{\gam_{\eps,\rho}}) + (\eta_\eps\circ \bar d_{\gam_{\eps,\rho}}) \nabla \chi_{\tau, \rho}\right|^2 \ d \OL x.
  \end{align}
  For the estimate we first consider the term
  \begin{align}
    I_1 \ %
    &:= \ \int_{\R^3_+}\left|\chi_{\tau, \rho}\OL \nabla (\eta_\eps\circ \bar d_{\gam_{\eps,\rho}})\right|^2 \ d \OL x \ %
      = \ \int_{\R^3_+} \chi_{\tau, \rho}(x)^2 |\eta_\eps'(\bar d_{\gam_{\eps,\rho}}(\OL x))|^2 |\OL \nabla \, \OL d_{\gam_{\eps,\rho}}(\OL x)|^2  \ d \OL x.
  \end{align}
  Since $|\OL \nabla \bar d_{\gam_{\eps,\rho}}|=1$ a.e., by the coarea formula,
  since $\spt \eta_\eps' \SUS (-2\d_\eps,2\d_\eps)$ we then have
  \begin{align} \label{def-I1-here} %
    I_1 \ %
    &= \ \int_{0}^{2\d_\eps} |\eta'_\eps(t)|^2\Big(\int_{\{\bar d_{\gam_{\eps,\rho}} = t\}} \chi_{\tau, \rho}^2(x) \ d\HH^2(\OL x) \Big) \ dt \ %
    \leq \ \int_{0}^{2\d_\eps} |\eta'_\eps(t)|^2 \ \HH^2(\Gam_{t,\rho}) \ dt,
  \end{align}  
  where
  $\Gamma_{t,\rho}:=\{\bar d_{\gam_{\eps,\rho}} = t\}\cap (B_{\rho-\tau}\times
  \R_+)$. To estimate $\HH^2(\Gam_{t,\rho})$, we introduce
    $g : \R_+^3 \to \R$, $g(\bar x):=x_3$ and consider the slices
  $\Gam_{t,\rho} \cap g^{-1}(s)$ for ${s} \in (0,t)$. We note that
  $\OL x = (x, {s}) \in \Gam_{t,\rho}$ if and only if
  $d_{\gam_{\eps,\rho}}(x)^2 + {s}^2 = t^2$, i.e.
  $x \in d_{\gam_{\eps,\rho}}^{-1}(\sqrt{t^2 - {{s}}^2})$. Since
  $B_{\rho - \tau} \SUS B_{\rho - 8\d_\eps}$ due to $\tau\geq 8\delta_\eps$, this implies
  $\HH^1(\Gam_{t,\rho} \cap g^{-1}(s)) \leq
  \HH^1(d_{\gam_{\eps,\rho}}^{-1}(\sqrt{t^2 - {s}^2}) \cap B_{\rho-8
    \d_\eps})$ for ${s} \in [0,t]$ and $0$ else. An application of Lemma \ref{lem-separate2}
  \ref{sep2-0a} then yields
    \begin{align} \label{H1-est-loc}
      \HH^1(\Gam_{t,\rho} \cap g^{-1}(s) ) \ %
      \leq \ 2\HH^1(\gam_{\eps,\rho}) \qquad \text{for $t \in (0,2\d_\eps)$, ${s} \in (0,t)$.}
    \end{align}
    By the coarea formula, applied to the level set of $g$, we hence get
    \begin{align}  \label{HH2-graph} %
      \HH^2(\Gam_{t,\rho}) \ %
      &= \ \int_0^t \HH^1(\Gam_{t,\rho} \cap g^{-1}(s)) \frac{1}{|\nabla_{\Gamma_{t,\rho}} g|} \ d {s} \\ %
      &= \ \int_0^t \HH^1(\Gam_{t,\rho} \cap g^{-1}(s)) \frac{t}{\sqrt{t^2 - {s}^2}} \ d {s} \ %
      \lupref{H1-est-loc}\leq \ \pi t \HH^1(\gam_{\eps,\rho}),
    \end{align}
    where $\nabla_{\Gamma_{t,\rho}} g =\frac{\sqrt{t^2 - {s}^2}}{t}$ is the
    projection of the full gradient $\nabla g$ onto the tangent space of
    $\Gamma_{t,\rho}$.  Inserting estimate \eqref{HH2-graph} into
    \eqref{def-I1-here} and by an application of Lemma \ref{lem-test1d}(i) we
    arrive at
  \begin{align} \label{I1-eest} %
    I_1 \ &\lupref{HH2-graph}\leq \ \pi \HH^1(\gam_{\eps,\rho})\int_{0}^{2\d_\eps} t |\eta'_\eps(t)|^2 \ dt  \ %
            \leq \ \Big(\frac{\pi}{|\ln \eps|}+\frac{C}{|\ln \eps|^{\frac{3}{2}}}\Big) \HH^1(\gam_{\eps,\rho}).
  \end{align}
  To continue with the estimate of \eqref{eq-ext}, we similarly apply
  the coarea formula to get
  \begin{equation}
    I_2  \ := \ \int_{\R^3_+}(\eta_\eps\circ \bar d_{\gam_{\eps,\rho}})^2|\nabla \chi_{\tau, \rho}|^2 \ \ d \OL x \ 
    = \ \int_{0}^{2\d_\eps}|\eta_\eps(t)|^2 \Big(\int_{\{\bar d_{\gam_{\eps,\rho}}=t\}} |\nabla \chi_{\tau, \rho}|^2 d\HH^2\Big) \ dt.
  \end{equation}
  Since $|\nabla\chi_{\tau, \rho}| \leq 2\tau^{-1}$, by \eqref{HH2-graph} and
  by Lemma
  \ref{lem-test1d}(ii), as well as $8\delta_\eps \leq \tau$ and since
    $\delta_\eps = |\ln \eps|^{-\frac 14}$, we further get the bound
  \begin{equation} \label{I2-eest} %
    I_2 \ %
    \leq \  \ C \HH^1(\gamma_{\eps,\rho}) \frac {1}{\tau^2} \int_{0}^{2\d_\eps} t|\eta_\eps(t)|^2 \ dt \ %
    \leq \ \frac{C \d_\eps^2 \HH^1(\gamma_{\eps,\rho})}{|\ln\eps|^{2}\tau^2} \ %
    \leq \ \frac{C\HH^1(\gamma_{\eps,\rho})}{|\ln\eps|^{\frac{9}{4}}\tau}.
  \end{equation}
  By Cauchy-Schwarz and the estimates \eqref{I1-eest} and \eqref{I2-eest}, we
    also have
    \begin{align} \label{mix-eest} %
      \Big|\int_{\R^3_+}\chi_{\tau, \rho}(\eta_\eps \circ \bar
      d_{\gam_{\eps,\rho}})\OL \nabla(\eta_\eps\circ \bar
      d_{\gam_{\eps,\rho}})\cdot\nabla\chi_{\tau, \rho} \ d\OL x \Big| \ %
      \leq \ \frac{C\HH^1(\gamma_{\eps,\rho})}{|\ln\eps|^{\frac{7}{4}}\tau}.
  \end{align}
  Estimates \eqref{I1-eest}, \eqref{I2-eest} and \eqref{mix-eest} together
  yield the desired upper bound.

  \medskip

  \textit{(ii)--(iv):} We only give the estimate for (ii), since the estimates
  for (iii) and (iv) follow similarly. By the coarea formula, since
  $|\nabla d_{\gam_{\eps,\rho}}| = 1$ a.e., since $\tau \geq 8 \d_\eps$,
  $\spt \Phi_{\eps,\rho} \Subset d_{\gam_{\eps,\rho}}^{-1}([0,2\delta_\eps))$
  and by Lemma \ref{lem-separate2}\ref{sep2-0a}, we have
  \begin{align}
    \int_{B_{\rho}} \chi_{\tau,\rho}^2 |\Phi_{\eps,\rho}|^2 \ dx \ %
    &\leq \ \int_{0}^{2\d_\eps}\eta_\eps^2(t) \ \HH^1\big(d_{\gam_{\eps,\rho}}^{-1}(t)\cap B_{\rho-8\d_\eps}\big) \ dt \\ %
    &\leq \ 2\HH^1(\gam_{\eps,\rho})\int_{\R} \eta_\eps^2(t) \  \ dt \ %
    \leq  \frac{C}{|\ln \eps|^{\frac 94}} \HH^1(\gam_{\eps,\rho}),
  \end{align}
  where we used  Lemma \ref{lem-test1d}(ii) for the last estimate.

  \medskip
  
  \textit{\ref{it-Phi-sup}:} %
  We first note that by \eqref{def-Phieps} we have
  $|\nabla \Phi_{\eps\hks{,\rho}}(x)| \leq |\eta'_\eps
  (d_{\gam_{\eps,\rho}}(x))|$.  From the explicit expression for $\eta'_\eps$ in
  \eqref{eq-dereta} \hks{and the fact that
    $d_{\gam_{\eps,\rho}}(x)\leq d_{\eps,\rho}^{(k)} :=\dist(\p
      G_{\eps,\rho}^{(k)},\gamma_{\eps,\rho})$ for $x\in G_{\eps,\rho}^{(k)}$,
    we then get
  \begin{align} \label{thises} %
    |\nabla \Phi_{\eps\hks{,\rho}}(x)| \ %
    \lupref{def-Phieps}\leq \ |\eta'_\eps(d_{\gamma_{\eps,\rho}}(x))| \ %
    \leq \ 
    \hks{\frac 1{|\ln\eps|} \frac 1{\max\{d_{\eps,\rho}^{(k)}-{\eps},
    \eps \delta_\eps\}}} \qquad %
    \text{for $x\in G_{\eps,\rho}^{(k)}$}.
  \end{align}
  To get (v), it is hence enough } to show that \hks{there is a universal constant
    $C>0$ such that}
  \begin{align} \label{v-toshow}
      d \ \leq \ \hks{C}|\ln \eps|^{\frac 12} \max\{d -
    {\eps},\eps\hks{\delta_\eps}\} \qquad %
    \text{ for all $d > 0$ and all
  $\eps \in (0,\frac 14)$.} 
  \end{align}
  \hks{Indeed, when $d-{\eps} \geq \eps\delta_\eps$, it follows from
    $d\mapsto \frac{d}{d-{\eps}}$ is monotone decreasing that
    $\frac{d}{d-{\eps}}\leq
    \frac{{\eps}+\eps\delta_\eps}{\eps\delta_\eps}=|\ln\eps|^{\frac
      14}+1\leq 2|\ln\eps|^{\frac 12}$; when $d-{\eps}<\eps\delta_\eps$, one
    has
    $d\leq {\eps} + \eps\delta_\eps\leq 2|\ln\eps|^{\frac 12}
    \eps\delta_\eps$. Together, this yields \eqref{v-toshow}.}
\end{proof}
We next give the estimate for the term on the left hand side of the duality
estimate \eqref{duality-local}. We show that our test function asymptotically
captures the total charge of the transition layer via an application of the
divergence theorem: 
\begin{proposition}[Lower bound for duality estimate] \label{prp-LHS} %
  \hks{Let $\gam_{\eps_j,\rho}$ be a sequence of separating curves given in
    Lemma \ref{lem-separate2}. Let $\Phi_{\eps_j,\rho}$ and $\chi_{\tau,\rho}$
    be the test function and cut-off function given in Definition
    \ref{def-psieps} and Definition \ref{def-cut-off}, respectively. Assume that
    $\tau\geq 8\delta_{\eps_j}$ with $\d_\eps$ given in \eqref{def-deps}.}
 Then there is
  $C = C(\NNN{\nabla \cdot M}{L^\infty}) > 0$ such that for $\eps_j$ sufficiently
  small, we have
  \begin{align} \label{est-LHS} %
    \Big|\int_{\R^2} \chi_{\tau, \rho}^2\Phi_{\eps_j,\rho} \sig_{\eps_j}  \ dx %
    +2\int_{\gam_{\eps_j,\rho} }\chi_{\tau, \rho}^2 (n_{\eps_j} \cdot e_1)  \ d\HH^1 \Big| \ %
    \leq \ \frac{C \ed{(K+1)}}{|\ln {\eps_j}|^{\frac 12}} \Big(1 + \frac{1}{\tau|\ln\eps_j|^{\frac 12}} \Big),
  \end{align}
  where $n_{\eps_j}$ is the unit outer normal of
  $\Ome^{(0)}_{\eps_j, \rho}$ along $\gamma_{\eps_j,\rho}$ and $K$ is
  given in \eqref{meps-bound}.
\end{proposition}
\begin{proof}
  We write $\eps := \eps_j$  and $n_{\eps_j,1} := n_{\eps_j} \cdot e_1$.   We first note that for the estimate we
  can replace $\sig_\eps = \nabla\cdot (m_\eps - M)$ on the left hand side
    of \eqref{est-LHS} by $\nabla\cdot m_\eps$, since by Proposition
  \ref{prp-RHS}(iii) the integral involving the background magnetization $M$ is
  bounded by $CK|\ln\eps|^{-\frac{5}{4}} \NNN{\nabla \cdot M}{L^\infty}$.
  Integrating by parts yields
  \begin{align} \label{pi-rhs} %
    \int_{\R^2} \chi_{\tau, \rho}^2\Phi_{\eps,\rho}\nabla\cdot m_\eps \ dx \ %
    &= \ -\int_{\R^2} \nabla(\chi_{\tau, \rho}^2\Phi_{\eps,\rho})\cdot m_\eps \ dx. %
  \end{align}
  We use the decomposition
    $\R^2 = \Ome_{\eps,\rho}^{(0)} \cup (\Ome_{\eps,\rho}^{(0)})^c$ and claim
  that
  \begin{align}
    \Big|\int_{\Ome_{\eps,\rho}^{(0)}}\nabla(\chi_{\tau, \rho}^2\Phi_{\eps,\rho})\cdot m_\eps \ dx %
    -\int_{\gam_{\eps,\rho}} \chi_{\tau, \rho}^2 n_{\eps,1} \ d\HH^1\Big| \ %
      \leq \ \frac{C \ed{(K+1)}}{|\ln {\eps}|^{\frac 12}} \Big(1 + \frac{1}{\tau|\ln\eps|^{\frac 12}} \Big) \label{eq-div-E}
  \end{align}
  and that the same estimate holds with $\Ome_{\eps,\rho}^{(0)}$ replaced
    by its complement. In the following, we give the argument for
    \eqref{eq-div-E}, noting that the estimate  for the integral over
  $(\Ome_{\eps,\rho}^{(0)})^c$ can be shown with an analogous argument by
  replacing $e_1$ with $-e_1$ in the following proof.

  \medskip
  
  Since $\Phi_{\eps,\rho}=1$ on
  $\gam_{\eps,\rho} = \p \Ome_{\eps,\rho}^{(0)} \cap B_\rho$ and by the
  divergence theorem we have
  \begin{equation}
    \int_{\Ome_{\eps,\rho}^{(0)}}\nabla(\chi_{\tau, \rho}^2\Phi_{\eps,\rho})\cdot e_1 \ dx \  %
    =  \ \int_{\Ome_{\eps,\rho}^{(0)}}\nabla \cdot (\chi_{\tau, \rho}^2\Phi_{\eps,\rho} e_1) \ dx \   %
    = \ \int_{\gam_{\eps,\rho}} \chi_{\tau, \rho}^2 n_{\eps,1} \ d\HH^1.
  \end{equation}
  To prove \eqref{eq-div-E}, it hence remains to show that
  \begin{align} \label{III-est} %
    \Big|\int_{\Ome_{\eps,\rho}^{(0)}} \nabla (\chi_{\tau, \rho}^2\Phi_{\eps,\rho})
    \cdot (m_\eps-e_1) \ dx \Big| \ %
    \leq \ \frac{C \ed{(K+1)}}{|\ln {\eps}|^{\frac 12}} \Big(1 +
    \frac{1}{\tau|\ln\eps|^{\frac 12}} \Big).
  \end{align}
  To show \eqref{III-est}, we further decompose the domain of integration
    and write
  \begin{align}
    \hspace{6ex} & \hspace{-6ex} %
                   \Big|\int_{\Ome_{\eps,\rho}^{(0)}} \nabla (\chi_{\tau, \rho}^2\Phi_{\eps,\rho})\cdot (m_\eps-e_1) \ dx \Big| \ \\
                 &\leq \ \int_{\Ome_{\eps,\rho}^{(0)}\BS \Ome_{\eps,\rho}^{(1)}}\left|\nabla (\chi_{\tau, \rho}^2\Phi_{\eps,\rho})\cdot (m_\eps-e_1)\right| \ dx %
                   + 2 \int_{\Ome_{\eps,\rho}^{(0)} \cap \Ome_{\eps,\rho}^{(1)}}\left|\nabla (\chi_{\tau, \rho}^2\Phi_{\eps,\rho})\right| \ dx \\
    &=: \ R_1 + R_2.
  \end{align}
  We conclude the proof by estimating $R_1$ and $R_2$ separately.

  \medskip

  \textit{Estimate for $R_1$:} By Lemma
  \ref{lem-separate2}\ref{sep2-dazu}, we have $u_\eps + 1> \d_\eps$ in
  $\Ome_{\eps,\rho}^{(0)} \BS \Ome_{\eps,\rho}^{(1)}$. Together with
  $|m_\eps|^2= u_\eps^2 + v_\eps^2 = 1$, we hence get
  \begin{align} \label{R1-0} %
    \int_{\Ome_{\eps,\rho}^{(0)} \BS
    \Ome_{\eps,\rho}^{(1)}} |m_\eps-e_1|^2 \ dx \ 
    &= \ \int_{\Ome_{\eps,\rho}^{(0)} \BS
      \Ome_{\eps,\rho}^{(1)}} \frac{2v_\eps^2}{1+u_\eps} \ dx \ %
      \leq \ \frac{2}{\d_\eps} \int_{\Ome_{\eps,\rho}^{(0)} \BS \Ome_{\eps,\rho}^{(1)}} v_\eps^2 \ dx \\
    &\upref{EE}\leq \ \frac{4 \eps}{\d_\eps} E_\eps[m_\eps] \ \upref{def-deps}\leq \ 4K \eps |\ln
      \eps|^{\frac 14}.
  \end{align}
  On the other hand, since $|\chi_{\tau,\rho}| \leq 1$,
  $|\nabla\chi_{\tau, \rho}|\leq 2\tau^{-1}$, by an application of Proposition
  \ref{prp-RHS}(ii)--(iii) and since $\HH^1(\gamma_{\eps,\rho})\leq K$, we obtain
  \begin{align}
    \int_{\R^2} |\nabla(\chi_{\tau, \rho}^2\Phi_{\eps,\rho})|^2 \ dx \ %
    &\leq \ {\hks{2}} \int_{\R^2} \chi_{\tau, \rho}^{ 2} |\nabla\Phi_{\eps,\rho}|^2 \ dx %
    + \frac {\hks{8}}{\tau^2}  \int_{\R^2} \chi_{\tau, \rho}^{2} \Phi_{\eps,\rho}^2 \ dx\\
    &\leq \ C K \Big( \frac 1{\eps |\ln\eps|^{\frac 74} } + \frac{1}{\tau^2 |\ln\eps|^{\frac 94}} \Big). \label{R1-2}
  \end{align}
  Using Cauchy-Schwarz together with inequalities \eqref{R1-0}--\eqref{R1-2}
  then yields
  \begin{align}
    R_1 \ %
    \leq \  C K \Big(\frac{1}{|\ln\eps|^{\frac{3}{4}}} + \frac{\eps^{\frac 12}}{\tau|\ln\eps|} \Big) \ %
    \leq \  \frac{C K}{|\ln {\eps}|^{\frac 12}} \Big(1 + \frac{1}{\tau|\ln\eps|^{\frac 12}} \Big). \ %
  \end{align}

  \medskip

  \textit{Estimate for $R_2$:} Using $|\chi_{\tau, \rho}|\leq 1$,
  $|\nabla\chi_{\tau, \rho}|\leq 2\tau^{-1}$ and Cauchy-Schwarz, we have
  \begin{align}
    R_2 \ &\leq \ 2 \int_{\Ome_{\eps,\rho}^{(1)}}\chi_{\tau,\rho}^2 |\nabla \Phi_{\eps,\rho}| \ dx + \frac 4\tau |\Ome_{\eps,\rho}^{(1)}|^{\frac 12} \Big(\int_{\Ome_{\eps,\rho}^{(1)}} \chi_{\tau, \rho}^2 \Phi_{\eps,\rho}^2 \  dx \Big)^{\frac 12}\\
          &\leq \ 2 \int_{\Ome_{\eps,\rho}^{(1)} \cap B_{\rho-4 \d_\eps}} \chi_{\tau,\rho}^2|\nabla \Phi_{\eps,\rho}| \ dx + \frac {C\ed{(K+1)}}{\tau |\ln \eps|^{\frac 54}}.%
  \end{align}
  For the second estimate, we have applied Lemma
  \ref{lem-separate2}\ref{sep2-auchnoch} and Proposition \ref{prp-RHS}(ii)
  \hks{and we have used that $\tau \geq 8 \d_\eps$ by assumption.}  To estimate
  the first term we recall from Lemma \ref{lem-separate2}\ref{sep2-4b} that
  $\Ome_{\eps,\rho}^{(1)}\cap B_{\rho-4\delta_\eps} = \bigcup_{k}
  G_{\eps,\rho}^{(k)}$ with
  $\HH^1(\p G_{\eps,\rho}^{(k)}) \leq 2\pi d_{\eps,\rho}^{(k)}$ for
  $d_{\eps,\rho}^{(k)}:=\dist(\p G_{\eps,\rho}^{(k)}, \gam_{\eps,\rho})$.  By
  the isoperimetric inequality and Lemma \ref{lem-separate2}\ref{sep2-4b} we
  also have
  $|G_{\eps,\rho}^{(k)}|\leq C \HH^1(\p G_{\eps,\rho}^{(k)})^2 \leq C
  d_{\eps,\rho}^{(k)} \HH^1(\p G_{\eps,\rho}^{(k)})$.  Together with the bound
  on $\|\nabla \Phi_\eps\|_{L^\infty(G^{(k)}_{\eps,\rho})}$ in \eqref{etas-1}
  we then get
  \begin{align}
    \int_{\Ome_{\eps,\rho}^{(1)}} \chi_{\tau,\rho}^2|\nabla \Phi_{\eps,\rho}| \ dx \ %
    &\leq \  \sum_{k} \frac{C |G_{\eps,\rho}^{(k)}|}{d_{\eps,\rho}^{(k)} |\ln\eps|^{\frac 12}}  \ %
    \leq \  \frac{C}{|\ln\eps|^{\frac 12}} \sum_{k} \HH^1(\p G_{\eps,\rho}^{(k)})  \ %
      \leq \ \frac{C\ed{(K+1)}}{|\ln\eps|^{\frac 12}}.
  \end{align}
  For the last inequality we have used that
  $\sum_{k}\HH^1(\p G_{\eps,\rho}^{(k)})= \HH^1(\p (\Ome_{\eps,\rho}^{(1)}\cap
  B_{\rho-4\delta_\eps}))\leq C \ed{(K+1)}$ by Lemma
  \ref{lem-separate2} \ref{sep2-4a}. The above estimates
  together yield the desired bound for $R_2$.
\end{proof}
From Propositions \ref{prp-RHS}--\ref{prp-LHS} we obtain the lower bound for the
leading order terms: 
\begin{proposition}[Lower bound for leading order terms] \label{prp-LOT} %
  Consider a sequence $m_\eps \to m$ with $\eps\to 0$ which satisfies
  \eqref{ass-meps} and let $\Ome_0$, $\SS_m$, $n$ be given by
  \eqref{def-Ome0}. Let $\widehat x\in \SS_m$ and let $\hat \rho\in (0,1)$ be
  the constant from Lemma \ref{lem-separate2}. Then for any
  $\rho\in (0,\hat\rho)$, there exists a subsequence $\eps_j\to 0$ and a
  sequence $\rho_j\leq \rho$, $\rho_j\to \rho$ such that with $\tau_j:=8\delta_{\eps_j}$ and
  $B_\rho := B_\rho(\widehat x)$ we have
 \begin{align}\label{eq-leading}  %
   \liminf_{j\to  \infty} \left(\NN{Du_{\eps_j}}(B_\rho)+N_{\eps_j}[\chi_{\tau_j,\rho_j}]\right) \ %
   &\geq \ 2f(\sqrt{\lambda} \left|\OL n_{\rho} \cdot e_1\right|)\HH^1(\SS_m \cap B_\rho),
 \end{align}
 where $\chi_{\tau,\rho}$ is the cut-off function given in Definition
 \ref{def-cut-off}, $f$ is given in \eqref{def-f} and where
  \begin{align}
      \OL n_{\rho}\  := \ \dashint_{\SS_m\cap B_{\rho}} n \ d\HH^1.
  \end{align}
\end{proposition}
\begin{proof}
  Let $\gamma_{\eps_j,\rho}$ be a sequence of separating curves such that
    the assertions of Lemma \ref{lem-separate2} hold and let $\alp\geq 1$ be
    given \eqref{def-alp2}. Since $\HH^1(\gamma_{\eps_j,\rho})\leq K<\infty$ by
  Lemma \ref{lem-separate2}\ref{sep2-0a}, there is a subsequence of $\eps_j$
  (still denoted by $\eps_j$) and a sequence $\rho_j\to \rho$, such that
\begin{align}\label{eq-nonconcentration}
  \HH^1(\gamma_{\eps_j,\rho}\cap (B_{\rho_j}\BS B_{\rho_j-2\tau_j}))\to  0 %
  \qquad \text{ as } j\to  \infty.
\end{align}
In view of the duality estimate \eqref{duality-local}, we have
  \begin{align} \label{duality-local-2} %
    N_{\eps_j}[\chi_{\tau_j,\rho_j}] \ %
    \lupref{duality-local}\geq \
        \frac{\pi \lam}{2|\ln \eps_j|} \frac{\left|\langle \chi_{\tau_j, \rho_j}\sigma_{\eps_j}, \chi_{\tau_j, \rho_j}
    \Phi_{\eps_j,\rho_j}\rangle_{L^2}\right|^2}{\NNN{|\nabla|^{\frac 12} (\chi_{\tau_j, \rho_j} \Phi_{\eps_j,\rho_j})}{L^2}^2}.
  \end{align}
  An application of Proposition \ref{prp-RHS}(i) and Proposition \ref{prp-LHS} leads to
  \begin{align} \label{est-Neps-final-almost} %
    \liminf_{j\to  \infty} N_{\eps_j}[\chi_{\tau_j,\rho_j}] \ %
    &\geq \  \Big(\frac {\lambda}{\displaystyle  2 \limsup_{\eps \to 0} \HH^1(\gam_{\eps_j,\rho})}\Big) \displaystyle \liminf_{\eps \to 0} \Big( 2\int_{\gam_{\eps_j,\rho} }\chi_{\tau_j, \rho_j}^2 (n_{\eps_j} \cdot e_1)  \ d\HH^1 \Big)^2.
  \end{align}
  Using Lemma \ref{lem-separate2}\ref{sep2-3}--\ref{sep2-2} together with
  \eqref{eq-nonconcentration} then yields
  \begin{align} \label{est-Neps-final} %
    \liminf_{j\to  \infty} N_{\eps_j}[\chi_{\tau_j,\rho_j}] \ &\geq \ \frac{2\lambda}{\alp} \left|\OL n_{\rho}\cdot e_1\right|^2 \HH^1(\SS_m \cap B_\rho).
  \end{align}
  Combining the above estimate with \eqref{def-alp2} we arrive at
  \begin{align}
    \liminf_{j\to  \infty} \frac{\NN{Du_{\eps_j}}(B_\rho)+N_{\eps_j}[\chi_{\tau_j,\rho_j}]}{\HH^1(\SS_m \cap B_\rho)} \ %
    \geq  \ 2  \big(\alp+\frac{\lam}{\alp}\left|\OL n_{\rho}\cdot e_1\right|^2\big) \ %
    \upref{def-f}\geq\ 2 f(\sqrt{\lam} \big|\OL n_{\rho}\cdot e_1\big|).
  \end{align}
\end{proof}

\subsection{Localization argument and conclusion of proof} \label{sus-localize} %

In this section, we give the proof of the \ed{liminf--inequality}
\eqref{low-bd}.  Next to the estimate of the leading order terms from
Proposition \ref{prp-LOT} in the last section, this requires a localization
argument and the estimate of the interaction energy between different
sets. We first show that the interaction energy between a ball
  $B_\rho(\hat x)$ and its complement is negligible as $\eps \to 0$ for almost
  every choice of radius in the following sense:
\begin{lemma} \label{lem-inter} %
  Consider a sequence $m_\eps \to m$ which satisfies \eqref{ass-meps} and
    \eqref{meps-bound} for some sequence $\eps \to 0$.
    Then for any $\widehat x \in Q_\ell$ there is a subsequence $\eps_j\to 0$
  and $\mathcal N\subset (0,1)$ with $|\mathcal N|=0$ such that for any
  $\rho\in (0,\min\{1,\frac{\ell}{2}\})\BS \mathcal N$ we have
  \begin{align}\label{eq-inter}
    \lim_{\eps_j\to  0}\frac{1}{|\ln\eps_j|}\int_{\R^2\BS B_\rho(\widehat x)}\int_{B_\rho(\widehat x)} \frac{|m_{\eps_j}(x)-m_{\eps_j}(y)|^2}{|x-y|^3}\ dxdy \ %
    = \ 0.
  \end{align}
\end{lemma}
\begin{proof}
  The proof uses similar ideas as \cite{AB98}. By the change of variable $h=y-x$
  and with the notation $B_\rho := B_\rho(\widehat x)$ and
  $D_{\rho,h}:=\{x\in B_\rho: x+h \nin B_\rho \}$ we have
  \begin{align}
    \int_{\R^2\BS B_\rho}\int_{B_\rho} \frac{|m_{\eps}(x)-m_{\eps}(y)|^2}{|x-y|^3}\ dxdy \ %
    = \ \int_{\R^2}\int_{D_{\rho,h}} \frac{|m_\eps(x)-m_\eps(x+h)|^2}{|h|^3} \ dxdh.
  \end{align}
  Since
  $\NTL{m_\eps(\cdot+h)-m_\eps}{Q_\ell}^2 \leq \NTL{\nabla m_\eps}{Q_\ell}^2 |h|^2
  \leq \frac K{\eps} |h|^2$ and $|D_{\rho,h}| \leq \pi\rho^2 \leq \pi$, this
  implies
  \begin{align}
    \int_{|h| < \eps|\ln\eps|^{\frac{1}{4}}}\int_{D_{\rho,h}}\frac{|m_\eps(x)-m_\eps(x+h)|^2}{|h|^3} \ dxdh \ %
    &\leq \ \int_{|h| < \eps|\ln\eps|^{\frac{1}{4}}}\frac{CK}{\eps |h|}\ dh \ %
      =  \ CK |\ln\eps|^{\frac{1}{4}}.
  \end{align}
  For $|h|\geq |\ln\eps|^{-\frac{1}{4}}$, we use $|m_\eps|\leq 1$ to get
  \begin{align}
    \hspace{6ex} & \hspace{-6ex} %
                   \int_{ |\ln\eps|^{-\frac{1}{4}} \leq |h|}\int_{D_{\rho,h}}\frac{|m_\eps(x)-m_\eps(x+h)|^2}{|h|^3}\ dx dh \ %
                   \leq \ \int_{ |\ln\eps|^{-\frac{1}{4}} \leq |h|} \frac{C}{|h|^3}\ dh \ %
                   \leq \ C |\ln\eps|^{\frac{1}{4}}.
  \end{align}
  It hence remains to give an estimate for the integral
  \begin{align}
    I_{\eps,\rho}(m_\eps) \ %
    := \ \frac{1}{|\ln\eps|}\int_{\eps|\ln\eps|^{\frac 14} \leq |h|<|\ln\eps|^{-\frac 14}}\int_{D_{\rho,h}} \frac{|m_\eps(x)-m_\eps(x+h)|^2}{|h|^3}\ dxdh.
  \end{align} 
  We first note that for each $x\in D_{\rho,h}$, the line segment $[x,x+h]$
  intersects $\p B_\rho$ at a unique point $\sigma = x+th$ with $t\in
  [0,1]$. Thus one can apply the change of variables
  $D_{\rho,h} \ni x\mapsto (\sigma, t)\in \p B_\rho \times [0,1]$, and the
  Jacobian of the map $(\sig,\ed{t}) \mapsto x$ is bounded from above by
  $|h|$. Using Fubini, the triangle inequality and $|m_\eps|\leq 1$, we hence
  get
  \begin{align}
    I_{\eps,\rho}(m_\eps) \ %
    &\leq \  %
    \frac{1}{|\ln\eps|}\int_{0}^{1}
    \int_{\eps|\ln\eps|^{\frac 14} \leq |h|<|\ln\eps|^{-\frac 14}}\int_{\p B_\rho} 
 \frac{|m_\eps(\sigma-th)-m(\sig)|}{|h|^2} \ d\sigma dh dt  \\
&\qquad +    \frac{1}{|\ln\eps|}\int_{0}^{1}
    \int_{\eps|\ln\eps|^{\frac 14} \leq |h|<|\ln\eps|^{-\frac 14}} \int_{\p B_\rho}
\frac{|m_\eps(\sigma+(1-t)h)-m(\sig)|}{|h|^2} \ d\sigma dh dt \\ %
    &=: \ X_1 + X_2.
  \end{align}
  By symmetry, it suffices to give the estimate for $X_1$.  With
  $h =: t^{-1} |\ln\eps|^{-\frac{1}{4}} y$ we get
  \begin{align}
    X_1 \ %
    &= \ \frac 1{|\ln\eps|} \int_0^1 \int_{t \eps |\ln\eps|^{\frac 12} \leq |y| \leq t}\int_{\p B_\rho}
      \frac{|m_\eps(\sigma-\d_\eps y)-m(\sig)|}{|y|^2} \ d\sigma dy dt %
    \ = \   \int_{0}^1 \Psi_{\eps,t}(\rho)\ dt,
  \end{align}
  where
  \begin{align}
    \Psi_{\eps,t}(\rho) \ &:= \ \frac 1{|\ln \eps|}  \int_{t \eps |\ln\eps|^{\frac 12} \leq |y| \leq t} \int_{\p B_\rho}\frac 1{|y|^2} \big| m_\eps(\sigma-\delta_\eps y)-m(\sig) \big| \ d\sigma dy.
  \end{align}
  It remains to show that $\lim_{\eps_j\to 0} I_{\eps_j,\rho}(m_{\eps_j}) = 0$
  for a subsequence $\eps_j\to 0$ and a.e.  $\rho\in (0,1)$.  For this, is
  enough to show that along a subsequence $\eps_j\to 0$, we have
  \begin{align} \label{psi-con-toshow} %
    \sup_{t \in (0,1)}\Psi_{\eps_j, t}(\rho) \ \to \ 0 \qquad \text{ for a.e. } \rho\in (0,1).
  \end{align}
  To see \eqref{psi-con-toshow}, for $t\in (0,1)$ fixed we integrate in
  $\rho$. By the coarea formula and by the triangle inequality we then have
  \begin{align}
    \hspace{6ex} & \hspace{-6ex} %
                   \int_0^1  \Psi_{\eps,t}(\rho)\ d\rho \ %
    \leq \ \frac 1{|\ln \eps|} \int_{t \eps |\ln\eps|^{\frac 12} \leq |y| \leq t}\int_{B_1}\frac{|m_\eps(x-\delta_\eps y)-m(x)|}{|y|^2}\ dx dy \\
    &\leq \ \Big(\frac 1{|\ln \eps|}  \int_{\eps |\ln\eps|^{\frac 12} \leq |y| \leq 1} \frac{1}{|y|^2} \ dy \Big) \sup_{|y| \in (0,1)} \Big(\NP{m_\eps-m}{1}+\NP{m(\cdot-\delta_\eps y)-m}{1} \Big).
  \end{align}
  The integral on the right hand
  side above is uniformly bounded for $\eps \in (0,1)$ as a straightforward calculation shows.  Since $m_\eps \to m$
  in $L^1$ and since
  $\NP{m(\cdot-\delta_\eps y)-m}{1}\leq \delta_\eps|y|\|Dm\|(Q_\ell) \leq CK
    |y|$ by \cite[Lemma 3.24]{Ambrosio-Book}, we conclude that
  \begin{align}
    \lim_{\eps\to  0} \sup_{t\in (0,1)}  \int_{0}^1  \Psi_{\eps,t}(\rho)\ d\rho \ = \ 0.
  \end{align}
  It follows that \eqref{psi-con-toshow} holds along a subsequence
  $\eps_j\to 0$ which completes the proof.
\end{proof}
\begin{proposition}[Liminf--inequality]\label{prp-localization}
  For any sequence $m_\eps \in \AA$ with $m_\eps \to m = (u,0) \in \AA_0$ in
  $L^1(Q_\ell)$ for some sequence $\eps \to 0$, we have
  $\liminf_{\eps \to 0} E_\eps[m_\eps] \ \geq \ E_0[m]$.
\end{proposition}
\begin{proof}
  By the argument following equation \eqref{ass-meps}, we can assume that
    $m_\eps$ satisfies \eqref{ass-meps} and \eqref{meps-bound}.  Let
    $\eps_j \to 0$ be a subsequence which attains the infimum, i.e.
    $\lim_{j_ \to \infty} E_{\eps_j}[m_{\eps_j}] = \liminf_{\eps \to 0}
    E_\eps[m_\eps]$.

  \medskip
  
 \medskip

 \emph{Step 1: Localization.} First we claim that for each $\hat x\in \SS_m$,
 there is a set $\mathcal{N}\subset (0,1)$ with $|\mathcal{N}|=0$, such that for
 each $\rho\in (0,\hat \rho)\setminus \mathcal{N}$, where $\hat \rho\in (0,1)$
 is the same constant as in Lemma \ref{lem-separate2}, there is a subsequence of
 $\eps_j$, which we do not relabel, and a sequence
 $\rho_j\rightarrow \rho, \rho_j\leq \rho$, with the following properties:
\begin{gather}  %
               \liminf_{j\to  \infty} \left(\NN{Du_{\eps_j}}(B_\rho(\widehat x))+N_{\eps_j}[\chi_{\tau_j,\rho_j}]\right) \ %
       \geq \ 2f(\sqrt{\lambda} \left|\OL n_{\widehat x, \rho} \cdot e_1\right|)\HH^1(\SS_m \cap B_\rho(\widehat x)), \label{eq-annuli-0} \\
    \lim_{\eps_j\to  0}\frac{1}{|\ln\eps_j|}\int_{\R^2\BS B_\rho(\widehat x)}\int_{B_\rho(\widehat x)}\frac{|m_{\eps_j}(x)-m_{\eps_j}(y)|^2}{|x-y|^3}\ dxdy \ %
    = \ 0, \label{eq-annuli-2}
    \end{gather}
\begin{align}\label{eq-annuli}
      \begin{aligned}
        \NN{Du_{\eps_j}} \big(B_{\rho}(\widehat{x})\BS B_{\rho_j-4\tau_j}(\widehat{x})\big)\ &\to  \ 0 \qquad \text{ as } j\to  \infty,\\
        \NN{Du}(\p B_{\rho}(\widehat x)) \ &= \ 0. 
      \end{aligned}
    \end{align}
    Here
    $\OL n_{\widehat x,\rho} := \dashint_{\SS_m\cap B_{\rho}(\widehat x)} n \
    d\HH^1$. Indeed, by Proposition \ref{prp-LOT} and Lemma \ref{lem-inter}, for
    each $\widehat x\in \SS_m$ there is $\mathcal{N}_1\subset (0,1)$ with
    $|\mathcal{N}_1|=0$ such that for each
    $\rho\in (0,\hat \rho)\setminus \mathcal{N}_1$, one can find a subsequence
    of $\eps_j$ (not relabelled) and a sequence $\rho_j\rightarrow \rho$,
    $\rho_j\leq \rho$ such that \eqref{eq-annuli-0} and \eqref{eq-annuli-2} hold
    true.  Moreover, by the uniform boundedness of $\|Du_\eps\|(Q_\ell)$
    (cf. \eqref{eq-compactness}), there is $\mathcal{N}_2\subset (0,1)$,
    $|\mathcal{N}_2|=0$ such that for each
    $\rho\in (0,1)\setminus (\mathcal{N}_1\cup \mathcal{N}_2)$ up to a further
    subsequence of $\eps_j$ (and $\rho_j$ correspondingly) one has
    \eqref{eq-annuli}. Thus the claim follows with
    $\mathcal{N}=\mathcal{N}_1\cup \mathcal{N}_2$.

\medskip

    Let $N \in \N$ be fixed. Since $f$ is Lipschitz continuous and since
    $\HH^1(\SS_m) \leq K$ (cf. \eqref{def-f}, \eqref{eq-compactness}), for any
    $\widehat x \in \SS_m$ there is
    $\rho_0=\rho_0(\widehat{x}, \sqrt{\lam}, N, K)\in (0,\hat \rho)$ such that
    for any $\rho\in (0,\rho_0)$
    \begin{align} \label{lip-ball} %
      2 \int_{\SS_m \cap B_{\rho}(\widehat x)} \hspace{-1ex}\big( f \big(\sqrt{\lam} |n \cdot
      e_1| \big) - f \big(\sqrt{\lam}|\OL n_{\widehat x, \rho} \cdot e_1| \big)
      \big) \ d \HH^1 \ %
     \ed{ \leq} \  \frac{\HH^1(\SS_m \cap B_{\rho}(\widehat x))}{2
      N \HH^1(\SS_m)},
  \end{align}
   The family of balls
  $\mathcal{F}:=\{B_{\rho}(\widehat{x}): \widehat x\in \SS_m,\ \rho \in
  (0,\rho_0)\BS \mathcal N\}$ is a fine cover of $\SS_m$. By the
  Vitali-Besicovitch covering theorem applied to the Radon measure
  $f(\sqrt{\lam} |n\cdot e_1|)\HH^1\llcorner \SS_m$
  \cite[Thm.2.19]{Ambrosio-Book}, there are finitely many disjoint balls
  $B_k := B_{\rho_k}(\widehat{x}_k)\in \mathcal{F}$, $k=1, \ldots, K_0$ with
 \begin{align} \label{est-lo-3} %
   \sum_{k=1}^{K_0}  2 \int_{\SS_m\cap B_k}f(\sqrt{\lam} |n \cdot e_1|) \ d\HH^1   \ %
     \geq \ E_0[m] - \frac{1}{2N}.
 \end{align}
Combining \eqref{lip-ball} and
 \eqref{est-lo-3} and abbreviating $\OL n_k:=\OL n_{\widehat{x}_k,\rho_{k}}$, we
 obtain
\begin{align} \label{est-lo-4}
  \sum_{k=1}^{K_0}  2f(\sqrt{\lambda} \left|\OL n_{k} \cdot e_1\right|)\HH^1(\SS_m\cap B_k) \ %
  \geq \ E_0[m]-\frac{1}{N}.
\end{align}
Since there are only finitely many balls, one can
take a subsequence, which we still denote by $\eps_j$, such that
\eqref{eq-annuli-0}--\eqref{eq-annuli} hold in each $B_k$, i.e. there is a subsequence $\eps_j$ (independent of $k$) such that for each
$B_k=B_{\rho_k}(\widehat x_k)$, $k=1,\cdots, K_0$, there exist
$\rho_{k,j}\to \rho_k$ as $j\to \infty$ such that
\eqref{eq-annuli-0}--\eqref{eq-annuli} hold true with $\rho$, $\widehat x$ and $\rho_j$
replaced by $\rho_k$, $\widehat x_k$ and $\rho_{k,j}$.

\medskip

\emph{Step 2: Decomposition of the energy and conclusion.} Let
$\chi_{k,j}:=\chi_{\tau_j, \rho_{k,j}}(\cdot -\widehat{x}_k)$ be the sequence of
cut-off functions with respect to $B_k$ (cf. \eqref{def-chitau}). The proof is
then concluded by showing the lower bound
\begin{align} \label{R-geq0} %
  \liminf R_{\eps_j} \ \geq \ 0.
\end{align}
where the remainder term $R_{\eps_j}$ is given by
\begin{align} \label{Rep-lob} %
  R_{\eps_j} \ %
  &:= \ %
    E_{\eps_j}[m_{\eps_j}] -  \sum_{k}\Big[\nu_{\eps_j}(B_k)  + N_{\eps_j}[\chi_{k,j}] \Big].
\end{align}
and where $\nu_{\eps_j}(B_k) + N_{\eps_j}[\chi_{k,j}]$ (cf. \eqref{def-Neps})
are the leading order terms in each ball which have been estimated in
Proposition \ref{prp-LOT}. Indeed, \eqref{eq-annuli-0} together with
\eqref{R-geq0} gives the lower bound
\begin{align} %
  \lim_{j\to  \infty} E_{\eps_j}[m_{\eps_j}] \ %
  = \ \liminf_{\eps\rightarrow 0} E_{\eps}[m_\eps] \ %
  \geq \ \sum_{k = 1}^{K_0}  2f(\sqrt{\lambda} \left|\OL n_k \cdot e_1\right|)\HH^1(\SS_m \cap B_k).
\end{align}
Together with \eqref{est-lo-4} and since $N\in \N$ is arbitrary, this
  yields the desired estimate.

\medskip

It remains to show \eqref{R-geq0}: Since
  $\nu_{\eps_j}(Q_\ell \BS \bigcup_k B_k) \geq 0$ and using \eqref{def-Neps} as
  well as the singular integral characterization of the
  $\dot{H}^{-\frac{1}{2}}$-norm in Lemma \ref{lem-sing_H12}, we have
\begin{align} \label{Rep-lob-lo} %
  R_{\eps_j} \ %
  &\geq \ \frac{\lam}{4|\ln \eps_j|} \bigg( \lim_{N\to \infty, N \in
    \N} \int_{Q_\ell} \int_{\R \times [-N\ell,N
    \ell]} \frac{\sig_{\eps_j}(x)\sig_{\eps_j}(x-h)}{|h|}\ dh dx \\ %
  &\qquad\qquad\qquad - \sum_k \int_{\t B_k} \int_{\R^2} \frac{(\t
    \chi_{\eps_j,k}\sig_{\eps_j})(x)(\t \chi_{\eps_j,k}\sig_{\eps_j})(x-h)}{|h|}\ dhdx
    \bigg).
  \end{align}
  \hks{Here, $\tilde B_k\subset \R^2$ is any single connected component of
  $\Pi^{-1}(B_k)$, where $\Pi:\R^2\rightarrow Q_\ell$ is the canonical
  projection and where $\t \chi_{k,j}\in C_c^\infty(\t B_k)$ is the
  corresponding cut-off function supported in $\t B_k$.}  We first estimate the
  long-range interactions: By Lemma \ref{lem-tworep} we have
\begin{align}  \label{est-ff} %
  \hspace{6ex} & \hspace{-6ex} %
                 \lim_{N\rightarrow \infty, N \in \N} \bigg| \int_{Q_\ell}\int_{\R \times ([-N\ell,N\ell]\setminus [-\ell, \ell])}\frac{\sig_{\eps_j}(x)\sig_{\eps_j}(x-h)}{|h|} \ dh dx \bigg| \\
  &\leq \lim_{N \to \infty, N \in \N} \int_{Q_\ell} \int_{\R\times ([-N\ell,N\ell] \BS [-\ell,\ell])}  \frac{|\tilde{m}_{\eps_j}(x-h)-\tilde{m}_{\eps_j}(x)|^2}{|h|^3} \ dh dx \ %
  \leq \ C(\ell, M),
\end{align}
with $\tilde m_\eps:=m_\eps-M$ and where we used $\NI{\t m_\eps} \leq 2$ and
$\spt \tilde m_\eps \CUS Q_\ell$. Therefore, we have
\begin{align}
  \liminf_{j \to \infty} R_{\eps_j} \ \geq \ \liminf_{j \to \infty} \Big(\sum_{k=1}^{K_0} X_{\eps_j}^{(1,k)} + X_{\eps_j}^{(2)} \Big),
\end{align}
where we define
\begin{align} \label{def-X-1-2} %
  X_{\eps_j}^{(1,k)}
  &:=\ \frac{\lam}{4|\ln\eps_j|} \int_{\t B_k}\int_{\R\times  [-\ell, \ell]}\frac{(\t\chi_{k,j}\sig_{\eps_j})(x)((1-\t \chi_{k,j})\sig_{\eps_j})(x-h)}{|h|} \ dhdx, \\
  X_{\eps_j}^{(2)} \ %
  &:= \ \frac{\lam}{4|\ln\eps_j|} \int_{Q_\ell}\int_{\R\times [-\ell, \ell]}\frac{((1-\eta_j)\sigma_{\eps_j})(x) \sigma_{\eps_j}(x-h) }{|h|} \ dhdx
\end{align}
using the notation $\eta_j:= \ \sum_k\chi_{k,j}$.  In the next lemma we show
  that $X_{\eps_j}^{(1,k)} \to 0$ as $j\rightarrow \infty$ and $\liminf_{j\rightarrow \infty} X_{\eps_j}^{(2)} \geq 0$, which
  concludes the proof of \eqref{R-geq0} and hence of the proposition.
\end{proof}
\begin{lemma}\label{lem-inter-energy} With the assumptions of Proposition
    \ref{prp-localization}, for $X_{\eps_j}^{(1,k)}$ and
    $X_{\eps_j}^{(2)}$ defined in \eqref{def-X-1-2}, we have
  \begin{align}
    \lim_{j \to \infty} X_{\eps_j}^{(1,k)} \ = \ 0 \qquad \text{for $k \in \{ 1, \ldots, K_0 \}$}, \text{\qquad\qquad and \quad} %
    \liminf_{j \to \infty} X_{\eps_j}^{(2)} \ \geq \ 0. 
  \end{align}
\end{lemma}
\begin{proof}
  We use the notation from Proposition \ref{prp-localization} and its proof.
  For notational simplicity, we assume that $\widehat x_k = 0$, $\rho_k=\rho$
  and write $\eps:=\eps_j$, $\chi_j:=\chi_{k,j}$, $\rho_j:=\rho_{k,j}$,
  $\tau_\eps:=\tau_{ \eps_j}$, keeping in mind that $\eps\to 0$,
  $\chi_j\to \chi_{B_\rho}$, $\rho_j\to \rho$ and $\tau_j\to 0$ as
  $j\to \infty$. 

  \medskip

  \emph{Estimate of $X_{\eps_j}^{(1,k)}$:} The proof is based on an
  integration by parts and the estimate of the long-range interaction given in
  Lemma \ref{lem-inter}. Integrating by parts and using $\spt \sig_\eps\Subset Q_\ell$, we then get
  \begin{align}
    |X_\eps^{(1,k)}| \ %
    &\leq  \ \frac{\lam}{|\ln\eps|}\int_{\t B_k}\int_{\R\times  [-\ell, \ell]}\frac{|\tilde m_\eps(x)-\tilde m_\eps(x-h)|^2\t\chi_j(x)(1-\t\chi_j)(x-h)}{|h|^3}\ dh dx\\
    &\qquad+ \frac{\lam}{|\ln\eps|}\int_{\t B_k}\int_{\R\times  [-\ell, \ell]} \frac{|\nabla \t\chi_j(x)||\tilde m_\eps(x)-\tilde m_\eps(x-h)|}{|h|^2}\ dhdx\\
    &\qquad + \frac{\lam}{|\ln\eps|}\int_{\t B_k}\int_{\R\times  [-\ell, \ell]}\frac{|\nabla\t\chi_j(x)||\nabla \t\chi_j(x-h)||\tilde m_\eps(x)-\tilde m_\eps(x-h)|}{|h|}\ dhdx\\
    &=: \ I_1+ I_2+ I_3.
  \end{align}
  Since $\t \chi_j = 0$ outside $\t B_{\rho}$ and $1-\t \chi_j = 0$ in
  $\t B_{\rho-2\tau_\eps}$, then with
  $N_{\rho,\tau_\eps}:=\t B_{\rho}\BS \t B_{\rho_j-2\tau_\eps}$ we have
\begin{align}  \label{I-one} %
  I_1 \ &\leq \ %
          \frac{\lam}{|\ln\eps|}\left(\int_{\t B_{\rho}^c}\int_{\t B_{\rho}}\frac{|\tilde m_\eps(x)-\tilde m_\eps(y)|^2}{|x-y|^3}\ dxdy  \right. \label{darin} \\
        &\qquad + \left.\iint_{N_{\rho,\tau_\eps}\times N_{\rho,\tau_\eps}}\frac{|\tilde u_\eps(x)-\tilde u_\eps(y)|^2}{|x-y|^3} dxdy + \iint_{N_{\rho,\tau_\eps}\times N_{\rho,\tau_\eps}}\frac{|\tilde v_\eps(x)-\tilde v_\eps(y)|^2}{|x-y|^3}dxdy\right),
\end{align}
where $\t{u}_\eps=u_\eps-U$ and $\t{v}_\eps=v_\eps-V$. With our selection of
$\rho$ by \eqref{eq-annuli-2} the first integral on the right hand side of
  \eqref{I-one} vanishes in the limit $\eps \to 0$. \hks{For the second integral we
write
\begin{align}
  \hspace{6ex} & \hspace{-6ex} %
                 \frac{1}{|\ln\eps|} \iint_{N_{\rho,\tau_\eps}\times N_{\rho,\tau_\eps}}\frac{|\tilde u_\eps(x)-\tilde u_\eps(y)|^2}{|x-y|^3} dxdy \\ %
&=\frac{1}{|\ln\eps|}\bigg( \int_{N_{\rho,\tau_\eps}}\int_{ N_{\rho,\tau_\eps}\cap \{y: |x-y|\leq \eps\}}\frac{|\tilde u_\eps(x)-\tilde u_\eps(y)|^2}{|x-y|^3} \ dydx \\
&\qquad\qquad+ \int_{N_{\rho,\tau_\eps}}\int_{N_{\rho,\tau_\eps}\cap \{y: \eps\leq |x-y|\leq |\ln\eps|^{-1}\}} \frac{|\tilde u_\eps(x)-\tilde u_\eps(y)|^2}{|x-y|^3} \ dydx \\
&\qquad\qquad + \int_{N_{\rho,\tau_\eps}}\int_{N_{\rho,\tau_\eps}\cap \{y: |\ln\eps|^{-1}\leq |x-y|\leq 4\tau_\eps\}}\frac{|\tilde u_\eps(x)-\tilde u_\eps(y)|^2}{|x-y|^3} \ dydx\bigg) \\
  &\leq \ \frac{C}{|\ln\eps|}\|\nabla u_{\eps}\|_{L^2(Q_\ell)}^2+C \NPL{\nabla \t{u}_\eps}{1}{B_\rho\BS B_{\rho_j-4\tau_\eps}}+\frac{C}{|\ln\eps|} \int_{N_{\rho,\tau_\eps}}\int_{ |\ln\eps|^{-1}\leq |h|\leq 4\tau_\eps}\frac{1}{|h|^3}\ dh dx\\
  &\leq \ \frac{CK}{|\ln\eps|}+C\NPL{\nabla \t{u}_\eps}{1}{B_\rho\BS B_{\rho_j-4\tau_\eps}} + C\tau_\eps.
\end{align}}
The above expression hence converges to zero by \eqref{eq-annuli}. For the third
integral on the right hand side of \eqref{I-one} we do not have a uniform bound
on $\NPL{\nabla\t v_\eps}{1}{Q_\ell}$. Instead, by \eqref{meps-bound} we have
$\|v_\eps(\cdot)-v_\eps(\cdot-h)\|_{L^2(Q_\ell)}^2\leq
\|v_\eps\|_{L^2(Q_\ell)}^2|h|^2\leq \frac K\eps |h|^2$ and
$\NT{v_\eps}^2 \leq K\eps$. We estimate
\begin{align}
  \hspace{6ex} & \hspace{-6ex} %
                 \frac{1}{|\ln\eps|} \iint_{N_{\rho,\tau_\eps}\times N_{\rho,\tau_\eps}}\frac{|\tilde v_\eps(x)-\tilde v_\eps(y)|^2}{|x-y|^3}dxdy \\ %
               &\leq \ \frac{C}{|\ln\eps|} \bigg( \frac{K}{\epsilon} \int\limits_{0<|h|\leq \eps} \frac 1{|h|} \ dh
                 +   K\eps \int\limits_{\eps<|h|\leq 4\tau_\eps}\frac 1{|h|^3} \ dh + \|\nabla V\|_{L^2}^2 \int\limits_{0<|h|\leq 4\tau_\eps}\frac{1}{|h|} \ dh \bigg)  \\
               &\leq  \ \frac{C}{|\ln\eps|}  \big(K+\|\nabla V\|_{L^2}^2\big) \  \to \ 0 \qquad %
                 \text{as $\eps \to 0$.}
\end{align}
The above estimates show that $I_1 \to 0$ as $\eps \to 0$. Similar
but simpler arguments yield $I_2 + I_3$ $\leq$
$\frac{C}{|\ln\eps|\tau_\eps} (\NNN{\nabla \tilde u_\eps}{L^1(B_{\rho})}+K)$
$\to \ 0$. Together we get $X_\eps^{(1,k)} \to 0$ as $\eps \to 0$.

\medskip

\emph{Estimate of $X_{\eps}^{(2)}$:} Writing
$(1-\eta_j)\sigma_{\eps}= \nabla\cdot ((1-\eta_j)\tilde m_{\eps})+
\nabla\eta_j\cdot \tilde m_{\eps}$ we have
\begin{align}
  \frac{4|\ln\eps|}{\lam} X_{\eps}^{(2)} \ %
  &= \  \int_{Q_\ell}\int_{\R\times  [-\ell, \ell]}\frac{\nabla\cdot ((1-\eta_j)\tilde m_{\eps})(x)\nabla\cdot ((1-\eta_j)\tilde m_{\eps})(x-h)}{|h|} dhdx\\
  &\quad\ + \int_{Q_\ell}\int_{\R\times  [-\ell, \ell]}\frac{\nabla\cdot ((1-\eta_j)\tilde m_{\eps})(x) \nabla\cdot (\eta_j \tilde m_{\eps})(x-h)}{|h|}\ dh dx\\
  &\quad \ + \int_{Q_\ell}\int_{\R\times  [-\ell, \ell]} \frac{(\nabla\eta_j\cdot \tilde m_{\eps})(x) (\nabla\cdot \tilde m_{\eps})(x-h)}{|h|} \ dh dx.
\end{align}
Arguing as for  \eqref{est-ff} we hence get the estimate
\begin{align}
X_{\eps}^{(2)} \ %
&\geq \ \frac{\lam}{4|\ln\eps|}\Big( 2\pi \left\||\nabla|^{-\frac{1}{2}}\nabla\cdot ((1-\eta_j)\tilde m_{\eps})\right\|_{L^2(Q_\ell)}^2 -C(\ell, M) \Big)\\
&\quad\ +\frac{\lam}{4|\ln\eps|}\int_{Q_\ell}\int_{\R\times  [-\ell, \ell]}\frac{\nabla\cdot ((1-\eta_j)\tilde m_{\eps})(x) \nabla\cdot (\eta_j \tilde m_{\eps})(x-h)}{|h|}\ dh dx\\
  &\quad \ +\frac{\lam}{4|\ln\eps|}\int_{Q_\ell}\int_{\R\times  [-\ell, \ell]} \frac{(\nabla\eta_j\cdot \tilde m_{\eps})(x) (\nabla\cdot \tilde m_{\eps})(x-h)}{|h|} \ dh dx \\
  &=: \ R_{\eps}^{(1)} + R_{\eps}^{(2)} + R_{\eps}^{(3)}.
\end{align}
Clearly, we have $\liminf_{\eps \to 0} R_{\eps}^{(1)} \geq 0$. An
integration by parts yields
\begin{align}
  |R_{\eps}^{(3)}| \ %
  \leq \ \frac{\lambda}{|\ln\eps|}\int_{Q_\ell}\int_{\R\times  [-\ell, \ell]} \frac{ |\nabla\eta_j\cdot \tilde m_\eps(x)| |\tilde m_\eps (x-h)-\tilde m_\eps (x)|}{|h|^2} \ dh dx.
\end{align}
Using that $|\nabla\eta_j|\leq \tau_\eps^{-1}$ and that
$\|\nabla \tilde{u}_\eps\|_{L^1(Q_\ell)}+ \eps \|\nabla \tilde
m_\eps\|_{L^2(Q_\ell)}^2 + \eps ^{-1}\|v_\eps \|_{L^2(Q_\ell)}^2 \leq K$ as well
as $\|\nabla V\|_{L^\infty}\leq C$ with a similar but simpler argument as
  for the estimate of $I_1$ we get $|R_{\eps}^{(3)}|$ $\leq$
$C(\ell, M)\lambda( K+1) (|\ln \eps| \tau_\eps)^{-1}$ $\to 0$ as $\eps \to
0$. For $R_{\eps}^{(2)}$, we use the decomposition
\begin{align}
\frac{4|\ln\eps|}{\lambda} R_{\eps}^{(2)}&= \ \int_{Q_\ell}\int_{\R\times  [-\ell, \ell]} \frac{\left((1-\eta_j)\nabla\cdot \tilde m_\eps\right)(x) \left(\eta_j \nabla\cdot \tilde m_\eps \right) (x-h)}{|h|} \ dh dx \\
              &\qquad+ \int_{Q_\ell}\int_{\R\times  [-\ell, \ell]} \frac{\left((1-\eta_j)\nabla\cdot \tilde m_\eps\right)(x) \left(\tilde m_\eps \cdot \nabla\eta_j\right) (x-h)}{|h|} \ dh dx \\
              &\qquad+ \int_{Q_\ell}\int_{\R\times  [-\ell, \ell]} \frac{\left(-\nabla \eta_j \cdot \tilde m_\eps\right)(x) \nabla\cdot (\eta_j \tilde m_\eps ) (x-h)}{|h|} \ dh dx .
\end{align}
The estimate for $R_{\eps_j}^{(2)}$ follows using a similar argument as for
$X_{\eps_j}^{(1,k)}$.  
\end{proof}

\section{Proof of Theorem \ref{thm-gamma} -- \ed{limsup--inequality}} \label{sec-recovery}

In this section, we give the proof of the limsup--inequality in Theorem
  \ref{thm-gamma}. We hence consider
  $m = (\chi_{\Ome_0} - \chi_{\Ome_0^c}) e_1 \in \AA_0$ for some $\Ome_0$ with
  reduced $\SS_m$ and outer normal $n$.

\subsection{Construction of recovery sequence}

In this section, we give the construction of the recovery sequence.  We say
  that $\Ome \SUS Q_\ell$ is a polygonal set, if its boundary is the union of a
  finite number of geodesic (straight line) segments. By the approximation
  results for sets with finite perimeters cf. eg. \cite[Remark
  13.13]{Maggi-Book} and the continuity of the energy with respect to the
  convergence of the total variation measures, it is enough to consider the
  situation when $\Ome_0$ is a polygonal set. By further approximations, we may
  assume that each vertex of $\Ome_0$ is shared by exactly two edges, and that
  each edge has length less or equal to $\frac{\ell}{2}$ (by adding finitely
  many artificial vertices).  The latter ensures that for each such segment
  $\gamma\subset Q_\ell$ and any connected component $\tilde{\gamma}$ in its
  preimage in $\R^2$, the quotient map $\R^2\rightarrow Q_\ell$ restricts to a
  \emph{distance-preserving} bijection $\tilde{\gamma}\rightarrow \gamma$.  In
the supercritical case $\lam > 1$, due to the anisotropic effect of the stray
field (i.e. the stray field energy penalizes those transition layers with
$|n_1|>\lambda^{-\frac 12}$), we can reduce the construction to the case of
polygonal set where the condition $|n_1|\leq \lambda^{-\frac 12}$ is
satisfied:
\begin{lemma}[Modified polygonal set]\label{lem-tOme} %
  Let $\lam>1$. Then for any polygonal set $\Ome_0 \SUS Q_\ell$ there is a
  sequence of polygonal sets $\Ome_0^{(k)} \SUS Q_\ell$ such that
  \begin{enumerate}
\item \label{it-pol-con} $|\Ome_0^{(k)} \Delta \Ome_0| \to 0$ as $k \to \infty$.
  \item\label{it-pol-esam} For $m := (\chi_{\Ome_0}-\chi_{\Ome_0^c},0)$ and
      $m_k := (\chi_{\Ome_0^{(k)}}-\chi_{(\Ome_0^{(k)})^c},0)$, we have
    \begin{align} \label{pol-Eid} %
      E_0[m_k] \
      = \ E_0[m] \qquad \text{for any $k \in \N$.}
    \end{align}
  \item \label{it-pol-normal} For all $k \in \N$, the
    unit normal $n^{(k)}$ of $\Ome^{(k)}_0$ satisfies
    \begin{align}
      \NIL{n^{(k)}\cdot e_1}{\p \Ome_0^{(k)}} \ \leq \ \lam^{-\frac 12}.
    \end{align}
  \item \label{it-pol-further} For all $k \in \N$, each vertex of
      $\Ome^{(k)}_0$ is shared by exactly two edges, and the length of each edge
      is no larger than $\frac{\ell}{2}$.
  \end{enumerate}
\end{lemma}
\begin{proof}
  Let $c_* :=\lam^{-\frac 12} \in (0,1)$.  Let $J$ be an edge of $\Ome_0$ with
  normal $n:=(n_1,n_2)$ such that $|n_1|>c_*$. Due to the symmetry we can assume
  without loss of generality that $n_1 > 0$ and $n_2 \geq 0$. We consider a
  sequence of zigzag lines $\ZZ^{(k)}$, $k\in\N$, such that $\ZZ^{(k)}$ connects
  the two end points of the edge $J$ and consists of $2j$ line segments with
  alternating outer unit normal $n_{ \pm}^{(k)}:=(c_*, \pm
  \sqrt{1-c_*^2})$. Moreover, the zigzags of $\ZZ^{(k)}$ have equal length and
  are chosen such that $\ZZ^{(k)}$ does not intersect with any other edge of the
  polygonal set or zigzag lines for $k$ sufficiently large. We
  now replace each edge with normal $n:=(n_1,n_2)$ such that $|n_1|>c_*$ by a
  sequence of zigzag lines as described above.  This defines a sequence of sets
  $\Ome_0^{(k)}$ such that \ref{it-pol-normal} and \ref{it-pol-con} hold.

  \medskip
  
  \ref{it-pol-esam}: It is enough to consider the line energy of 
  $\ZZ^{(k)}$ compared to the line energy of the edge $J$ it replaces (as
    described above). More precisely, for each edge $J$ of $\Ome_0$ with
  $|n_1|>c_*$ we consider the sequence of zigzag lines $\ZZ^{(k)}$, constructed
  as above. We then need to show that
  \begin{align} \label{ab-eq} %
    \int_{\ZZ^{(k)}} \big(1+\lam|(n_\pm^{(k)})_1|^2\big) \ d\HH^1 \ %
    = \ 2\int_{J}\sqrt{\lam}|n_1| \ d\HH^1.
  \end{align} 
  First we note that the total length $\HH^1(\ZZ^{(k)})$ of the zigzag line is determined
  uniquely and is independent of $k$. Indeed, let $\lam_+$ and $\lam_-$ be such
  that $n=\lam_+ n_+^{(k)} + \lam_- n_-^{(k)}$. Direct computation gives that
  \begin{align}
    \lam_+ \ = \ \frac 12\Big(\frac{n_1}{c_*}+\frac{n_2}{\sqrt{1-c_*^2}}\Big), \quad \lam_-=\frac 12\Big(\frac{n_1}{c_*}-\frac{n_2}{\sqrt{1-c_*^2}}\Big).
  \end{align}
  By our assumption $n_1>c_*$ and $\lam_\pm > 0$. Then the total
  length of the edges with normal $n_\pm^{(k)}$ is
  $\lam_\pm \HH^1(J)$.

  \medskip
  
  To prove \eqref{ab-eq}, we notice that, after plugging in the values of $\lam_\pm$ and
  $c_*$, both sides of \eqref{ab-eq} have the same value
  $\left(1+\lam|c_*|^2\right) (\lam_++\lam_-)\HH^1(J)$.
\end{proof}
For the construction of the recovery sequence, we assume that $\Ome_0$ is a
polygonal set of the form in Lemma \ref{lem-tOme}. The recovery sequence is
constructed by patching together rescaled versions of one--dimensional
transition layers along the edges of the polygonal set $\Ome_0$:
\begin{definition}[Construction of recovery sequence] \label{def-meps} %
  Let $m = (\chi_{\Ome_0}-\chi_{\Ome_0^c},0)$, where $\Ome_0$ is a
  polygonal set with normal $n$ such that 
  $\NIL{n_1}{\p \Ome_0} \leq \lam^{-\frac 12}$, and each vertex of $\Ome_0$ is shared by exactly two edges and the length of each edge is no larger than $\frac{\ell}{2}$.
  \begin{enumerate}
  \item For $\eps$ sufficiently small depending on $\Omega_0$ and with the
      notation $\bet_\eps := \eps^{\frac{5}{6}}$, we define the
      regularized set $\Ome_\eps$ with boundary $\gam_\eps := \p \Ome_\eps$ and
      outer unit normal $n_\eps$ by
      \begin{align}
              \Ome_\eps := \Ome_{2\bet_\eps}^{i} \cup (\Ome_0^c \BS \Ome_{2\bet_\eps}^a)^o
      \end{align}
      $\Ome_{2\bet_\eps}^i := \bigcup \big \{ B_{2\bet_\eps} : B_{2\bet_\eps}
      \SUS \Ome_0 \big \}$ and %
      $\Ome_{2\bet_\eps}^a := \bigcup \big \{ B_{2\bet_\eps} : B_{2\bet_\eps}
      \SUS \Ome_0^c \big \}$.  Here, the union is taken over all balls with
      radius $2\bet_\eps$, included in $\Ome_0$ (resp. $\Ome_0^c$).
    \item By construction, in the $\bet_\eps$--neighborhood
      $\mathcal N_\eps := \mathcal{N}_{\bet_\eps}(\gam_\eps)$ of $\gam_\eps$
      one has the tubular coordinates $x =(\sigma,t)$, where $\sigma\in \gam_\eps$
      is the projection of $x$ onto $\gam_\eps$ and
      $t=(x-\sig)\cdot n_\eps\in (- \bet_\eps, \bet_\eps)$. We write
      $D_{\eps}^+ \ := \ \Ome_\eps \BS \mathcal N_\eps$ and %
      $D_{\eps}^- \ := \ \Ome_\eps^c \BS \mathcal N_\eps$, this induces the
      decomposition $Q_\ell=\mathcal N_\eps \cup D_{\eps}^+\cup D_{\eps}^-$.
  \item       We define $m_\eps:=(u_\eps,v_\eps)$ by  $v_\eps:=\sqrt{1-u_\eps^2}$ and 
    \begin{equation}\label{eq-construction-u}
      u_\eps(x) \ := \ %
      \begin{TC}
        \displaystyle \sin\bigg(\frac{\pi}{2}\frac{\arcsin(\tanh \frac t\eps )}{\arcsin(\tanh
            \frac {\bet_\eps}\eps )}\bigg) %
        \qquad\qquad &\text{for $x = (\sig,t) \in \mathcal N_\eps$,} \vspace{0.5ex}\\
        \pm 1 &\text{for $x \in D_{\eps}^\pm$}.
      \end{TC}
    \end{equation}
  \end{enumerate}
  \end{definition}
  In the following, we will show how this construction yields the limsup
    inequality in our $\Gamma$-convergence result. We remark that the precise
    choice of $\bet_\eps$ above is not essential as long as
    $\eps\ll \beta_\eps \leq C \eps^{\frac{2}{3}}$.  By construction of
    $\Ome_\eps$, all corners of $\Ome_0$ have been replaced by arc segments with
    curvature $\frac 1{2\bet_\eps}$.

\subsection{Estimate for recovery sequence}

We first give estimates for the one-dimensional transition layer, given in
Definition \ref{def-meps}. The one-dimensional transition layer is given by a
standard Ginzburg-Landau type profile. For a similar construction in the context
of micromagnetics, we refer e.g. to \cite[Lemma
4.2]{KnuepferMuratovNolte-2019}. We remark the scales of the transition
  layers below: The parameter $\eps$ captures the lengthscale where most of the
  transition takes place, the parameter $\bet \geq \eps$ captures the total
  width of the transition layer between the two regions $\t u = \pm 1$.
\begin{lemma}[One--dimensional transition layer]\label{lem-nonlupper} %
  For $\bet\in (0,1]$ and $\eps\in (0, \bet)$, let
  \begin{align}
    \t u_\eps(t) \ := \ %
    \begin{TC}
      \displaystyle \sin\left(\frac{\pi}{2}\frac{\arcsin(\tanh \frac t\eps )}{\arcsin(\tanh \frac \bet\eps )}\right) \qquad &\text{for } |t|\leq \bet, \vspace{0.5ex}\\
      \pm 1 &\text{for $\pm t\geq \bet$ }
    \end{TC}
  \end{align}
  and $\t v_\eps(t) := \sqrt{1-\t u_\eps^2(t)}$.  Then for any $R\geq \bet$ and
  for universal $C, c_0 > 0$, we have
  \begin{enumerate}
  \item\label{it-local-u}
    $\displaystyle \frac 12 \int_ \R \eps \big[|\t u'_\eps|^2 + |\t v'_\eps|^2 \big] + \frac 1{\eps} \t v_\eps^2   \  dt \ %
    \leq \ 2 + Ce^{-\frac{c_0\bet}\eps}, $
  \item\label{it-1d-u}
    $\displaystyle \frac 12 \int_{-R}^R \int_{-R}^R
    \frac{|\t u_\eps(t)-\t u_\eps(s)|^2}{|t-s|^2} \ dtds \ %
    \leq \ 4 \ln \left(1 + \frac R\eps\right)+C$,
  \item\label{it-logprof}
    $\displaystyle\int_{ \R} \int_{\R} \t u'_\eps(t)\t u'_\eps(s)\ln \frac
    1{|t-s|} \ dt ds \ %
    \leq \ 4 \ln\Big(\frac 1\eps + \frac 1\bet \Big) +C$, 
  \item\label{it-1d-v}
    $\displaystyle \int_{-R}^R \int_{-R}^R \frac{|\t v_\eps(t)-\t
      v_\eps(s)|^2}{|t-s|^2} \ dtds %
    + \int_{\R} \int_{\R} \t v'_\eps(t)\t v'_\eps(s)\ln \frac 1{|t-s|} \ dt ds
    \ %
    \leq \ C$.
  \end{enumerate}
\end{lemma}
\begin{proof}
  We first note that for universal  constants $C, c_0 > 0$, we have
    \begin{align}  %
      |\text{sgn}(t)-\t u_\eps(t)| 
      + |\t v_{ \eps}(t)| \ %
      + \eps |\t u'_\eps(t)|+ \eps |\t v'_\eps(t)|\ &\leq \ Ce^{-\frac{c_0|t|}\eps}\qquad \forall t \in \R. \label{12-exp-bound}
    \end{align}
    We next turn to the proof of the estimates:

  \medskip
  
  \textit{(i):} This follows directly from \ed{\cite[Lemma 4.2]{KnuepferMuratovNolte-2019}}.

  \medskip

  \textit{(ii):} Using that $\NI{\t u} \leq 1$, we first calculate
  \begin{align}
    \int_{-R}^0\int_{\eps}^R\frac{|\t u_\eps(t)-\t u_\eps(s)|^2}{|t-s|^2} \  dtds\ %
   \leq \ \int_{-R}^0\int_{\eps}^R \frac{4}{|t-s|^2} \ dtds \  %
    = \ 4 \ln\Big(1 + \frac R\eps \Big) + 4  \ln \frac 12.
  \end{align}
  By the change of variables $t\mapsto \frac t\eps$ and
    $s\mapsto \frac s\eps$ and due to \eqref{12-exp-bound}, the corresponding integral over
  the set $(-\eps,\eps)^2$ is estimated by a universal constant. By the
    exponential decay \eqref{12-exp-bound}, the integral for the remaining region
    $(t,s) \in (0,R)\times (-\eps,R) \BS (0,\eps) \times (-\eps,\eps)$ is
    estimated by a universal constant. Since the integral is symmetric in $t,s$, the integrals with $t,s$ exchanged yield
    the same terms again.

  \medskip
  
  \textit{\ref{it-logprof}:}  Since $\t{u}'_\eps(t)=0$ for $|t|\geq \beta$ and integrating by parts in $t$
  and $s$, we obtain
  \begin{align}
    \int_{\R}\int_{\R} \t u'_\eps(t)\t u'_\eps(s)\ln \frac 1{|t-s|} \ dt ds \ %
    = \ \frac 12\int_{- \bet}^{\bet} \int_{- \bet}^{\bet} \frac{|\t u_\eps(t)-\t u_\eps(s)|^2}{|t-s|^2} dt ds  + B.
  \end{align}
  The boundary term $B$  from the integration by parts is given by $B = B_1 + B_2$ where 
  \begin{align}
    B_1\ &= \ %
          \frac 12 \int_{-\bet}^{\bet}\frac{(\t u_\eps(-\bet)-\t u_\eps(s))^2}{\bet+s} \ ds + \frac 12 \int_{-\bet}^{\bet}\frac{(\t u_\eps(\bet)-\t u_\eps(s))^2}{\bet-s} \ ds, \\ %
    B_2 \ &= \ \int_{-\bet}^{\bet} (\t u_\eps(\bet)-\t u_\eps(t)) \t u_\eps'(t) \ln \frac 1{\bet-t} \ dt %
         - \int_{-\bet}^{\bet} (\t u_\eps(-\bet)-\t u_\eps(t)) \t u_\eps'(t) \ln \frac 1{\bet+t} \ dt.
  \end{align}
  Integrating by parts again, we get $B_2 = B_1 + B_3$, where
  $B_3 = (\t u_\eps(\bet) - \t u_\eps(-\bet))^2 \ln \frac 1{2\bet}$.  We note
  that from
  $|\t u_\eps(\bet)-\t u_\eps(s)|\leq
  \frac{C|\bet-s|}{\eps}e^{-\frac{c\bet}{\eps}}$ for $s\in (\bet/2,\bet)$,
  which follows from \eqref{12-exp-bound}, and by symmetry we get
  $|B_1| \leq C(\frac{\bet}{\eps})^2 e^{-\frac{c\bet}{\eps}}+C\leq C$ for some
  universal $c,C < \infty$. Since also $|B_3| \leq 4 \ln \frac 1{2\bet}$, the
  estimate \ref{it-logprof} follows from the above estimates together with \ref{it-1d-u}.

\medskip

\medskip

\textit{\ref{it-1d-v}:} The estimate of the first integral follows by changing
variables $t\mapsto \frac t\eps$, $s\mapsto \frac s\eps$ together with
  \eqref{12-exp-bound}. Since $\t{v}(t)=\t{v}'_\eps(t)=0$ for $|t|\geq \beta$
  and integrating by parts as before we get
\begin{align}
\int_{\R} \int_{\R} \t v'_\eps(t)\t v'_\eps(s)\ln \frac 1{|t-s|} \ dt ds \ %
    = \ \frac 12\int_{-\bet}^{\bet}\int_{-\bet}^{\bet}\frac{|\t v_\eps(t)-\t v_\eps(s)|^2}{|t-s|^2} dt ds  + 2\t B_1,
  \end{align}
where 
\begin{align}
  \t B_1 \ = \ \frac 12 \int_{-\bet}^{\bet}\frac{|\t v_\eps(s)|^2}{\bet+s} ds + \frac 12 \int_{-\bet}^{\bet}\frac{|\t v_\eps(s)|^2}{\bet-s} ds.
\end{align}
Since $\t v_\eps$ is even and
$|\t v_\eps(\beta)-\t
v_\eps(s)|\leq\frac{C|\bet-s|}{\eps}e^{-\frac{c\bet}{\eps}}$ for
$s\in (\frac \bet 2,\bet)$, which follows from \eqref{12-exp-bound}, we
have that $|\t B_1|\leq C$. Together with the above estimate this yields the
estimate for the second integral in \ref{it-1d-v}.
\end{proof}
The next lemma is concerned about the self-interaction energy over
$\mathcal{N}_\eps$:
\begin{lemma}[Recovery sequence -- nonlocal terms]\label{lem-self-inter}
  Let $\Ome_0 \SUS Q_\ell$ be a polygonal set which satisfies
      the assumptions of Definition \ref{def-meps}.  Let
    $m_\eps\to m = (\chi_{\Ome_0} - \chi_{\Ome_0^c}) e_1$ be the sequence from
    Definition \ref{def-meps} and let $\sig_\eps:=\nabla\cdot (m_\eps-M)$.  Then
    there is $\eps_0(\Ome_0) >0 $ and $C=C(\Omega_0,\ell, M)$ such that for
  $0<\eps \leq \eps_0(\Ome_0)$, we have
  \begin{align}
    \frac{\pi}{2|\ln\eps|}\int_{Q_\ell} \big| |\nabla|^{-\frac 12} \sig_\eps|^2 \  dx \ %
    \leq \ 2 \int_{\p \Omega_0} |n \cdot e_1|^2 \ d\HH^1 + \frac {C}{|\ln\eps|^{\frac 12}}.
  \end{align}
\end{lemma}
\begin{proof}
  In view of Lemma \ref{lem-sing_H12}, we need to show
  \begin{align}
    \lim_{N \to \infty} \int_{Q_\ell} \int_{\R\times [-N\ell,N\ell]} \frac{\sig_\eps(x-h)\sigma_\eps(x)}{|h|}\ dh dx\ %
    \leq \ 8 |\ln \eps| \int_{\p \Omega_0} |n \cdot e_1|^2 \ d\HH^1 + C(\Ome_0,\ell,M) |\ln\eps|^{\frac 12}.
  \end{align}
  First by \eqref{est-ff} the far-field interaction satisfies
\begin{align}  %
  \hspace{6ex} & \hspace{-6ex} %
                 \lim_{N \to \infty} \int_{Q_\ell} \int_{\R\times ([-N\ell,N\ell] \BS [-\ell,\ell])}  \frac{\sig_\eps(x-h)\sig_\eps(x)}{|h|} \ dh dx \leq \ C(\ell, M).
\end{align}
 We also note that
\begin{align}
  \int_{Q_\ell}\int_{\R\times [-\ell,\ell]}\frac{\nabla\cdot M(x-h) \nabla\cdot M(x)}{|h|} \ dh dx \ %
  \leq \ C(\ell, \|DM\|_{L^\infty})
\end{align}
by an application of Cauchy-Schwarz. It is hence enough to show that
\begin{align}
  \int_{Q_\ell} \int_{\R\times [-\ell,\ell]}  \frac{\sig_\eps(x-h)\sig_\eps(x)}{|h|} \ dh dx \ %
  \leq \ 8|\ln\eps| \int_{\p \Omega_0} |n \cdot e_1|^2 \ d\HH^1 + C(\Ome_0,\ell,M) |\ln\eps|^{\frac 12}.
\end{align}
Again by an application of Cauchy-Schwarz, it is enough to show that
\begin{align}
  \int_{Q_\ell}\int_{\R\times [-\ell,\ell]} \frac{\p_1 u_\eps(x-h)\p_1u_\eps(x)}{|h|} \ dh dx \ %
  &\leq \ 8 |\ln\eps| \int_{\p \Omega_0} |n\cdot e_1|^2 d\HH^1 + C(\Omega_0), \label{main-upper} \\
  \int_{Q_\ell}\int_{\R\times[-\ell,\ell]} \frac{\p_2 v_\eps(x-h)\p_2 v_\eps(x)}{|h|} \ dh dx \ %
  &\leq  \ C(\Omega_0).\label{non-main-upper}
\end{align}
The proof of \eqref{main-upper} is given in the sequel. \ed{The proof of
  \eqref{non-main-upper} follows with the same arguments using the corresponding
  estimates in Lemma \ref{lem-nonlupper}\ref{it-1d-v} instead of Lemma
  \ref{lem-nonlupper}\ref{it-1d-u}--\ref{it-logprof}.}

\medskip

We first note that by the construction in Definition \ref{def-meps}
$\nabla \cdot m_\eps$ has support in
$\CN_\eps := \CN_{\bet_\eps}(\gam_\eps) \SUS Q_\ell$ with
$\bet_\eps = \eps^{\frac 56}$. Furthermore, the set $\mathcal N_\eps$ can be
expressed as finite union of rectangles $R_\eps^{(k)}$ (covering the edge
regions without the corners) and annulus sectors $C_\eps^{(k)}$,
$1\leq k\leq N$, joining each rectangle, i.e.
$\mathcal N_\eps = \bigcup_{k=1}^N R_\eps^{(k)} \cup C_\eps^{(k)}$. 

\medskip

\textit{Proof of \eqref{main-upper}:} Since $\p_1u_\eps$ is supported in
$\mathcal N_\eps$, we hence need to estimate terms of the form
  \begin{align}
    X(A,B) \ := \ %
    \int_{A} \int_{B \cap \{ (x,y) : x-y \in \R \times [-\ell,\ell] \}} \frac{\p_1 u_\eps(x) \p_1 u_\eps(y)}{|x-y|} \ dx dy,
\end{align}
where $A \in \PP_{\eps}$ and $B \in \t \PP_\eps(A)$ and where
  \begin{align}
    \PP_{\eps} \ := \ \{ R_\eps^{(k)} , C_\eps^{(k)} \SUS \R^2
    : 1 \leq k \leq N \}
  \end{align}
  is a set of connected representatives of the edge or corner regions. \hks{Here, for
  simplicity we use the same notation for $R_\eps^{(k)} \subset Q_\ell$ (and
    ($C_\eps^{(k)}$)) and its connected representative
  $\tilde R_{\eps}^{(k)} \subset \R^2$, which is a connected component of
  the pre-image of $R_\eps^{(k)}$ under the quotient map.} Furthermore,
    \begin{align}
      \t \PP_\eps(A)  \ %
      := \ \{ \text{$B = R_\eps^{(k)}$ or $B = C_\eps^{(k)}$} \ : \  (B - A) \cap \R \times [-\ell,\ell] \neq \emptyset \}
    \end{align}
    is the finite set of connected representatives of the edge and corner
    regions which are close to $A$.  It hence remains to estimate terms of the
    form for the self-interaction energy of edges and corners, and terms of the
    form for the interaction between different edges, and for the interaction
    energy of an edge with a corner.  The estimates are given below:

  \medskip

  \textit{(i) Self-interaction energy on edge regions:} We claim that
  \begin{align} \label{est-rkrk} %
    X(R_\eps^{(k)},R_\eps^{(k)}) \ \leq \ 8(n_k\cdot e_1)^2
    \HH^1(\gam_k)|\ln\eps|+C(\Omega_0) \qquad \text{for any edge region
    $R_\eps^{(k)} \in \PP_\eps$,}
  \end{align}
  where $n_k\in \mathbb{S}^1$ is the outer unit normal for the edge
    $\gamma_k\subset \p\Omega_0$, where $\gamma_k$ is parallel with
    $\gamma_\eps^{(k)}:=\gamma_\eps\cap R^{(k)}_\eps$. Indeed, by a change of
  variables, we can write 
  \begin{align}\label{eq-Rk}
    X(R_\eps^{(k)},R_\eps^{(k)}) \ %
    & = \ |n_k\cdot e_1|^2 \int_0^{\ell_\eps^{(k)}} \int_0^{\ell_\eps^{(k)}} \int_{-\bet_\eps}^{\bet_\eps} \int_{-\bet_\eps}^{\bet_\eps}  \frac{\t u'_\eps(t) \t u'_\eps(s)}{\sqrt{|t-s|^2+|x_2-y_2|^2} } \ dt ds  dx_2  dy_2,
  \end{align}
 where $\ell_\eps^{(k)}:=\HH^1(\gamma^{(k)}_\eps)$ is the length of the rectangle $R^{(k)}_\eps$. We have used that, within each rectangle $R_\eps^{(k)}$, $u_\eps$ is a
  one-dimensional transition layer across the straight line segment
  $\gamma_\eps^{(k)}$. Using the fact that
  $|t-s|\leq 2\bet_\eps\ll \ell_k:=\HH^1(\gam_k)$, a direct computation yields
  \begin{align}
    \int_0^{\ell_k}\int_0^{\ell_k}\frac{1}{\sqrt{|t-s|^2+|x_2-y_2|^2}} \ dx_2 dy_2 \ %
    \leq \ C + 2 \ell_k \ln \frac 1{|t-s|}
  \end{align}
 for some universal constant $C < \infty$. Since $\t u'_\eps\geq 0$ and $\ell_\eps^{(k)} \leq \ell_k$, we hence get
  the bound
  \begin{align} \label{rkrk-1} %
    X(R_\eps^{(k)},R_\eps^{(k)}) \ %
    &\leq \ |n_k\cdot e_1|^2 \Big(C +
      2\ell_k\int_{-\bet_\eps}^{\bet_\eps}\int_{-\bet_\eps}^{\bet_\eps}
      \t u'_\eps(t)\t u'_\eps(s) \ln \frac 1{|t-s|} \ dt ds \Big).
  \end{align}
  An application of Lemma \ref{lem-nonlupper}\ref{it-logprof} (with
  $\bet=R=\bet_\eps$) and since $\eps \ll \bet_\eps=\eps^{5/6} \ll 1$, we get
  the estimate \eqref{est-rkrk}. The sum over all terms of the form
    $X(R_\eps^{(k)},R_\eps^{(k)})$ for $R_\eps^{k} \in \PP_\eps$ hence
    yields the right hand side of \eqref{main-upper} and it remains to show that
    the other terms are of lower order. These estimates are given below:

\medskip 

\textit{(ii) Interaction energy related to corner regions:} Each
corner region $C_\eps^{(k)} \in \t \PP_\eps$ is an annulus sector of the
form $C_\eps^{(k)}-q_k\subset B_{3\beta_\eps}\setminus B_{\beta_\eps}$ for some
$q_k$. We claim that
\begin{align} \label{est-ckck} %
  X(C_\eps^{(k)},C_\eps^{(j)}) \ \leq \ C(\Ome_0) \eps^{\frac
  12} \qquad %
  \text{ for any corner regions $C_\eps^{(k)} \in \PP_\eps$,
  $C_\eps^{(j)} \in \t \PP_\eps (C_\eps)$ } .
  \end{align}
  By the
  change of variables $x \mapsto (x-q_k)/\bet_\eps$ and since by construction we
  have $|\nabla u_\eps|\leq C\eps^{-1}$, the claim \eqref{est-ckck} then follows
  from the estimate
  \begin{align}
    X(C_\eps^{(k)},C_\eps^{(j)}) \ %
    \leq \ \frac{C(\Ome_0) \bet_\eps^3}{\eps^2}\int_{B_{ 3}\BS B_{1}}\int_{B_{ 3}\BS B_{ 1}}\frac{1}{|x-y|} \ dxdy \ %
    \leq \ \frac{C(\Ome_0) \bet_\eps^3}{\eps^2} \ %
    = \ C(\Ome_0) \eps^{\frac 12}.
  \end{align}
  Now we consider any corner $C_\eps^{(k)} \in \PP_\eps$ and a corresponding
  adjacent edge region $R_\eps^{(j)} \in \t \PP_\eps(C_\eps^{(k)})$. In this case, we can use
  Cauchy-Schwarz, i.e. ($X(A,B)\leq \sqrt{X(A,A)}\sqrt{X(B,B)}$) together
    with (i) and (ii), to get
  $X(C_\eps^{(k)},R_\eps^{(j)}) \ \leq \ C(\Ome_0)$.

  \medskip

  \textit{(iii) Interaction between different edge regions:} We
  claim that
  \begin{align} \label{est-rkrj} %
    X(R_\eps^{(k)},R_\eps^{(j)}) \ \leq \ C(\Ome_0) \ < \ \infty
  \end{align}
  for any two different edge regions $R_\eps^{(k)} \in \PP_\eps$ and
  $R_\eps^{(j)} \in \t \PP_\eps(R_\eps^{(k)})$. We first consider the case
    of two adjacent edge regions: \ed{Let $\alp_k = \alp_k(\Ome_0)$ be the angle
  between the two edges regions (cf. Fig \ref{fig-transf}).  Let
  $T:\R^2\rightarrow \R^2$ be any fixed transformation which is one-to-one and
  satisfies $T\big|_{R_\eps^{(j)}}$ is the identity and $T\big|_{R_\eps^{(k)}}$
  is the rotation about $p_k$ (cf. Fig \ref{fig-transf}), such that up to a
  rotation and translation of the coordinates
  $\tilde{R}_\eps^{(k)}:=T(R_\eps^{(k)})=
  (-\ell_\eps^{(k)}-C\beta_\eps,0)\times (-\beta_\eps,\beta_\eps)$ and
  $R_\eps^{(j)}= (\ell_\eps^{(k)}+C\beta_\eps,0)\times (-\beta_\eps,\beta_\eps)$
  for some $C=C(\alpha_k)$. With such transformation we have
  $\frac{1}{|x-y|}\leq C(\alp_k)\frac{1}{|T(x)-T(y)|}$ for
  $x\in R_\eps^{(j)}$ and $y\in R_\eps^{(k)}$.}
  \begin{figure}
    \centering %
    \includegraphics[height=3.8cm]{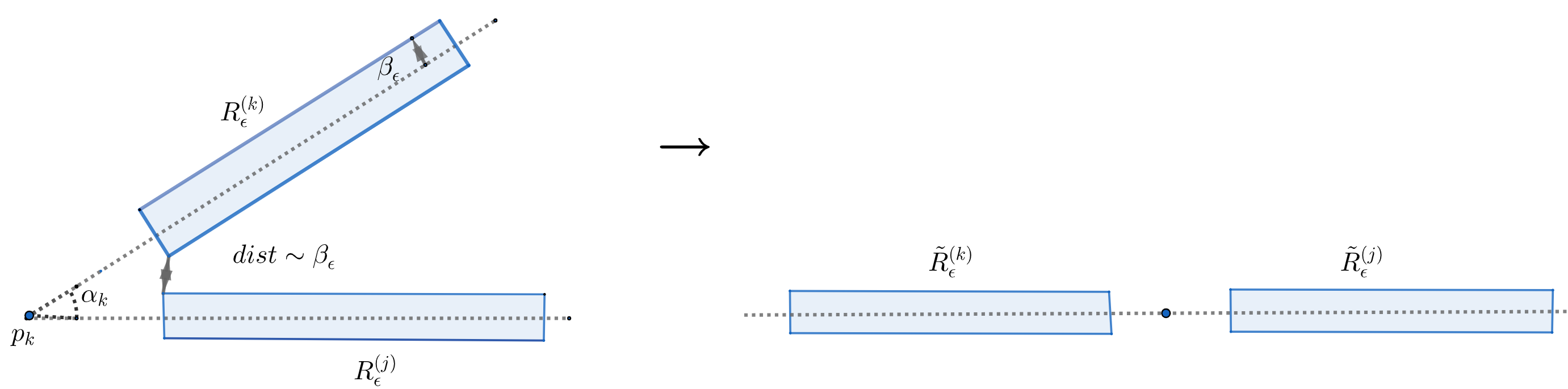}   
    \caption{The transformation $T$.}
    \label{fig-transf}
  \end{figure}
  Moreover, we get the estimate
  \begin{align}
    X(R_\eps^{(k)},R_\eps^{(j)}) \ %
    &\leq \ C(\Ome_0) \int_{-\ell_\eps^{(k)}-C\beta_\eps}^0\int_{-\bet_\eps}^{\bet_\eps}\int_0^{\ell_\eps^{(j)}+C\beta\eps}\int_{-\bet_\eps}^{\bet_\eps} \frac{\t u'_\eps(t)\t u'_\eps(s)}{\sqrt{|t-s|^2+|x_2-y_2|^2}} dtdx_2dsdy_2\\
    &\leq \ C(\Ome_0)  \Big(\int_{-\bet_\eps}^{\bet_\eps}  \t u'_\eps(t) \  dt \Big)^2  \Big( \int_{0}^{\ell_\eps^{(k)}+C\beta_\eps} \int_{0}^{\ell_\eps^{(j)}+C\beta_\eps}  \frac{1}{x_2+y_2} dx_2 dy_2 \Big) \label{bin-rhs}.
  \end{align}
   Both integrals on the right hand side of \eqref{bin-rhs} are estimated by
    a universal constant. For two non-adjacent edge regions one has
    $|x-y|=O(1)$ for $x\in R_\eps^{(k)}$ and $y\in R_\eps^{(j)}$. With an
    analogous transformation as before we hence have 
\begin{align}
X(R_\eps^{(k)}, R_\eps^{(j)})\ \leq \ C(\Omega_0) \int_{-\beta_\eps}^{\beta_\eps} \int_{-\beta_\eps}^{\beta_\eps} \t u'_\eps(t) \t u'_\eps(s) \ dtds \leq C(\Omega_0).
\end{align}
\end{proof}
In the next lemma we estimate the local term in the energy:
\begin{lemma}[Recovery sequence -- local terms] \label{lem-feps-pol} %
  Let $\Ome_0 \SUS Q_\ell$ be a polygonal set which satisfies the assumptions of
  Definition \ref{def-meps}. Let
  $m_\eps\to m = (\chi_{\Ome_0} - \chi_{\Ome_0^c}) e_1$ be the sequence from
  Definition \ref{def-meps} and let $\sig_\eps:=\nabla\cdot (m_\eps-M)$. Then
  \begin{align} %
    \frac 12 \int_{Q_\ell} \Big(\eps |\nabla m_\eps|^2 +\frac{v_\eps^2}{\eps} \Big) \ dx \ %
    \leq \ 2 \int_{\SS_m} 1 \ dx  + C(\Ome_0) \eps^{\frac 56}.
  \end{align}
\end{lemma}
\begin{proof}
  We use the notations of the proof of Definition \ref{def-meps}. Let the
    edge and corner regions $R_\eps^{(k)}$ and $C_\eps^{(k)}$ be
    given as in the proof of Lemma \ref{lem-self-inter}.  Then
  \begin{align}
    \frac 12 \int_{Q_\ell} \Big(\eps |\nabla m_\eps|^2 + \frac{v_\eps^2}{\eps} \Big) \ dx \ 
    = \ \frac 12 \sum_{k=1}^N \int_{R_\eps^{(k)} \cup C_\eps^{(k)}} \Big(\eps |\nabla m_\eps|^2 + \frac{v_\eps^2}{\eps} \Big) \ dx.
  \end{align}
  By Lemma \ref{lem-nonlupper}\ref{it-local-u}, we have
  \begin{align}
     \frac 12 \int_{R_\eps^{(k)}} \Big(\eps |\nabla m_\eps|^2  + \frac{ v_\eps^2 }{\eps} \Big) \ dx \ %
    = \  \frac {\ell_\eps^{(k)}}2 \int_{- \bet_\eps }^{ \bet_\eps } \Big(\eps  [\t u_\eps'^2 + \t v_\eps'^2]  + \frac {\t v_\eps ^2}{\eps} \Big)  \ dt \ %
    \leq \ 2 \ell_\eps^{(k)} + Ce^{-\frac{c_0 \bet_\eps}{\eps}}.
  \end{align}
  In the corner regions $C_\eps^{(k)}$, all level sets
  $d_{\p \Ome_\eps}^{-1}(s)$ have length no larger than $8\pi
  \bet_\eps$. Thus by the coarea formula and Lemma
  \ref{lem-nonlupper}\ref{it-local-u}, and the choice
  $\bet_\eps=\eps^{\frac{5}{6}}$
  \begin{align}
    \int_{C_\eps^{(k)}}\eps |\nabla m_\eps|^2 + \frac{v_\eps^2}{\eps} \  dx \ %
    \leq \ C\bet_\eps \int_{-\beta_\eps}^{\beta_\eps} \eps (\t u_\eps'^2 + \t v_\eps'^2) + \frac{\t v_\eps^2}{\eps}  \ dt \ %
    \leq \ C \eps^{\frac{5}{6}}.
  \end{align}
  The assertion follows by summing up the above estimates.
\end{proof}
With Lemmas \ref{lem-feps-pol}--\ref{lem-self-inter} at hand, we are ready
  to give the proof of the limsup--inequality:
\begin{proposition}[\ed{Limsup--inequality} for
  $\Gam$--limit] \label{prp-recovery} %
  For any $m \in \AA_0$ there exists a recovery sequence $m_\eps \in \AA$ with
  $m_\eps\to m$ in $L^1(Q_\ell)$ such that
  $\limsup_{\eps \to 0} E_\eps[m_\eps] \ \leq \ E_0[m]$.
\end{proposition}
\begin{proof} %
  Given $m\in \AA_0$ there is a sequence of $m_j\in \AA_0$ with polygonal jump
  set such that
  \begin{align}\label{eq-polygon_appro}
    m_j\to m \text{ in } L^1(Q_\ell) \quad \text{ and } \quad %
    \|\nabla m_j\|\to \|Dm\| \qquad \text{ as } j\to  \infty,
  \end{align}
  cf. \cite[Remark 13.13]{Maggi-Book}. Since $f$ is Lipschitz continuous, by Reshetnyak's Theorem \cite{Reshetnyak68}
  $E_0$ is continuous with respect to the convergence in variation of measures,
  i.e. $E_0[m_j]\to E_0[m]$, where $m_j$ satisfies \eqref{eq-polygon_appro}.  By
  Lemma \ref{lem-tOme}, for each $m_j$, there is a sequence of magnetizations
  $m_{j,k}$ with polygonal jump sets whose normals satisfy
  $|n_1|\leq \lambda^{-\frac 12}$, such that $m_{j,k}\to m_j$ in
  $L^1(Q_\ell)$ as $k\to \infty$ and they have the same limit energy,
  i.e. $E_0[m_{j,k}]=E_0[m_j]$ for all $k$.  By a standard diagonal argument, it
  is then enough to construct recovery sequences $m_{j,k}^\eps$ for each
  $m_{j,k}$, which satisfies
  $\limsup_{\eps\to 0} E_{\eps}[m_{j,k}^\eps]\leq E_0[m_{j,k}]$.

\medskip
  
Thus we may assume that $m= (\chi_{\Ome_0} - \chi_{\Ome_0^c}) e_1$, where
$\Ome_0$ is a polygonal set which satisfies the assertions of Lemma
\ref{lem-tOme}.  Let $m_\eps=(u_\eps, v_\eps)$ be the sequence constructed in
Definition \ref{def-meps}.  By application of Lemmas
\ref{lem-self-inter}--\ref{lem-feps-pol}, we have
  \begin{align} \label{ls-eq} %
    E_{\eps}[m_{\eps}] \ &\leq  \ 2\int_{\SS_m}\left(1+\lambda|n_1|^2\right) d\HH^1 + \frac C{|\ln\eps|^{\frac 12}} \ %
= \ E_0[m]+\frac{C}{|\ln\eps|^{\frac 12}}
  \end{align}
  for some $C = C(\Ome_0,\ell, M)$.  Taking the limsup for
  $\eps \to 0$ in \eqref{ls-eq} we conclude the proof.
\end{proof}

\section{Solution for limit problem} \label{sec-limit}

In this section, we derive the solution of the limit model. More precisely, we
derive the ground state energy in the subcritical ($\lambda\leq 1$) and
supercritical ($\lambda>1$) case, as stated in \eqref{e-ground} in the
introduction, and provide a characterization of the corresponding
minimizers:
\begin{proposition}[Solution of limit model] \label{prp-limit}%
  The minimal energy for $m \in \AA_0$ is given by
     \begin{align}
       e(\lam) \ := \ \min_{m \in \AA_0} E_0[m]  \ \ = \ %
       \begin{cases}
         2(1+\lambda)\ell \qquad&\text{ for } \lambda\leq 1,\\
         4\sqrt{\lambda}\ell &\text{ for } \lambda>1.
     \end{cases}
     \end{align}
     Global minimizers are those configurations, where the jump set $\SS_m$
     is a graph of the form $x_1 = \gam(x_2)$ with normal vector $n$
     (pointing outside $\{m=e_1\}$) satisfying
     \begin{align}
       \min \{ 1, \lam^{-\frac 12} \} \ \leq \ -n_1 \ \leq \ 1 \qquad\qquad %
       \text{for $\HH^1$-a.e. $x \in \SS_m$}.
     \end{align}
 \end{proposition}
\begin{proof}
  From the boundary condition \eqref{bc}, for every $m =(u,v)\in \AA_0$ it follows
  that
  \begin{align} \label{cons} %
    \frac 12 \int_{Q_\ell} e_1\cdot \nabla u \ d\HH^1 \ %
    = \ -\int_{\SS_m} n_1\ d\HH^1 \ %
    = \ \ell.
\end{align}
In the subcritical case $\lambda\leq 1$, by H\"older's inequality and
\eqref{cons} we get
\begin{align}
  E_0[m] \ = \ 2 \int_{\SS_m} \big( 1 + \lambda |n_1|^2 \big) \ d\HH^1 \ %
  \ \lupref{cons}\geq \ 2 \left(\HH^1(\SS_m)+ \frac{\lambda \ell^2}{\HH^1(\SS_m)}\right).
\end{align}
Thus the minimum is achieved for $\HH^1(\SS_m)=\ell$, when $\SS_m$ is
a single line segment from $a$ to $a+ \ell e_2$ for some
$a\in [-1,1]\times\{0\}$, and the minimal energy is $2(1+\lambda)\ell$. For $\lambda>1$ we have
\begin{align} 
  E_0[m] \ = \ 2 \int_{\SS_m} \inf_{\alp \geq 1} \Big[ \alp + \frac{\lam |n_1|^2}{\alp} \Big]  \ d\HH^1 \ %
  \geq \ 4 \int_{\SS_m} \sqrt{\lam}| n_1 | \ d\HH^1 \ %
  \lupref{cons}\geq  4\sqrt{\lam} \ell.
\end{align}
Equality is achieved if and only if $-n_1 \geq \lambda^{-\frac 12}$ $\HH^1$-a.e. on $\SS_m$. This yields that $\SS_m$ is a single graph $x_1=\gamma(x_2)$
for some function $\gamma$.
\end{proof}

\appendix

\section{Real space representation of the stray field}

We first recall the following standard representation of the homogeneous
  $H^{\frac 12}$--norm. We give the short proof since the constant in front of
  the identity is essential in our arguments:
  \begin{lemma}[Finite difference representation of $H^{\frac 12}$--norm] %
    For $m \in H^{\frac 12}(Q_\ell)$ we have
    \begin{align} \label{rep-h12} %
      \int_{Q_\ell} \big||\nabla|^{\frac 12} m\big|^2 \ dx \ %
      = \ \frac 1{4\pi} \int_{Q_\ell} \int_{\R^2} \frac{|m(x+h)-m(x)|^2}{|h|^3} \ dh dx.
    \end{align} 
  \end{lemma}
  \begin{proof}
     Using Plancherels identity \eqref{plancherel-id} and Fubini's theorem, we
     obtain
   \begin{align}
     &\int_{Q_\ell} \int_{\R^2} \frac{|m(x+h)-m(x)|^2}{|h|^3} \ dx dh %
     = \ \int_{Q_\ell} |\widehat m(\xi)|^2 \int_{\R^2} \frac{|1-e^{i\xi \cdot h}|^2}{|h|^3}  \  dh d\xi \\
    &= \ \int_{Q_\ell} |\xi||\widehat m(\xi)|^2 \int_{\R^2} \frac{|1-e^{i\xi \cdot h}|^2}{|h|^3|\xi|} \ dh d\xi%
     = \ 4\pi \int_{Q_\ell} |\xi| |\widehat m(\xi)|^2 \ d\xi.
   \end{align}
   The last identity follows with the change of variables
   $h\mapsto \frac{h}{|\xi|}$ and since
   $\int_{\R^2} \frac{|1-e^{ih_1}|^2}{|h|^3} \ dh = 4\pi$
   (cf. \cite[(39)]{DKO-2006}).
 \end{proof}
 The next lemma yields another representation for the
 $H^{\frac 12}$--norm when $j \to - \infty$ and $k \to \infty$:
\begin{lemma}[$H^{\frac 12}$--norm vs. $H^{-\frac 12}$--norm] \label{lem-tworep} %
    For $f \in C^\infty_c(Q_\ell;\R^2)$ we have
  \begin{align} \label{thesam} %
    \hspace{6ex} & \hspace{-6ex} %
                    \frac 12 \iint_{\t Q_\ell \times \t P_\ell} \frac{|f(y)-f(x)|^2}{|x-y|^3}  \ dxdy \\ %
                  &= \ \iint_{\t Q_\ell \times \t P_\ell} \frac{(\nabla \cdot f)(x)(\nabla \cdot f)(y)}{|x-y|}\ dxdy \   
                    + \iint_{\t Q_\ell \times \t P_\ell} \frac{(\nabla \times f)(x)(\nabla \times f)(y)}{|x-y|}\ dxdy \\
                 &\geq \ \frac 12 \iint_{\t Q_\ell \times \t P_\ell} \frac{(\nabla \cdot f)(x)(\nabla \cdot f)(y)}{|x-y|}\ dxdy.
  \end{align}
  where $\t Q_\ell = \R \times [0, \ell)$ and
  $\t P_\ell = \R \times [j \ell, k \ell)$ for any $j < k \in \Z$.
\end{lemma}
\begin{proof}
  Let $I$ (resp. $J$) be first (second) integral on the second line of
  \eqref{thesam}.  We recall the identity $D (\frac z{|z|^3})$ =
  $D^t (\frac z{|z|^3})$ $=$ $\frac 1{|z|^5}(z^\perp \otimes z^\perp$ $-$
  $2 z \otimes z)$ where $z^\perp := (-z_2,z_1)$.  Integrating by parts in $x$
  and $y$, since $\nabla_z (\frac 1{|z|}) = - \frac z{|z|^3}$ and since
  $\nabla (f \cdot g) = (D^t g) f + (D^tf) g$ then yields
 \begin{align}
   I 
   &= - \iint_{\t Q_\ell \times \t P_\ell} f(x) \cdot  \frac{y-x}{|x-y|^3} \big[\nabla_y \cdot (f(y) - f(x))\big] \\
   &=  - \iint_{\t Q_\ell \times \t P_\ell} \frac{2  f(x)\cdot (y-x) (f(y)-f(x))\cdot (y-x)}{|x-y|^5}  %
     +  \frac{f(x) \cdot (y-x)^\perp (f(y)-f(x))\cdot (y-x)^\perp}{|x-y|^5}.
 \end{align}
 First integrating in $y$ then in $x$ similarly yields
 \begin{align}
   I&= \ \iint_{\t Q_\ell \times \t P_\ell} \frac{ 2 f(y)\cdot (y-x) (f(y)-f(x))\cdot (y-x)}{|x-y|^5}  - \frac{f(y) \cdot (y-x)^\perp (f(y)-f(x))\cdot (y-x)^\perp}{|x-y|^5}.
 \end{align}
 Taking the sum of these two expressions, one gets
 \begin{align} \label{thesam-in} %
   I & = \ \iint_{\t Q_\ell \times \t P_\ell} \frac{1}{|x-y|^3} \Big(\Big|(f(y)-f(x))\cdot \frac{y-x}{|y-x|}\Big|^2 %
         -  \frac 12 \Big|(f(y)-f(x))\cdot \big(\frac{y-x}{|y-x|} \big)^\perp \Big|^2 \Big). 
 \end{align}
 Since $\nabla \times f = \nabla \cdot f^\perp$, the same calculation as before,
 replacing $f$ by $f^\perp$, yields
 \begin{align}
   J \ & \ = \ \iint_{\t Q_\ell \times \t P_\ell} \frac{1}{|x-y|^3} \Big(\Big|(f(y)-f(x))\cdot \big(\frac{y-x}{|y-x|} \big)^\perp\Big|^2 %
         -  \frac 12 \Big|(f(y)-f(x))\cdot \frac{y-x}{|y-x|} \Big|^2 \Big) \ dxdy. 
 \end{align}
 The identity \eqref{thesam} follows by taking the sum of the last two
 equations. The inequality in \eqref{thesam} follows from
   \eqref{thesam-in}.
\end{proof} 
We have the following singular integral characterization for the magnetostatic
energy:
\begin{lemma}[Integral representations of magnetostatic
    energy]\label{lem-sing_H12}
  Let $\sigma\in L^2(Q_\ell)$ with
  \begin{align} \label{mz-app} %
     \quad
    \text{with } \spt \sig \CUS Q_\ell \quad \text{ and } \quad
    \int_{Q_\ell} \sig \ dx \ = \ 0.
  \end{align}
  Then there is a unique $q \in H^1(Q_\ell;\R^2)$ with
  $\nabla \cdot q = \sig$ and $\nabla \times q = 0$ such that
  \begin{align}\label{above-RHS}
    \int_{Q_\ell} \big| |\nabla|^{-\frac 12}\sigma\big|^2 \ dx \ %
    &= \  \frac{1}{4\pi}  \int_{Q_\ell} \int_{\R^2} \frac{|q(x+h) - q(x)|^2}{|h|^3} \ dh dx \\ %
    &= \  \frac{1}{2\pi}  \lim_{N \to \infty, N \in \N} \int_{Q_\ell} \int_{\R\times [-N\ell,N\ell]}\frac{\sigma(x+h)\sigma(x)}{|h|} \ dh dx. \label{below-RHS}
  \end{align}
\end{lemma}
\begin{proof}
  By assumption \eqref{mz-app} we have $\widehat \sig(0) = 0$ and
  $\nabla \widehat \sig \in L^\infty(\R \times \frac {2\pi}{\ell} \Z)$.
  This implies that $q \in H^1(Q_\ell;\R^2)$, where $q$ is defined by its
  Fourier transform $\widehat{q} := -i \frac \xi{|\xi|^2} \widehat \sig$.  By
  construction $q$ satisfies $\nabla \cdot q = \sig$ and $\nabla \times q =
  0$. This solution is unique by the uniqueness of the Helmholtz
  decomposition. By \eqref{frac-sob}, since
    $|\widehat \sig| = |\xi| |\widehat q|$ and by \eqref{rep-h12} we then get
  \begin{align}
    \int_{Q_\ell} ||\nabla|^{-\frac 12} \sig|^2 \ d\xi \ %
    \lupref{frac-sob}= \ \int_{\R \times \frac{2\pi}{\ell} \Z} |\xi| |\widehat {q}|^2 \ d\xi \ 
    \lupref{rep-h12}= \ \frac 1{4\pi} \int_{Q_\ell} \int_{\R^2} \frac{|q(x+h)-q(x)|^2}{|h|^3} \ dh dx.
  \end{align}
  Together with Lemma \ref{lem-tworep} this yields \eqref{above-RHS}. 
\end{proof}

\textbf{Acknowledgements:} We are grateful to the referee for carefully reading
the file and his useful comments. We also thank J. Fabiszisky for carefully
proofreading and drawing some of the pictures. H.~Kn\"upfer was partially
supported by the German Research Foundation (DFG) by the project \#392124319 and
under Germany's Excellence Strategy – EXC-2181/1 – 390900948. W.~Shi was
partially supported by the German Research Foundation (DFG) by the project SH
1403/1-1.  

\small \bibliographystyle{plain}
\bibliography{bib,zigzag}

\end{document}

%% file: all-def.tex
\newcommand{\BS}[0]{\backslash}

\newcommand{\FA}[1]{\text{for all $#1$}} 
\newcommand{\FU}[1]{\text{for $#1$}}

\newcommand{\HW}[1]{} 

\renewcommand{\t}{\tilde}
\newcommand{\p}{\partial}
\renewcommand{\d}{\delta}

\newcommand{\SUS}{\subset}

\newenvironment{TC} {\left \{\begin{array}{ll}} {\end{array} \right.}

\newcounter{pcounter}

\renewcommand{\AA}{{\mathcal A}}

\newcommand{\CC}{{\mathcal C}}

\newcommand{\EE}{{\mathcal E}}
\newcommand{\FF}{{\mathcal F}}
\newcommand{\GG}{{\mathcal G}}
\newcommand{\HH}{{\mathcal H}}
\newcommand{\II}{{\mathcal I}}
\newcommand{\JJ}{{\mathcal J}}

\newcommand{\PP}{{\mathcal P}}
\newcommand{\QQ}{{\mathcal Q}}

\renewcommand{\SS}{{\mathcal S}}

\newcommand{\VV}{{\mathcal V}}

\newcommand{\ZZ}{{\mathcal Z}}

\newcommand{\R}{\mathbb{R}}
\newcommand{\N}{\mathbb{N}}
\newcommand{\Z}{\mathbb{Z}}

\newcommand{\alp}{\alpha}
\newcommand{\bet}{\beta}
\newcommand{\gam}{\gamma}
\newcommand{\eps}{\epsilon}
\newcommand{\zet}{\zeta}

\newcommand{\lam}{\lambda}
\renewcommand{\phi}{\varphi}

\newcommand{\sig}{\sigma}

\newcommand{\Gam}{\Gamma}

\newcommand{\Ome}{\Omega}

\newcommand{\nin}{\not\in}

\DeclareMathOperator*{\dv}{\ \nabla \- \cdot \-}

\DeclareMathOperator*{\dist}{dist}

\DeclareMathOperator*{\spt}{spt}

\renewcommand{\iint}{\int\!\!\!\!\int}

\def\Xint#1{\mathchoice
{\XXint\displaystyle\textstyle{#1}}%
{\XXint\textstyle\scriptstyle{#1}}%
{\XXint\scriptstyle\scriptscriptstyle{#1}}%
{\XXint\scriptscriptstyle\scriptscriptstyle{#1}}%
\!\int}
\def\XXint#1#2#3{{\setbox0=\hbox{$#1{#2#3}{\int}$}
\vcenter{\hbox{$#2#3$}}\kern-.5\wd0}}

\def\dashint{\Xint-}

\newcommand{\upref}[2]{\hspace{-0.8ex}\stackrel{\eqref{#1}}{#2}} 

\newcommand{\lupref}[2]{\hspace{0ex} \stackrel{\eqref{#1}}{#2}} 

\usepackage{color}
\definecolor{verylightblue}{rgb}{0.95, 0.95, 0.95}  
\definecolor{lightblue}{rgb}{0.7, 0.7, 1}

\definecolor{eqyellow}{rgb}{0.9375,0.8984,0.5469}
\definecolor{subeqyellow}{rgb}{1,0.9373,0.8353}

\definecolor{mygreen}{rgb}{0.3, 0.6, 0.3} 
\definecolor{verylightgreen}{rgb}{0.2, 0.8, 0.2} 

\definecolor{verydarkgreen}{rgb}{0, 0.2, 0}

\definecolor{darkgreen}{rgb}{0.85, 0.85, 0.85}  
\definecolor{mydarkgreen}{rgb}{0, 0.5, 0} 

\definecolor{camouflagegreen}{rgb}{0.47, 0.53, 0.42}

\definecolor{mybrown}{rgb}{0.85, 0.4, 0.3}
\definecolor{verylightbrown}{rgb}{0.98, 0.72, 0.58}
\definecolor{verydarkbrown}{rgb}{0.44, 0.26, 0.26}

\definecolor{orange}{rgb}{1, 0.5, 0}

\definecolor{BurntOrange}{rgb}{0.9,0.356,0.1}

\definecolor{burlywood}{rgb}{0.87, 0.72, 0.53}

\definecolor{amethyst}{rgb}{0.6,0.4,0.8}

\definecolor{brightcerulean}{rgb}{0.11, 0.67, 0.84}

\definecolor{mydarkred}{rgb}{1,0.086,0.255}

\definecolor{RoseVYDP}{rgb}{0.84,0.086,0.255}
\newcommand{\RoseVYDP}{\color{RoseVYDP}}
\definecolor{dgreen}{rgb}{0, 0.8, 0.5}     

\definecolor{CanaryBRT}{rgb}{1,0.76,0.26}

\definecolor{cyan}{rgb}{0, 1, 1}
\definecolor{verylightgray}{rgb}{0.95, 0.95, 0.95}
\definecolor{verylightgray}{rgb}{0.95, 0.95, 0.95}
\definecolor{verylightred}{rgb}{1, 0.8, 0.78}
\definecolor{verylightyellow}{rgb}{0.99, 0.98, 0.5}
\definecolor{lightgray}{rgb}{0.8, 0.8, 0.8}

\definecolor{cyanprocess}{rgb}{0.0, 0.72, 0.92}

\definecolor{darkcyan}{rgb}{0.0, 0.45, 0.95} 

\definecolor{darkelectricblue}{rgb}{0.33, 0.41, 0.47}

\definecolor{darkmidnightblue}{rgb}{0.0, 0.2, 0.4}

\definecolor{darkpowderblue}{rgb}{0.2, 0.2, 0.6}  

\definecolor{cadet}{rgb}{0.33, 0.41, 0.47}

\definecolor{bole}{rgb}{0.47, 0.27, 0.23}

\definecolor{browntraditional}{rgb}{0.59, 0.29, 0.0}

\definecolor{burgundy}{rgb}{0.5, 0.0, 0.13}

\newcommand{\ignore}[1]{{}}

\newcommand{\NN}[1]{\|#1\|}
\newcommand{\NNN}[2]{\|#1\|_{#2}}

\newcommand{\NT}[1]{\|#1\|_{L^2}}

\newcommand{\NTL}[2]{\|#1\|_{L^2({#2})}}
\newcommand{\NI}[1]{\|#1\|_{L^\infty}}
\newcommand{\NIL}[2]{\|#1\|_{L^\infty({#2})}}
\newcommand{\NP}[2]{\|#1\|_{L^{#2}}}
\newcommand{\NPL}[3]{{\|#1\|_{L^{#2}(#3)}}}

\newcommand{\cciL}[1]{C_c^{\infty}(#1)}

\makeatletter
\newcommand{\VEC}[2][r]{
  \gdef\@VORNE{1}
  \left(\hskip-\arraycolsep%
    \begin{array}{#1}\vekSp@lten{#2}\end{array}%
  \hskip-\arraycolsep\right)}
\def\vekSp@lten#1{\xvekSp@lten#1;vekL@stLine;}
\def\vekL@stLine{vekL@stLine}
\def\xvekSp@lten#1;{\def\temp{#1}%
  \ifx\temp\vekL@stLine
  \else
    \ifnum\@VORNE=1\gdef\@VORNE{0}
    \else\@arraycr\fi%
    #1%
    \expandafter\xvekSp@lten
  \fi}
\makeatother

